\documentclass[reqno,11pt]{amsart}
\allowdisplaybreaks

\usepackage{longtable}
\usepackage{hyperref}
\usepackage[bbgreekl]{mathbbol}
\usepackage{stmaryrd}
\usepackage{enumerate}
\usepackage{paralist}
\usepackage{amssymb}
\usepackage{MnSymbol}
\usepackage{amsmath}
\usepackage{amsthm}
\usepackage[all]{xy}
\usepackage{array}
\usepackage{epsf}
\usepackage{wrapfig}
\usepackage[mathscr]{eucal}
\usepackage{float}

 \usepackage[usenames,dvipsnames]{pstricks}
 \usepackage{epsfig}

\newcommand{\sfp}{\mathsf{p}}

\newcommand{\supp}{\text{spt}\,}

\newcommand{\gm}{\mathbb{G}_{\text{m}}}

\newcommand{\trans}[2]{{#1}^\vee_{#2}}

\newcommand{\codim}{\mathrm{codim}}

\newcommand{\shcg}[3]{CH_\Delta^{#1}(#3, #2)}  
\newcommand{\reg}{\mathsf{Reg}}
\newcommand{\vtriang}{\Delta \hspace{-0.265cm}\mathord{\raisebox{-0.42\depth}{\scalebox{0.46}{\( \Delta\)} } }
\hspace{-0.27cm}\mathord{\raisebox{-0.50\depth}{\scalebox{0.64}{\( \Delta\)}}}} 

\newcommand{\simplex}[1]{\vtriang{\kern.06em}^{#1}} 

\newcommand{\scA}{\mathscr{A}}
\newcommand{\scB}{\mathscr{B}}
\newcommand{\scC}{\mathscr{C}}
\newcommand{\scD}{\mathscr{D}}

\newcommand{\scF}{\mathscr{F}}
\newcommand{\scG}{\mathscr{G}}
\newcommand{\scH}{\mathscr{H}}
\newcommand{\scI}{\mathscr{I}}
\newcommand{\scL}{\mathscr{L}}

\newcommand{\scO}{\mathscr{O}}

\newcommand{\scS}{\mathscr{S}}

\newcommand{\twopii}{(2 \pi \mbi)}

\newcommand{\norm}[1]{|\hspace{-1pt}| #1 |\hspace{-1pt}|}

   \newcommand{\iaf}[1]{{#1}_{!}}
   \newcommand{\tsn}{\textsc{n}}
   \newcommand{\tsp}{\textsc{p}}

\newcommand{\cs}{{ \bbc^{\times} } }

\newcommand{\dcz}[3]{H^{#1}_{\scD}(#3;\bbz(#2))} 

\newcommand{\om}[2]{\omega_{#1}^{#2}}
\newcommand{\tet}[2]{\theta_{#1}^{#2}}

\newcommand{\sdf}[2]{\mathscr{A}^{#1}_{#2}}         
\newcommand{\df}[2]{\mathscr{A}^{#1}(#2)}         
\newcommand{\dfcs}[2]{\mathscr{A}_c^{#1}(#2)}  
\newcommand{\cdf}[2]{\mathscr{C}\mathscr{A}^{#1}(#2)} 
\newcommand{\bbdf}[2]{\mathscr{B}\mathscr{A}^{#1}(#2)} 
\newcommand{\dfpq}[3]{\mathscr{A}^{#1,#2}(#3)} 
\newcommand{\shflp}[3]{\Omega^{#1}_{#2}\la \log{#3} \ra} 
\newcommand{\shfnull}[3]{\Omega^{#1}_{#2}\la \operatorname{null}{#3} \ra} 
\newcommand{\sdflog}[4]{\scA^{#1,#2}_{#3}\la \operatorname{log}{#4} \ra} 
\newcommand{\sdfnull}[4]{\scA^{#1,#2}_{#3}\la \operatorname{null}{#4} \ra} 
\newcommand{\llocdf}[2]{\scL^1_\text{loc}\mathscr{A}^{#1}(#2)} 
\newcommand{\cdfcs}[2]{\mathscr{C}\mathscr{A}_c^{#1}(#2)}  

\newcommand{\vol}[1]{\varOmega_{#1}} 

\newcommand{\triple}[1]{ \left( {#1}_{\simplex{}}, {#1}_{\Theta}, {#1}_{W} \right)}

\newcommand{\cur}[2]{{'\mathscr{D}}^{#1}(#2)}    
\newcommand{\scur}[2]{{'\mathscr{D}}^{#1}_{#2}}
\newcommand{\curpq}[3]{{'\mathscr{D}}^{#1,#2}(#3)}    

\newcommand{\icur}[2]{\mathscr{I}^{#1}(#2) }    
\newcommand{\ncur}[2]{\mathscr{N}^{#1}(#2) }    

\newcommand{\curmeas}[2]{\mathscr{M}^{#1}(#2)} 
\newcommand{\lnor}[2]{\mathscr{N}_{\mathrm{loc}}^{#1}(#2)} 
\newcommand{\nor}[2]{\mathscr{N}^{#1}(#2)} 

\newcommand{\lint}[2]{\mathscr{I}_{\mathrm{loc}}^{#1}(#2)} 
\newcommand{\intcur}[2]{\mathscr{I}^{#1}(#2)} 
\newcommand{\curlog}[3]{{'\mathscr{D}}^{#1}(#2)\la \log{#3} \ra} 
\newcommand{\curpqlog}[4]{{'\mathscr{D}}^{#1,#2}(#3)\la \log{#4} \ra} 
\newcommand{\realan}[2]{\scI_{#1}^\omega(#2)} 

\newcommand{\lloc}{\mathscr{L}^1_\mathrm{loc}}

\newcommand{\dlog}[1]{\frac{d z_{#1}}{z_{#1}}}

\newcommand{\field}[1]{\ensuremath{\mathbb{#1}}}
\newcommand{\bba}{\field{A}}
\newcommand{\bbc}{\field{C}}
\newcommand{\bbf}{\field{F}}

\newcommand{\bbh}{\field{H}}

\newcommand{\bbn}{\field{N}}
\newcommand{\bbp}{\field{P}}
\newcommand{\bbq}{\field{Q}}
\newcommand{\bbr}{\field{R}}

\newcommand{\bbt}{\field{T}}
\newcommand{\bbz}{\field{Z}}

\newcommand{\cala}{\mathcal{A}}

\newcommand{\calh}{\mathcal{H}}

\newcommand{\call}{\mathcal{L}}
\newcommand{\calm}{\mathcal{M}}

\newcommand{\calx}{\mathcal{X}}

\newcommand{\calz}{\mathcal{Z}}

\newcommand{\bone}{\mathbf{1}}

\newcommand{\mba}{\mathbf{a}}
\newcommand{\mbb}{\mathbf{b}}

\newcommand{\mbi}{\mathbf{i}}

\newcommand{\mbm}{\mathbf{m}}

\newcommand{\mbs}{\mathbf{s}}

\newcommand{\mbt}{\mathbf{t}}
\newcommand{\mbu}{\mathbf{u}}

\newcommand{\mbw}{\mathbf{w}}
\newcommand{\mbx}{\mathbf{x}}
\newcommand{\mby}{\mathbf{y}}
\newcommand{\mbz}{\mathbf{z}}

\newcommand{\mbM}{\mathbf{M}}

\newcommand{\ve}{\varepsilon}

\newcommand{\la}{\langle}
\newcommand{\ra}{\rangle}
\newcommand{\llb}{\llbracket}
\newcommand{\rrb}{\rrbracket}

\newcommand{\sbcx}[3]{\mathscr{Z}_\Delta^{#1}(#3, #2)} 
\newcommand{\sbcxeq}[3]{\mathscr{Z}_{\Delta,\text{eq}}^{#1}(#3, #2)}

\newcommand{\equicycl}[3]{\mathscr{Z}_{\text{eq}}^{#1}(#2/#3)} 

\newcommand*\pFqskip{8mu}
\catcode`,\active
\newcommand*\pFq{\begingroup
        \catcode`\,\active
        \def ,{\mskip\pFqskip\relax}%
        \dopFq
}
\catcode`\,12
\def\dopFq#1#2#3#4#5{%
        {}_{#1}F_{#2}\left( \left. \genfrac..{0pt}{}{#3}{#4}\right| #5\right)%
        \endgroup
}

\newtheorem{theorem}{Theorem}[section]
\newtheorem{lemma}[theorem]{Lemma}
\newtheorem{proposition}[theorem]{Proposition}
\newtheorem{corollary}[theorem]{Corollary}

\theoremstyle{definition}

\newtheorem{definition}[theorem]{Definition}
\newtheorem{example}[theorem]{Example}

\newtheorem{properties}[theorem]{Properties}
\newtheorem{facts}[theorem]{Facts}
\newtheorem{claim}[theorem]{Claim}

\newtheorem{remark}[theorem]{Remark}
\newtheorem{conditions}[theorem]{Conditions}

\newtheorem{notation}[theorem]{Notation}

\theoremstyle{remark}

\date{}
\title[Regulator maps via Current Transforms]{Regulator Maps for Higher Chow Groups \\ via Current  Transforms   }
\author[dos Santos]{Pedro F. dos Santos}
\address{Departamento de Matem\'atica, Instituto Superior
T\'ecnico, Universidade de Lisboa, Portugal}
\email{pedro.f.santos@tecnico.ulisboa.pt}
\thanks{The first author was partially supported by FCT/Portugal through
Project PTDC/MAT-GEO/0675/2012}
\author[Hardt]{Robert M. Hardt}
\address{Rice University, USA}
\email{hardt@rice.edu}
\thanks{The second author was partially supported in by NSF DMS1207702.} 
\author[Lima-Filho]{Paulo Lima-Filho}
\address{Department of Mathematics, Texas A{\&}M University, USA}
\email{plfilho@math.tamu.edu}

\thanks{}

\begin{document}

\begin{abstract}
We show how  to use equidimensional algebraic correspondences between complex algebraic varieties to construct pull-backs and transforms of  certain  classes of geometric currents. Using  this construction we produce explicit formulas at the level of complexes for a regulator map from the Higher Chow groups of smooth complex quasi-projective  algebraic varieties to Deligne-Beilinson cohomology with integral coefficients. A distinct aspect of our approach  is the use of Suslin's complex  \(n \mapsto  \calz^p_{\Delta, \text{eq}}(X,n) \)  of \emph{equidimensional} cycles over \( \Delta^n \) to compute Bloch's higher Chow groups. We calculate explicit  examples involving the M\"{a}hler measure of Laurent polynomials.
\end{abstract}

\maketitle
\tableofcontents


\section*{Introduction}
Using general principles  S. Bloch shows in \cite{Blo-ACBC} the existence of  natural \emph{cycle maps} \(  c \colon \shcg{p}{n}{X}  \to \calh^{2p-n}(X, p)\), from the higher Chow groups of a smooth complex algebraic variety \(  X \) into any bigraded cohomology theory \(  \calh^{*}(-, \bullet) \) that: 
\begin{inparaenum}[1)]
\item satisfies \emph{homotopy axiom},
 \item admits functorial \emph{cycle classes} \(  [\Upsilon] \in \calh^{2b}_\Upsilon(X,b) \) for subvarieties \(  \Upsilon \subset X \) of pure codimension \(  b ,\)  and
 \item satisfies a \emph{weak purity} property. 
 \end{inparaenum}
 In particular, this shows the existence of  a \emph{regulator map} with values in Deligne-Beilinson cohomology.

The primary goal of this paper is to provide a structured  and explicit construction - at the level of complexes - of a regulator   map
\begin{equation}
\label{eq:CycleMap}
\reg{} \colon \shcg{p}{n}{X} \longrightarrow \dcz{2p-n}{p}{X},
\end{equation}
from the higher Chow groups of a smooth complex algebraic variety \(  X \),  in their simplicial formulation and  \(  \bbz \) coefficients  into integral Deligne-Beilinson cohomology. 

A distinct aspect of our approach  is the use of the complex  \(n \mapsto  \calz^p_{\Delta, \text{eq}}(X,n) \) consisting  of cycles \emph{equidimensional} over \( \Delta^n \)
to compute Bloch's higher Chow groups. Here we rely on  Suslin's \emph{generic equidimensionality} results  \cite{Sus-HighCh}, which imply that the inclusion into  Bloch's higher Chow complex \(  \calz^p_{\Delta, \text{eq}}(X,*) \hookrightarrow \calz^p_\Delta(X,*) \) is a quasi-isomorphism (under mild conditions).

To compute  the Deligne-Beilinson  cohomology, in the case where \(X\) is projective, we use cone complex 
\[
\bbz(p)^*_\scD(X) := \mathsf{Cone}\left\{     \scI(p)^*(X)  \oplus F^p\cur{*}{X}  \xrightarrow{\epsilon - \imath}   \cur{*}{X}  \right\}[-1],
\]
where  \(  \scI(p)^*(X) \) denotes the group of  integral currents, with \(\bbz(p)\) coefficients, and \(   F^p \cur{*}{X}  \) denotes the Hodge filtration on the de Rham currents in \(  X \) (see Appendix~\ref{sec:back_curr}).

An element  \(  \gamma \in \bbz(p)^k_\scD(X)  \) is represented as a triple 
\begin{equation}
\label{eq:gamma_triple_int}
\gamma = ( T, \theta, \varpi)\ \in\   \scI(p)^k(X)\oplus F^p{'\mathscr{D}}^k (X)   \oplus  \cur{k-1}{X},
\end{equation}
whose differential is then given by 
\(
\textsc{d}\gamma = (d T,\, d \theta, \, \theta - \epsilon(T) - d \varpi ),
\)
where \(  \epsilon  \) is the inclusion \(  \epsilon \colon \scI(p)^*(X) \hookrightarrow \cur{*}{X}\).
The regulator map must therefore associate to a  cycle \( \Upsilon \in  \calz^p_{\Delta, \text{eq}}(X,n) \) a triple 
\[
 \reg{( \Upsilon) } = \triple{\Upsilon} \in  \scI(p)^{2p-n}(X) \times  F^p\cur{2p-n}{X} \times \cur{2p-n-1}{X},
\]
such that
\(
\reg{( \partial \Upsilon) } = ( d \Upsilon_{\simplex{}}, d \Upsilon_{\Theta},  \Upsilon_{\simplex{}} - \Upsilon_{\Theta} - d \Upsilon_W )
\),
 where \(  \partial \Upsilon \) denotes the boundary in the higher Chow groups complex.

In order to define the triple \( \reg{( \Upsilon) } \) we first  introduce  geometrical constructions with currents that have independent interest. When translated to the equidimensional complex these associate  to  a codimension \( p \)  cycle  \( \Upsilon \)  on \( X\times \Delta^n \),
which is equidimensional over the algebraic simplex \( \Delta^n \) a \emph{transform homomorphism} \( \Upsilon^\vee\colon \calm^{k}( \bbp^n ) \to \calm^{k+2(p-n)}( X)  \), between groups of  currents defined by integration. We apply this construction to define  
\[
\reg{( \Upsilon) }= 
\left( (-1)^{\binom{n}{2}}(-2\pi \mbi)^p \, \Upsilon^\vee_{\simplex{n}}, \ \   (- 2\pi \mbi)^{p-n} \,\Upsilon^\vee_{\Theta_n}, \ \ (- 2\pi \mbi)^{p-n}\,\Upsilon^\vee_{W_n} \right), 
 \]
where \( \simplex{n} \)  denotes the  degree \( n \) current defined by integration on the topological simplex \( \simplex{n}\subset \Delta^n =\{ [ z_0 : \dotsc : z_n ] \in \bbp^n \mid z_0 + \dotsc + z_n \neq 0 \} \) (with the standard orientation)
, \( \Theta_n \) is the current represented by the meromorphic form  \( \sum_{i=0}^n (-1)^i \frac{dz_0}{z_0}\dotsb \widehat{\frac{dz_i}{z_i} } \dotsb \frac{dz_n}{z_n} \)  in \(  \bbp^n \) and \( W_n \) is a degree \( n-1 \) current relating \( \simplex{n} \) and \( \Theta_n \) (see Section~\ref{sec:geo_cur}).

The properties of equidimensional cycles make the transform of currents into a seamless operation yielding the desired map of complexes.
In particular, contrasting with other constructions of the cycle maps from the higher Chow groups to Deligne Beilinson cohomology  (\cite{KLMS-AJHCG}, \cite{Kerr-Lew_SimpAJ}),
no moving lemmas are needed in our construction and the resulting homomorphism  \eqref{eq:CycleMap} is defined with \( \bbz \) coefficients.

Below we summarize the content of each section of the paper.

We start with a brief recollection in Section \ref{sec:background} of  the notions of \emph{equidimensional} and \emph{relative algebraic cycles}, stating the results from \cite{Sus-HighCh} that are relevant in our constructions. Then we introduce the complexes we use to define Deligne-Beilinson cohomology, along with a glossary of the currents and forms that we use. For the reader's convenience, we provide in Appendix \ref{sec:back_curr} a brief review of geometric measure theory and detailed references.

%
%
%
%
%
%
%
%
%

The technical core of the paper lies in Section \ref{sec:cor_trans}, where we use algebraic correspondences to construct pull-backs of  currents:
if \( X \) is a smooth connect variety, \( B \) is a smooth variety  and \( \Upsilon \subset X \times B \) is a codimension \( p \)  subvariety, which  is dominant  over \( B \), 
we show in Proposition \ref{prop:sharp} the existence of pull-back maps \(  \Upsilon^\# \colon  \curmeas{k}{B} \to  \curmeas{k+2i}{X\times B} \),
where, as above,  \(   \curmeas{k}{B} \) denotes the currents of degree \(  k \) in \(  B \) that are representable by integration (measure coefficients).

Furthermore, \(  \Upsilon^\# \) sends currents of type \(  (p,q) \) to currents of type \(  (p+i, q+i) \). 
Under appropriate conditions (Proposition \ref{cor:nor_gen}.\ref{it:inter}) -- \emph{e.g.}, \( \Upsilon \) equidimensional over \( B \) -- the pull-back \(  \Upsilon^\# S \) of a locally normal current \(  S \) coincides with the intersection of currents \(  \llb \Upsilon \rrb \cap \{ \llb X \rrb \times S\} \); and \(  d(\Upsilon\# S) = \Upsilon^\# (dS) \) (Corollary \ref {cor:nor_gen}). 
When \(  B \)  is proper one can define a transform \(  \Upsilon^\vee \colon \curmeas{k}{B} \to \curmeas{k+2(i-n)}{X} \) by \( \Upsilon^\vee S = pr_{1\#} ( \Upsilon^\# S).  \) where \(  pr_1 \colon X\times B \to X \) is the projection. 
%
%
%
As mentioned above, this transform with \( B = \bbp^n \) is used in  the definition of the regulator map of a projective variety \( X \). 

If  \(  U \) is quasiprojective,  let \(  U\hookrightarrow X \hookleftarrow D = X-U \) be a projective compactification of \(  U \) with a simple normal crossings divisor \(  D. \)
If \(  \Upsilon \subset U \times \Delta^n \) is equidimensional over \(  \Delta^n = \bbp^n - H_\infty \), we show that the constructions   above induce a transform 
\[
\overline{\Upsilon}^\vee \colon \curmeas{k}{\bbp^n} \to \cur{k-2n}{X}\la \log D\ra ,
\]
and we study its behavior with respect to hyperplanes \(  H \neq H_\infty \) and boundaries. 
In particular we show the following.
\smallskip
%

\smallskip

\noindent{\bf Corollary \ref{rem:prop_for_reg}.}{\it \ 
Using the  notation in Definition \ref{def:log_transf}, the following holds. 
\begin{enumerate}[I.]
      \item Given a smooth hypersurface \(  H\subset \bbp^n, \  H\neq H_\infty \), denote \(  \mathring{\! H} = H  \cap \Delta^n \). Then 
\(  \overline{\Upsilon}_{|H} \)    and  \(  \overline{\left(  \Upsilon_{|\, \mathring{\! H}}  \right)} \) induce the same transform
   \[
\overline{\Upsilon}_{|H}^\vee \ =\  \overline{\left(  \Upsilon_{|\, \mathring{\! H}}  \right)}^{\, \vee}
\ \colon \  \curmeas{k}{H} \to \curlog{k+2(i-n)}{X}{D }.
\]
   \item If \(  S \) is a current in \(  \bbp^n \)
   vanishing suitably at \(  H_\infty \) (see Definition \ref{def:mild}) then the identity
   \( d \left( \overline\Upsilon^\vee_S \right) =  \overline\Upsilon^\vee_{dS}\)   
   holds in \(\curlog{k+2(i-n)+1}{X}{D}.\)
\end{enumerate}
}
\smallskip

In Section \ref{sec:geo_cur} we introduce a fundamental triple of currents \(  (\simplex{n}, \Theta_n, W_n) \) in complex projective space \(  \bbp^n = \bbp^n(\bbc)^\text{an} \), with the analytic topology. The construction starts with a special  nested sequence of closed semi-algebraic subsets \(  R_{n,0} \subset R_{n,1} \subset \cdots R_{n, n} = \bbp^n\), which are suitably oriented to define \emph{semi-algebraic chains} \(  \llb R_{n,j} \rrb  \in \intcur{n+j}{\bbp^n}\). The current \(  \llb R_{n,0} \rrb \)  corresponds to the natural orientation of the topological simplex \(  \simplex{n} := \Delta^n(\bbr_{\geq 0}) \subset \Delta^n(\bbc)^\text{an} \equiv \bbp^n - H_\infty,\) where \(  H_\infty = \{ [\mbz] \mid \ve_n(\mbz) := z_0 + \cdots + z_n = 0 \} \) is the hyperplane at infinity. 

Next, for \( 0\leq  j\leq n \), denote  \(  \theta^n_j := \sum_{r=0}^j (-1)^r \dlog{0}\wedge \cdots \wedge \widehat{\dlog{r}}\wedge \cdots \wedge \dlog{j} \ \in\ \Omega^{j}(\bbp^n)\la \log D_j\ra \), where \( D_j \) is the divisor given by \(  z_0\cdots z_j = 0,  \)  and define 
\[  
\omega^n_j := (-1)^j \log\left( 1 - \frac{\ve_j(\mbz)}{z_j} \right) \wedge \theta^n_{j-1},
\]
where \(  \ve_j(\mbz) := z_0 + \cdots + z_j. \) In Proposition \ref{prop:basic_cur-I} and Corollary \ref{cor:normal} we 
exhibit formulas for the boundaries of both \(  [\theta^n_j] \) and \(  \llb R_{n,j} \rrb \lefthalfcup \omega_j^n \), and show that they define \emph{normal currents} in \(  \bbp^n \)
(\emph{i.e.}, both the currents and their boundaries are representable by integration). 

With these preliminaries in place, we define the fundamental triple 
\[
(\simplex{n}, \Theta_n, W_n) \ \in \ \intcur{n}{\bbp^n} \oplus F^n \cur{n}{\bbp^n} \oplus \cur{n-1}{\bbp^n},
\]
where  \(  \Theta_n \in F^n \cur{n}{\bbp^n}\) denotes the current in \(  \bbp^n  \) represented 
by \(  \theta^n_n,  \)
and
\(  W_n \) is the \emph{normal current}
\(
W_n := \sum_{j=1}^n (-1)^{\frac{j(j+1)}{2}} (-2 \pi \mbi)^{n-j}\ \llb R_{n,j} \rrb \lefthalfcup \omega_j.
\)
For all \(  n\geq 0 \), the fundamental triple satisfies the  following identity, shown in  Corollary \ref{cor:fund_relation}):
\begin{align*}
d W_n  &= \ (-1)^{\frac{n(n+1)}{2}} \Theta_n - (2\pi\mbi)^n \simplex{n} -  \twopii \sum_{r=0}^n (-1)^r \iota_{r\#}\left( W_{n-1} \right) ,
\end{align*}
where \(  \iota_r \colon \bbp^{n-1} \hookrightarrow \bbp^n \) denotes the inclusion of the \(  r \)-th coordinate hyperplane. 

We conclude the section by establishing that the currents in the fundamental triple satisfy the conditions of Corollary \ref{cor:nor_gen} with respect to algebraic cycles in a product \(  X\times \bbp^n \) which are equidimensional over \(  \Delta^n \). This amounts essentially to having a controlled vanishing at infinity. 
\smallskip

In Section~\ref{sec:Regulator_Map} we use the constructions the previous sections to define our  map of complexes 
\begin{align*}
\reg \colon \mathscr{Z}^p_\text{eq}(U;n) & \longrightarrow \Gamma\left( X; \bbz(p)_{\scD, (X,U)}^{2p-n} \right),
\end{align*}
where \( \mathscr{Z}^p_\text{eq}(U;n)   \) is the Bloch-Suslin (chain) complex of equidimensional cycles and
\(   \bbz(p)_{\scD, (X,U)}^{2p-n} \) is a complex of acyclic sheaves computing the Deligne-Beilinson homology of \(  U. \) More precisely, if \(  \Upsilon \subset U\times \Delta^n \) lies in \(\mathscr{Z}^p_\text{eq}(U;n)    \) then
\(  \reg(\Upsilon) =  \triple{\Upsilon}, \) where
\[
\begin{aligned}
 {\Upsilon}_{\simplex{}}  :=   &  (-1)^{\binom{n}{2}}(-2\pi \mbi)^p \,  \left(\trans{\overline\Upsilon}{\simplex{n}}\right)\cap {U} \ &   \in \ \ &   {\scI^{2p-n}_\text{loc}(U; \bbz(p))} &   &    \\
   { {\Upsilon}_{\Theta} }  := &   (- 2\pi \mbi)^{p-n} \, \trans{\overline\Upsilon}{\Theta_n} \ & \in  \ \ &  { F^p \curlog{2p-n}{X}{D}    } &  \\
   {\Upsilon}_{W}  := &  (-2\pi \mbi)^{p-n}\, \left( \trans{\overline\Upsilon}{W_n}\right)\cap U & \in \ \ & \cur{2p-n-1}{U} .   &
\end{aligned}
\]
Here, \(   \scI^{2p-n}_\text{loc} \) denotes the sheaf of locally integral currents, and \(   F^p \curlog{*}{X}{D}  \) denotes the Hodge filtration on the currents in \(  X \) with logarithmic pole along \(  D \). See Appendix \ref{subsubsec:ccm}.

%

In the last section we calculate the regulator map in a  family of examples,  first introduced as Example \ref{exmp:mah_cor} to illustrate equidimensional cycles. 
There, starting with a subfield \(  \bbf \subset \bbc \), we consider  a (Laurent) polynomial  \(  \sfp( \mbt ) \in \bbf[\mbt ] \) of degree \(  d \) in \(n\)-variables
and introduce  a correspondence \(   \Gamma_\sfp \) in \(  U_\sfp \times \Delta^{n+1} \), where  \(  U_\sfp:= \left\{ (\gm)^n - Z_\sfp\right\} \times \gm  \) with
\(  Z_\sfp \subset (\gm)^n \)   the zero set of \(  \sfp( \mbt ). \)
We show  that \(  \Gamma_\sfp \) is an equidimensional correspondence in \(  \sbcxeq{n+1}{n+1}{U_\sfp}  \) which is, in fact,  a cycle in 
 \(  \sbcxeq{n+1}{n+1}{U_\sfp} \), thus representing an element in \(  CH^{n+1}_\Delta(U_\sfp, n+1). \) 
 
We show that - under simple conditions on the polynomial \(  \sfp \) - the class \(  \left[  \reg{(\Gamma_\sfp})\right] \in
  \dcz{n+1}{n+1}{U_\sfp} \) has a non-trivial ``transcendental'' component coming from \(  H^{n+1}(U_\sfp, \bbc)/H^{n+1}(U_\sfp, \bbz) \)
in the exact sequence
{ \begin{multline}
0\to  H^n(U_\sfp, \bbc)/H^n(U_\sfp, \bbz(n+1))  \to  \dcz{n+1}{n+1}{U_\sfp} \\ \to H^{n+1}(U_\sfp, \bbz(n+1))\oplus F^{n+1}H^{n+1}(U_\sfp, \bbc) \to \cdots .
\end{multline}
}
Evaluating the resulting homomorphism 
\(
 \gamma \colon H_n(U_\sfp, \bbz) \to \bbc/\bbz(n+1)
 \)
\eqref{eq:gamma} on a particular homology class 
we obtain
\[
 -  \twopii^{n} \mbm(\sfp) \in \bbc/\bbz(n+1), 
 \]
where \(  \mbm(\sfp) \) is the logarithmic M\"ahler measure of the polynomial \(  \sfp. \)
The polynomial \( \displaystyle  \sfp_\alpha(x, y) = \alpha  + x +\frac{1}{x} + y + \frac{1}{y}
 \), with \(  \alpha > 4 \) satisfies the conditions needed for the calculations. For example, when 
 \(  \alpha=8 \) one obtains
\(   
\gamma ( \llb \bbt_2\times \{ 1 \} \rrb) =  
- \twopii^2 \mbm(\sfp_8) = 96\,  L(E_{24}, 2) \neq 0 \in \bbc/\bbz(3),
\)
where \(  L(E_{24},z)  \) is the \(  L \)-series of the rational elliptic curve \(  E_{24} \)  of conductor \(  24. \)



\section{Preliminaries}
\label{sec:background}

In this section we  recall the necessary properties of equidimensional cycles, and introduce the notation for the forms and currents that are used throughout the paper. We conclude the section recalling Deligne-Beilinson cohomology.

\subsection{Complexes of equidimensional cycles}

Equidimensional  cycles over a base space play a key role in the study and applications of algebraic cycles. Examples include the development of morphic cohomology \cite{Fried-Law_Cocyc}  and the alternative presentation of motivic complexes in \cite{Fried&Voe-BivCy}. Here we summarize  Suslin's generic equidimensionality results \cite{Sus-HighCh} and \cite{Sus&Voe-RelCh}, the key ingredients in our applications. 

\begin{definition}
\label{def:rel_equid}
Let \(  \calx \to B \) be a scheme of finite type over a noetherian base scheme  \(  B \) and assume that \(  \calx  \) is irreducible, with \(  \dim{\calx}=d\) and \( \dim{B}= n .  \)
An algebraic cycle  \(  W = \sum n_i W_i  \)  on \(  \calx \) is said to be \emph{dominant} over \(  B \) if  each \(  W_i \) is dominant over a component of \(  B \).  It  is called \emph{equidimensional} of \emph{relative dimension \(  r \)} if for each \(  s\in B \) and each component \(  W_i \) of \(  W \), the fiber  \(  (W_{i})_s \)  is either empty or each of its components has dimension \(  r . \)
We denote by \(  \equicycl{q}{\calx}{B}  \) the group of algebraic cycles of codimension \( q \) in \(  \calx \) that are dominant and equidimensional over \(  B \), of relative dimension \(  r= d-n-q. \)
\end{definition}

The following result summarizes the key properties of dominant equidimensional cycles that are relevant to this discussion

\begin{theorem}\cite[3.3.15, 3.4.8]{Sus&Voe-RelCh}
\label{thm:equi_prop}
Let \(  W \) be an equidimensional dominant cycle on \(  \calx \) of relative dimension \(  r \) over a base scheme \(  B. \)
\begin{enumerate}[i.]
   \item If \(  B \)  is normal (or geometrically unibranch) then \(  W \) is a relative cycle of relative dimension \(  r. \)
   \item \label{it:equi_prop-iii}    If \(  B \) is regular, then \(  W \) is a universally integral relative cycle. Hence,  for each map \(  f\colon T \to B \), there exists a unique  relative cycle (with integral coefficients) \(  W_{|T} \) on \(  \calx \times_B T \) over \(  T \), such that {for every point  \(  t \in  T \), the pullback \(  t^*(W_{|T}) \) to \(  \calx_t \) agrees with \(  f(t)^*W \).  } 
 \end{enumerate}
\end{theorem}
\begin{remark}
\label{rem:pb}
\begin{enumerate}[i.]
\item The cycle \(  W_{|T} \) is called the \textbf{pullback} of \(  W \) along \(  f. \) 
\item It follows from the construction that whenever \(  W_{|T} \) is non-zero, then it is  a dominant equidimensional cycle over \(  T.  \)
\item \label{it:pb_funct} If \(  V \xrightarrow{g} T \xrightarrow{f} B \) are maps between regular schemes then \(  W_{|V} = ( W_{|T} )_{|V}  \).
\item If \(  T\subset B \) is a closed immersion of regular schemes and \(  W \) is an equidimensional cycle on \(  X \) over \(  B \), then the pull-back cycle \(  W_{|T} \) coincides with the image of \(  W \) under the intersection-theoretic  pull-back homomorphism induced by  \(  T\times_B X\to X \) as explained in \cite{Ser-ALM} and \cite{Ful-IT}.
\end{enumerate}
\end{remark}

\subsubsection{Simplicial groups of equidimensional cycles}

Let  \(  X \) be an equidimensional scheme of finite type over \( \mathbb{k} \),  and  denote  \(  \sbcxeq{p}{n}{X} := \equicycl{p}{\{X\times \Delta^n\}}{\Delta^n} \). It follows from Theorem \ref{thm:equi_prop}.\ref{it:equi_prop-iii} that the assignment \(  n \mapsto \sbcxeq{p}{n}{X} \) defines a simplicial subgroup \(  \sbcxeq{p}{\bullet}{X} \subset \sbcx{p}{\bullet}{X} \), with associated chain complex   \(  \sbcxeq{p}{*}{X} \).

\subsubsection{Generic equidimensionality}

The next result can be seen as a general \emph{moving lemma} that  has   geometric/measure-theoretic consequences in characteristic zero, yielding a natural  construction of the regulator maps. 
\begin{theorem}\cite[{Thm.}\,2.1]{Sus-HighCh}
\label{thm:Suslin-I}
Let \(  X \) be an equidimensional quasi-projective scheme of finite type over \(  \mathbb{k} \). Then the inclusion map
\(
\sbcxeq{p}{*}{X} \hookrightarrow \sbcx{p}{*}{X}
\)
is a quasi-isomorphism whenever \(  p\leq \dim{X}.  \)
\end{theorem}
\begin{remark}
The condition \(  p\leq \dim{X} \) imposes no restriction when addressing Higher Chow Groups in general. Indeed,  the homotopy property and general equidimensionality give natural quasi-isomorphisms:
\begin{equation}
\label{eq:quasi-simp}
\xymatrix{
\sbcx{p}{*}{X} \ar[rr]_-{\text{htpy. inv. }}^-{ \simeq}   & &  \sbcx{p}{*}{X\times \bba^p}   &  & \ar@{_{(}->}[ll]_-{\simeq}^-{\text{Thm.} \ref{thm:Suslin-I}}\ \ 
\sbcxeq{p}{*}{X\times \bba^p} 
}.
\end{equation}
\end{remark}

\begin{example}
\label{exmp:mah_cor}
Let \(  \bbf \subset \bbc \) be a field, and consider 
a polynomial  \(  \sfp( \mbt ) \in \bbf[\mbt ] \) of degree \(  d \) in \(n\)-variables
 \(  \mbt = (t_1, \ldots, t_n) \),  and 
let \(  Z_\sfp \subset (\gm)^n \)  be its zero set. 
Denote \(  U_\sfp:= \left\{ (\gm)^n - Z_\sfp\right\} \times \gm  \), with coordinates \(  (\mbt;\lambda) \), and  let 
 \(  \mbz = (z_0, \ldots ,z_{n+1}) \in \Delta^{n+1} \) be coordinates satisfying 
 \(  \sum_{r=0}^{n+1} z_r =1. \)
Define a correspondence \(   \Gamma_\sfp \) in \(  U_\sfp \times \Delta^{n+1} \) by
\begin{equation}
\label{eq:mah_cor}
\Gamma_\sfp := 
\begin{cases}
z_{n+1}( \lambda + \sfp(\mbt ) ) & = \lambda \\
z_0 -  t_1z_1  & = 0 \\
z_0 + z_1 - t_2 z_2  & = 0 \\
\vdots & \vdots \\
z_0 + z_1 + \cdots + z_{n-1} - t_n z_n & = 0
\end{cases}.
\end{equation}
We claim that \(   \Gamma_\sfp \)   is an \emph{equidimensional correspondence} in \(  \sbcxeq{n+1}{n+1}{U_\sfp}  \) and a \emph{cycle} in 
 \(  \sbcxeq{n+1}{*}{U_\sfp} \), thus representing an element in \(  CH^{n+1}_\Delta(U_\sfp, n+1).  \)
To prove this claim, consider the linear forms 
\begin{equation}
\label{eq:psi}
\ve_j(\mbz) = z_0 + \cdots + z_j , \ \  \text{ for }   j\geq 0, 
\end{equation} 
and define the auxiliary polynomial
\begin{equation}
\label{eq:Rp}  
R_\sfp(\mbz) = (z_1 \cdots z_n)^d\sfp\left( \frac{\ve_0(\mbz)}{z_1}, \frac{\ve_1(\mbz)}{z_2}, \ldots, \frac{\ve_{n-1}(\mbz)}{z_n}  \right) \in \bbq[\mbz] .
\end{equation}
Now, let \(  Y_\sfp \subset \Delta^{n+1} \) be the divisor given by \(  (z_0 z_1 \cdots z_{n+1})(\ve_0(\mbz) \cdots \ve_n(\mbz)) R_\sfp(\mbz) =0, \) 
and observe that \(  \Gamma_\sfp \cap (U_\sfp\times Y_\sfp) = \emptyset. \) 
It follows that, over \(  \Delta^{n+1} - Y_\sfp \),
the correspondence \(  \Gamma_\sfp \) is the graph of the map
\begin{align}
\label{eq:map_mah}
\psi \colon \Delta^{n+1} - Y_\sfp  & \longrightarrow  U_\sfp \\
\mbz & \longmapsto  \left( \mbt(\mbz); \frac{z_{n+1}}{1- z_{n+1}} \sfp(\mbt(\mbz)) \right),\notag
\end{align}
where \(  \mbt(\mbz) = 
\left( \frac{ \ve_0(\mbz)}{z_1}, \frac{ \ve_1(\mbz)}{z_2}, \ldots, \frac{ \ve_{n-1}(\mbz)}{z_n}  \right) . \)
 
We conclude that \(  \Gamma_\sfp \) is an equidimensional correspondence in \(  \sbcxeq{n+1}{n+1}{U_\sfp}  \). By definition,  the faces \(  \partial_i \Delta^{n+1}, i=0,\ldots, n+1,\) are contained in \( Y_\sfp \), and this suffices to show that  \(  \Gamma_\sfp  \) is a cycle in 
 \(  \sbcxeq{n+1}{*}{U_\sfp} \). 
\end{example}

\subsection{Forms, currents and Deligne cohomology}

\subsubsection{Glossary: Forms and currents} 

In Appendix~\ref{sec:back_curr}  the reader will find the  definitions and relevant properties of the  objects listed below, along with detailed references. 
Here, \(  M \) is a smooth oriented manifold of dimension \(  m  \), \(  X  \) is a smooth proper algebraic variety, and \(  D \) is a  divisor in \(  X \) with simple normal crossings.

\begin{longtable}{l c p{6.0cm}| l}
\textsc{Notation} & : & \textsc{Description} & \textsc{References} \\ \hline\hline
\( \df{k}{M} \) & : & Complex-valued differential forms of degree  \( k \)  on \( M\).  \hfill &   \ref{subsec:fsm} \\
\(  \dfcs{k}{M} \) & : &  Compactly supported differential forms of degree \(  k \)  on \(  M. \) \hfill &   \ref{subsec:fsm}  \hfill \\
\( \dfpq{p}{q}{X} \) & : & Smooth forms of type \(  (p,q) \)  on \(  X \). &  \ref{subset:fcm} \\
\(  \scO \) & : & Sheaf of holomorphic functions on complex manifolds. &   \\
\( \Omega^p(X) \) & : & Holomorphic forms of degree \(  p \)  on \(  X \). &  \\
\(  \shfnull{p}{X}{D} \subset \Omega^p_X\)  & : &  Subsheaf consisting of the holomorphic  \(  p \)-forms that vanish on \(  D.  \) & \ref{def:null}.\ref{it:null_hol} \\
\(  \sdfnull{p}{q}{X}{D} \) & : & \(=  \shfnull{p}{X}{D}\otimes \scA^{0,q}_X \) & Defn. \ref{def:null}.\ref{it:null_sm} \\
\(  \scA^{p,q}_c(X)\la \mathsf{null}\, D \ra \) & : & \(  = \Gamma_c(X,\sdfnull{p}{q}{X}{D} )  \) \\
\( \bbdf{k}{M} \) & : & Bounded Baire forms  of degree  \( k \). &  Defn. \ref{def:forms_norms}.\ref{it:bBaire}  \\
\( \cdf{k}{M} \) & : & Continuous  forms of degree  \( k \)  on \( M\). &  Defn. \ref{def:cont}.\ref{it:cont} \\
\( \cdfcs{k}{M} \) & : & Compactly supported continuous  forms of degree  \( k \)  on \( M\). \hfill &  \phantom{XXXXXX}  \\
%
\(  \cur{k}{M} \) & : & DeRham currents of degree \(  k \)  on \(  M. \) &  \ref{subsubsec:cur}\\
\(  \curpq{p}{q}{X}  \) &: & Currents of type \(  (p,q) \) on \(  X. \)& \ref{subsubsec:ccm} \\
\(  \curmeas{k}{M} \) & : &  Currents representable by integration.& \ref{subsubsec:spcur}.\ref{enu:c_meas}  \\
\(  \norm{T} \) & : &  Measure associated to \(  T \in \curmeas{k}{M} \).& \ref{subsubsec:spcur}.\ref{enu:c_meas}  \\
\(  \curlog{k}{X}{D} \) & : & Currents of degree \(  k \) on \(  X \) with log poles along \(  D. \) & Defn. \ref{def:curr_poles} \\
\(  \curpqlog{p}{q}{X}{D} \) & : & \( = \curlog{p+q}{X}{D} \cap \curpq{p}{q}{X}.  \)& \\
\(  F^p\scA^{k}_X \) & : & \(  = \bigoplus_{p \leq r \leq k} \scA_X^{r,k-r}\), Hodge filtration on sheaves of smooth forms. & \\
\(  F^p {'\scD}_X^k \) & : & \(  = \bigoplus_{p \leq r \leq k} {'\scD}_X^{r,k-r}\), Hodge filtration on sheaves of currents on \(  X \). & \\
\(  \icur{k}{M} \), \(  \lint{k}{M} \) & : & Integral and locally integral currents of degree \(  k \)  on \(  M, \) respectively. & \ref{subsubsec:spcur}.\ref{it:int_cur}   \\
\(  \scI_\text{loc}^k(M;G) \) & : & Locally integral chains of degree \(  k \) on  \(  M \) with coefficients in \(  G.  \) & \ref{subsubsec:spcur}.\ref{it:intG}\\
\(  \scI^k_\text{loc}(p)(M) \) & : & \(  =  \scI_\text{loc}^k(M;\bbz(p))  \)  & \\
\(  \llocdf{k}{X} \) & : & \(  L^1_\text{loc} \)-forms of degree \(  k \) on \( M.  \) & Defn. \ref{def:cont}.\ref{it: L1loc}    \\
\(  \ncur{k}{M} \), \(  \lnor{k}{M} \) & : & Normal and locally normal currents of degree \(  k \)  on \(  M. \)&   \ref{subsubsec:spcur}.\ref{it:nor} \\
\(  \la T, f, y \ra \) & : &  \emph{Slice} of a normal current \(  T \) on \(  M \) by a map \( f \colon M\to N\) at the point \(  y \in N. \) & \ref{subsecsec:slice} \\ \hline
\end{longtable}
\smallskip

\subsubsection{Deligne-Beilinson cohomology}
\label{subsubsec:DBcoh}
Given a subring 
\(  A\subset \bbr \) and \(  p\in \bbz, \) denote \(  A(p) := \twopii^p \cdot A \subset \bbc. \) For a topological space \(  X \), let \(  A(p)_X \) denote the locally constant sheaf on \(  X \) with values in \(  A(p).  \)

Let \( \xymatrix{ U = X-D\  \ar@{^{(}->}[r]^-{\jmath} & X & 
\ar@{_{(}->}[l]^-{} D,}  \)  be a good compactification of a smooth complex algebraic variety \(  U \), where \(  X \) is a smooth proper variety and \(  D \) is a DNC.
Denote by \(  \epsilon \colon A(p)_U \to \Omega_U^* \) and 
\(  \iota : F^p\Omega_X^*\la\log D \ra \to \jmath_*\Omega_U^* \) the natural inclusions.

\begin{definition}
\label{def:DBcoh}
The \emph{Deligne-Beilinson} complex   of \(  (X, U, \jmath)  \) is defined as
\begin{equation}
\label{eq:DB_cx}
 A(p)_\scD := A(p)_{\scD, (X,U)} := \mathsf{Cone}\left(  R\jmath_* A(p)_U \oplus F^p\Omega_X^*\la\log D \ra \xrightarrow{\epsilon - \imath}  R\jmath_*\Omega^*_U \right)[-1].
\end{equation}
The hypercohomology of this complex is independent of the good compactification  (up to canonical isomorphisms), and one defines the \emph{Deligne-Beilinson cohomology} of \(  U \) as 
\(  
H^k_{\scD}(U, A(p)) := \bbh^k\left( X, A(p)_{\scD, (X,U)} \right), \ k \geq 0.   
\)
We refer the reader to \cite{EV-DBC} for further details.
\end{definition}

We work primarily with \(  A = \bbz \) in this paper, but  the arguments hold  for arbitrary subrings \(  A \subset \bbr.  \)
Since the sheaves
\({'\mathscr{D}}_U^k \), \(   {'\mathscr{D}}_X^k \la \log D \ra \) and  \(  \scI(p)^k_{U,\text{loc}} \) are acyclic,
we use the quasi-isomorphisms
\(
\Omega^*_U \xrightarrow{\simeq} {'\mathscr{D}}_U^* \) and 
\( \bbz(p)_U \xrightarrow{\simeq}  \scI(p)^*_{U,\text{loc}} ,
\)
and the filtered quasi-isomorphism
\(
\left( \Omega^*_X\la \log D \ra, F^*\right) \xrightarrow{\simeq}
\left( {'\mathscr{D}}_X^* \la \log D \ra, \, F^* \right), 
\)
to obtain our preferred acyclic resolution of the Deligne-Beilinson complex. In particular, we 
use the identification
\[
H^k_{\scD}(U; \bbz(p)) 
= H^k\left(
\mathsf{Cone}\left\{  \scI^*_\text{loc}(p)(U)  \oplus F^p\curlog{*}{X}{D}  \xrightarrow{\epsilon - \imath}   \cur{*}{U}  \right\}[-1]
\right)
\]
to represent  an element \(  \gamma \in \bbz(p)^k_\scD(U)  \) as a triple 
\begin{equation}
\label{eq:gamma_triple}
\gamma = ( T, \theta, \varpi)\ \in\   \scI(p)^k_\text{loc}(U)\oplus F^p{'\mathscr{D}}^k (X)\la \log D \ra   \oplus  \cur{k-1}{U}.
\end{equation}
The differential \(  \textsc{d}\, \colon\, \bbz(p)^k_\scD(U) \to \bbz(p)^{k+1}_\scD(U) \) 
is then given by 
\begin{equation}
\label{eq:diff_DC}
\textsc{d}\gamma = (d T,\, d \theta, \, \theta_{|U} - \epsilon(T) - d \varpi ),
\end{equation}
where \(  \epsilon  \) is the inclusion \(  \epsilon \colon \scI(p)^*_\text{loc}(U) \hookrightarrow \cur{*}{U},\)
and \( \theta_{|U}  \) the restriction of  \( \theta \) to~\(  U.  \)

\begin{remark}
\label{rem:}
For simplicity, we often use \(  T \) instead of  \(  \epsilon(T) \) when considering a locally integral current \(  T \) simply as a De Rham current. 
\end{remark}
%



\section{Correspondences and transforms}
\label{sec:cor_trans}

In this section,  algebraic correspondences are used to construct pull-back homomorphisms and transforms on certain classes of currents. We first present the main results and  applications before exhibiting their proofs, starting with  the key result.
\begin{theorem}
\label{prop:sharp}
Let \(  X \) be a connected, smooth projective variety and let \(  B \) be a smooth variety, with \(  \dim{X} =m \) and \(   \dim{B} = n \). 
An  irreducible subvariety \(\Upsilon \subset X \times B\)  of codimension \(  i \) which is dominant over \(  B \)  induces  a pull-back homomorphism on currents represented by integration
\begin{equation}
\label{eq:pb-cur-intro}
\Upsilon^\#\ \colon\ \curmeas{k}{B} \longrightarrow \curmeas{k+2i}{X\times B}, \quad  k \geq 0.
\end{equation}
Furthermore, if \(  S \) is a current of type \(  (r,s) \) then \(  \Upsilon^\#S \) has type \(  (r+i, s+i).  \)
\end{theorem}
This theorem is proven in Section \ref{subsubsec:proj_case}. 
Next, we explain that the  pull-back homomorphisms \(  \Upsilon^\# \) behave particularly well when the currents satisfy appropriate conditions with respect to \(  \Upsilon \).

\begin{definition}
\label{def:mild}
Fix a Riemannian metric on \(  B. \) Given a smooth subvariety \(  H \subset B \),  let \( H\subset W^\tau  \) be a tubular          \(  \tau \)-neighborhood with smooth boundary \(  \partial W^\tau, \)  \(  \tau >0.\)  We say that a normal current \(  S \)  on \( B\) \emph{vanishes suitably} along \(  H \)  when
\begin{enumerate}[i.]
   \item \(  \norm{S}(H) = \norm{dS}(H) = 0;\) 
   \item The intersection \(  S\cap [\partial W^\tau]  \) exists for all \(  \tau \) sufficiently small, and 
 \(  S\cap [\partial W^\tau]  \) converges weakly to zero as \(  \tau \) goes to zero. 
\end{enumerate}
\end{definition}

If  \(  \Upsilon  \) is a correspondence as in Theorem \ref{prop:sharp}, then  there exists a closed subset \(  F \subset B \) with \(  \codim{F} \geq 2 \) such that \(   \Upsilon_{| B -F} := \Upsilon \cap \{ X\times (B-F) \} \) is equidimensional and dominant over \(  B-F.  \)\ The next result  is proven in Section \ref{subsubsec:proper}. 
\begin{proposition}
\label{cor:nor_gen}
Assume that  \(  \Upsilon \subset X \times B \) satisfies the conditions of   Theorem \ref{prop:sharp}.
\begin{enumerate}[a.]
 \item \label{it:inter} Let
\(  S \) be a \emph{locally normal} current in \(  B \) whose support is contained in \(  B - F \), the domain over which \(  \Upsilon  \) is equidimensional. Then the intersection of  currents \(  \llb \Upsilon \rrb \cap \left( \llb X \rrb \times S \right) \) --  in the sense of \cite[4.3.20]{Fed-GMT} -- exists and satisfies
\begin{equation}
\label{eq:pull_inter}
\Upsilon^\#S \ = \  \llb \Upsilon \rrb \cap \left( \llb X \rrb \times S \right) .
\end{equation}
\item  \label{it:bdry} Assume that \(  B \) is also proper and  that the exceptional set \(  F \) is contained in a smooth subvariety 
\( H \subsetneq  B \).  If  \(  S \) is  a normal current on \( B\) that vanishes suitably along \(  H \)
then
\begin{equation}
\label{eq:bdry_transf}
d\left( \Upsilon^\# S \right)\ = \ \Upsilon^\#(dS). 
\end{equation}
In particular, \(  \Upsilon^\#(S) \)  is a normal current. 
\end{enumerate}
\end{proposition}

When \(  \Upsilon  \) is not dominant over \(  B \) we set  \(  \Upsilon^\# \) to be the zero map, so that  the pull-back operation gives rise to a  pairing
\[
\mathcal{Z}^i(X\times B) \otimes \curmeas{k}{B}\longrightarrow   \curmeas{k+2i}{X\times B}, \quad \sigma\otimes S \longmapsto \sum_r n_r \Upsilon^\#_r S,
\]
where \(  \sigma = \sum_r n_r \Upsilon_r \) is an algebraic cycle of codimension \(  i \) in \(  X\times B. \)  

\begin{definition}
\label{def:transform}
Let \(  B \) be   a proper variety. The  projections \(  X \xleftarrow{\pi_1} X \times B \xrightarrow{\pi_2} B  \)\  induce   \emph{transform} homomorphisms
associated to an algebraic cycle \(  \sigma \in \calz^i(X\times B) \):
\begin{align}
\label{eq:transf0}
\trans{\sigma}{} \  \colon\ \curmeas{k}{B} & \longrightarrow \curmeas{k+2(i-n)}{X}, \quad 
S \longmapsto \trans{\sigma}{S} := \pi_{1\#} \left( \sigma^{\#}S\right).
\end{align}
%
%
\end{definition}

\begin{corollary}
\label{cor:integral}
If \(  S \) is locally integral (respec. sub-analytic, semi-algebraic) and \(  \supp{S} \subset B-F \), then \(  \Upsilon^\#S \)  and\, \(  \Upsilon^\vee_S \)\, are also locally integral (respec.  sub-analytic,  semi-algebraic).
\end{corollary}

We now study the pull-back and transform homomorphisms on quasiprojective varieties over algebraic simplices.  More precisely, we start with a smooth quasiprojective variety   \(  U \)  of dimension \(  m \) and let \(  \Upsilon \subset U \times \Delta^n \) be an irreducible subvariety of codimension \(  i \) in \(  U\times \Delta^n \), whose projection onto \(  \Delta^n \) is dominant and has equidimensional fibers.
Consider  a projective compactification 
\(  U\hookrightarrow X \hookleftarrow D \), where \(  D = X-U \) is a divisor with simple normal crossings, and let \(  \overline\Upsilon \subset X \times \bbp^n \) be the closure of \(  \Upsilon \) in \(  X\times \bbp^n \).
The next result is proven in Section \ref{subsubsec:quasi_pro}.

\begin{proposition}
\label{prop:log_cur}
Let \(  \Upsilon \subset U\times \Delta^n  \) be as above. Then \(  \Upsilon \) induces a well-defined pull-back  homomorphism
\begin{equation}
\label{eq:clos_pb}
\overline\Upsilon^\# \colon \curmeas{k}{\bbp^n} \longrightarrow \curlog{k+2i}{X\times \bbp^n}{\{D\times \bbp^n\} }, \quad S \mapsto\, {\overline\Upsilon}^\#\!S,
\end{equation}
satisfying the following properties.
\begin{enumerate}[a.]
   \item \label{proper:1st} Let \(  H\subset \bbp^n \) be a smooth hypersurface \(  H\neq H_\infty \). Then 
   the correspondences \(   \overline{\Upsilon}_{|H} \) and \( \overline{\left(  \Upsilon_{|\, \mathring{\! H}}  \right)} \) in \(  X\times H \) induce the same pull-back homomorphisms
   \[
\left( \overline{\Upsilon}_{|H}\right)^\# \   =   \ \overline{\left(  \Upsilon_{|\, \mathring{\! H}}  \right)}^{\, \# }\ 
\colon\ \curmeas{k}{H} \to \curlog{k+2i}{X\times H}{\{D \times H\}} .
\]
   \item \label{proper:2nd} If \(  S \) is a current in \(  \bbp^n \) vanishing suitably at \(  H_\infty \) (see Definition \ref{def:mild}) then the identity
   \( d \left( \overline\Upsilon^\# S \right)  =  \overline\Upsilon^\#  (dS)\)   
   holds in \(\curlog{k+2i+1}{X\times \bbp^n}{\{D \times \bbp^n\} }.\)
\end{enumerate}
\end{proposition}

Using the proposition above we can define transforms in the quasiprojective case. 

\begin{definition}
\label{def:log_transf}
Let \(  \Upsilon\subset U\times \Delta^n \) be equidimensional over \(  \Delta^n \),  and let \(  U\hookrightarrow X \) be a compactification of \(  U \) with NCD \(  D = X - U \).  If  \(  \pi_1 \colon X\times \bbp^n \to X \) is the projection, define the transform 
\begin{equation}
\label{eq:transf1}
\overline{\Upsilon}^\vee \colon \curmeas{k}{\bbp^n} \to \cur{k+2(i-n)}{X}\la \log D\ra 
\end{equation}
as the homomorphism \(  S \mapsto  \overline{\Upsilon}_S^\vee \ := \pi_{1\#} \overline{\Upsilon}^\#(S).  \)
\end{definition}
\begin{corollary}
\label{rem:prop_for_reg}
Using the  notation in Definition \ref{def:log_transf}, the following holds. 
\begin{enumerate}[I.]
      \item Given a smooth hypersurface \(  H\subset \bbp^n, \  H\neq H_\infty \), denote \(  \mathring{\! H} = H  \cap \Delta^n \). Then 
\(  \overline{\Upsilon}_{|H} \)    and  \(  \overline{\left(  \Upsilon_{|\, \mathring{\! H}}  \right)} \) induce the same transform
   \[
\overline{\Upsilon}_{|H}^\vee \ =\  \overline{\left(  \Upsilon_{|\, \mathring{\! H}}  \right)}^{\, \vee}
\ \colon \  \curmeas{k}{H} \to \curlog{k+2(i-n)}{X}{D }.
\]
   \item If \(  S \) is a current in \(  \bbp^n \)
   vanishing suitably at \(  H_\infty \) (see Definition \ref{def:mild}) then the identity
   \( d \left( \overline\Upsilon^\vee_S \right) =  \overline\Upsilon^\vee_{dS}\)   
   holds in \(\curlog{k+2(i-n)+1}{X}{D}.\)
\end{enumerate}
\end{corollary} 

The rest of this section is devoted to constructing  the pull-backs and transforms discussed above, and to proving these and other results  that may have an  independent interest. 
The reader mostly interested in the regulator maps can skip the rest of this section and proceed to Section~\ref{sec:geo_cur} with no loss of continuity in the narrative.

\subsection{Integration along the fiber}
\label{subsec:classical}
We first recall the main features of the classical construction on an oriented fiber bundle \(   (E, \pi, B, F) \) over an oriented \(  n \)-manifold \(  B \) with compact \(  r \)-dimensional fiber \(  F \).  The \emph{integration along the fiber} 
homomorphisms
\(
\iaf{\pi} \colon  \df{r+p}{E} \rightarrow \df{p}{B}, p\geq 0
\)
are completely characterized by the following. 
\begin{properties}
Given \(  \alpha \in \df{*}{B} \) and \(  \phi \in \df{*}{E} \) one has:
\begin{description}
   \item[\noindent {\it Projection formula}] \(  \iaf{\pi} \left( \pi^*\alpha \wedge \phi\right)  \ = \ \alpha \wedge \iaf{\pi}\phi  \)
   \item[{\it Fubini theorem}] \(  \int_E \pi^*\alpha \wedge \phi \ = \ \int_B \alpha \wedge \iaf{\pi}\phi  \)
\end{description}
\end{properties}
These homomorphisms can be used to define a \emph{pull-back} map on currents, as the adjoint operation \(  \pi^\# \ \colon\ \cur{k}{B} \to \cur{k}{E} \) that sends \(  S \in \cur{k}{B} \) to the current \(  \pi^\# S \) defined by \(  \pi^\# S\,  \colon \phi \mapsto S(\pi_!\phi) \). Using this pull-back  one can extend classical synthetic constructions in algebraic geometry to currents. As an example we have the \emph{algebraic join} of algebraic cycles, a key ingredient in the homotopy-theoretic applications  studied in \cite{LM-Alg_cyc_BP}, \cite{BLLFMM-alg_cyc_loop_sp}, \cite{LLFM-Alg_cyc_equiv_coh_th} and related work. 

\begin{example}
In the study of cycle maps for Lawson homology \cite{LF-GenCycMap},  the {algebraic join} operation on algebraic cycles was extended to  a complex join  of currents. The main idea is to write  \(  \bbc^{n+1} \) as a direct sum \(  \bbc^{n+1} = V \oplus W \) and let \(  \mathbf{B} \) denote the blow-up of \(  \bbp^n  \) at
\(  \bbp(V\oplus 0 ) \amalg \bbp(0\oplus W)\). Since  \(  \mathbf{B} \) is a \(  \bbp^1 \)-bundle over \(  \bbp(V)\times \bbp(W) \), it comes with a blow-down map 
\(  b \colon \mathbf{B} \to\bbp^n \) and a bundle projection 
\(  \pi \colon \mathbf{B} \to \bbp(V) \times \bbp(W). \)
The pull-back map \(  \pi^\# \colon\,   {'\scD}^k(\bbp(V)\times \bbp(W)) \to   {'\scD}^k(\mathbf{B})  \)   preserves algebraic cycles, semi-algebraic chains, normal and integral currents.  We can now define a pairing
\begin{equation}
\label{eq:cx_join}
\#_\bbc \ \colon \ \cur{r}{\bbp(V)} \times \cur{s}{\bbp(W)} \longrightarrow  \cur{r+s}{\bbp(V\oplus W)} 
\end{equation}
that sends \(  (R,S) \) to \(  R\#_\bbc S := b_\# \pi^\# ( R \times S). \) 
\end{example}

\subsection{Generalized integration along the fibers and current pull-backs}
\label{subsec:gen_IAF}

Our next step is to extend integration along the fibers to a broader context, at the expense of restricting the domain of the corresponding pull-back homomorphism of currents.

Start with  a smooth projective variety \(  X \),  and let \(  B \) be an arbitrary smooth variety, with \(  \dim{X} = m \) and \(  \dim{B}=n. \) Consider an irreducible subvariety \(  \Upsilon \subset X \times B \) of codimension \(  i \) which is  dominant over \(  B. \) 
Assume that \( X  \) and \(  B \) are connected, 
and consider the Zariski closed subset 
\begin{equation}
\label{def:bad_set}  
F= F_\Upsilon := \{ b \in B \mid \dim{\pi^{-1}_2(b)} > m-i \}\subset B.
\end{equation} 
It is clear that the codimension of \(  F \) is greater or equal than \(  2. \) 

\begin{proposition}
\label{prop:int_along_fib}
Given \(  \Upsilon \subset X \times B \) and \(  F \) as above
one can define for all \(  k\geq 0 \) an \emph{integration along the fiber} homomorphism 
\[
\Upsilon_! \ \colon \ \cdf{k}{X\times B}  \longrightarrow \bbdf{k-2(m-i)}{B},
\]
sending continuous forms on \(   X\times B \) to bounded Baire forms on \(  B, \)
so that for all \(  \phi \in \cdf{*}{X\times B} \)  the following holds.
\begin{enumerate}[i.]
   \item \(  \Upsilon_!\phi \) is continuous on \(  B-F \).
   \item If \(  \phi \) lies in \(  \df{k}{X\times B} \), then there is a dense Zariski open subset \(  V \subset B-F \) on which \(  \Upsilon_!\phi \) is a smooth form.
   \item \label{it:Fubini} 
   {\rm (Fubini)}\ The form \(  \Upsilon_!\phi \) represents the current \(  \pi_{2\#}\left( \llb \Upsilon \rrb \lefthalfcup \phi \right) \). In other words, for all \(  \beta  \in \dfcs{*}{B}\) one has
\[
\int_B \Upsilon_!\phi \wedge \beta \ = \ 
\llb \Upsilon \rrb \left( \phi\wedge \pi_2^* \beta \right). 
\]
In particular, \(  \Upsilon_!\phi \) represents a normal current when \(  \phi  \) is smooth.
   \item \label{it:projection} 
   {\rm (Projection formula)}\ For any \(  \alpha \in \cdf{*}{B} \) one has \(  \Upsilon_!(\phi \wedge \pi_2^* \alpha) = \Upsilon_!(\phi)\wedge \alpha \).
   \item \label{it:type}
   If \(  \phi \)  is a form of type \(  (p,q) \) then \(  \Upsilon_!\phi \) has type \(  (p+i-m, q+i-m ) \).
   \item \label{it:ineqs} There is a constant \(  \lambda >0 \), depending on \(  \Upsilon \), such that for a compact subset \(  K \subset B \)  the following inequality holds:
\[
\norm{ \Upsilon_! \phi }_K \leq \lambda \, {\mathbf{M}}_{K}(\llb \Upsilon \rrb)\, \norm{\phi}_{X\times K},
\]
where \(  \mathbf{M}_K(\llb \Upsilon \rrb) := \mathbf{M}(\chi_{X\times K} \llb \Upsilon \rrb) \), and \(  \chi_{X\times K} \) is the characteristic function of \(  X\times K. \)
\end{enumerate}
\end{proposition}
\begin{proof}
Let \(  \pi_1 \colon X \times B \to X \) and \(  \pi_2 \colon X \times B \to B \) denote the projections and let \(  \pi'_1 \colon \Upsilon \to X \) and \(  \pi'_2 \colon \Upsilon \to B \) 
denote the compositions \(  \pi_1\circ \jmath \) and \(  \pi_2\circ \jmath, \) respectively. 
\smallskip

\noindent{\textsc{Step I:} \emph{Smooth forms}}
\smallskip

Start with a smooth form \(  \phi\in \df{k}{X\times B}. \) 
Given a resolution of singularities \(  p \colon \widehat \Upsilon \to \Upsilon \),  let \(  \rho \colon \widehat\Upsilon \to B \) be the composition \(  \rho = \pi_2 \circ (\jmath\circ p) =  \pi_2'\circ p \). The constructions that follow are summarized in the following commutative diagram.
\[
\xymatrix{
 & &  
 \widehat{\Upsilon}  \ar[dl]_-{p} \ar[ddl]^(.3){\rho}|!{[rd];[ld]}\hole    &  & \widehat{\Upsilon}_{|V} = \rho^{-1}(V)  \ar@{_{(}->}[ll]^-{}  \ar[ddl]^(.3){\rho} \ar@{..>}[dl]  \\
X \times B \ar[dr]_{\pi_2} & \Upsilon \ar@{_{(}->}[l]_-{\jmath} \ar[d]_(0.35){\pi'_2} & &  \Upsilon_{|V}  = \pi_2'^{-1}(V) \ar[d]_(0.35){\pi_2'}  \ar@{_{(}->}[ll]_{} \\
& {B} & & V\ar@{_{(}->}[ll]^-{} &
 }
\]
By generic flatness and  smoothness, one can find a Zariski open (dense) \(  V \subset B  \) such that \(  \rho \colon    \widehat{\Upsilon}_{|V} \to V \) is smooth and \(  \pi'_2  \colon \Upsilon_{|V} \to V \) is flat.

Since algebraic-geometric smooth maps are submersions, it follows from Ehresman's fibration theorem that - in the analytic topology - they are smooth fiber bundles, in the differential geometric sense. 
%
%
Therefore, given \(  \phi \in \df{*}{X\times B} \), we can define a form \(  \hat{\rho}_!(\phi) \) on \(  V \) by
\begin{equation}
\label{eq:iaf_borel}
\hat\rho_!(\phi) =
\rho_!\left\{  (\jmath\circ p)^* \phi \right\}_{|\rho^{-1}V} .\end{equation}
where \(  \rho_!\left\{  (\jmath\circ p)^* \phi \right\}_{|\rho^{-1}V} \) is obtained as integration along the fiber of a smooth fiber bundle projection.

\begin{claim}
\label{claim:generic}
\it
The form \(  \hat\rho_!(\phi)  \) gives a well defined germ at the generic point of \(  B. \) In other words,  given any two resolutions of singularities the resulting forms agree on a non-empty Zariski open.  
\end{claim}
To prove the claim, recall that  any two resolutions are dominated  by a third one, and hence it suffices to assume that the resolution \(  p' \colon {\Upsilon}' \to \Upsilon  \) factors through \(  p\colon \widehat{\Upsilon} \to \Upsilon.  \)  Hence we have a proper birational isomorphism \(  \pi \colon \Upsilon' \to \widehat{\Upsilon} \) so that 
\(  {p}' = p \circ \pi.  \)
\begin{equation}
\label{eq:resols}
\xymatrix{
\Upsilon'_{|{V}'}\  \ar@{^{(}->}[r]^-{}  \ar[dd]_{\rho'} & \Upsilon'  \ar[ddr]_{\rho'} \ar[rr]^{\pi} \ar[dr]^-{p'}& &  \widehat{\Upsilon} \ar[dl]_-{p} \ar[ddl]^{\rho}  & \widehat{\Upsilon}_{|V} \ar@{_{(}->}[l]^-{}  \ar[dd]^{\rho} \\
& & \Upsilon \ar[d]|(0.35){\pi'_2} & \\
{V}'\  \ar@{^{(}->}[rr]^-{} & & {B} & & V\ar@{_{(}->}[ll]^-{}
 }
\end{equation}

Let \(  \emptyset \neq {V}' \subset B \) be a Zariski open so that \(  \Upsilon'_{|V'} \xrightarrow{\tilde{\rho}} V' \) is a smooth map.  Then \(  \pi \) sends \( \Upsilon'_{|V\cap V'} \) to \( \ \widehat{\Upsilon}_{|V\cap V'} \). Given \(  \phi \in \df{q}{X\times B} \) denote \(  \phi' := (\jmath \circ p')^*\phi  \in \df{q}{\Upsilon' }\) and  \( \widehat\phi:= (\jmath \circ p)^*\phi  \in \df{q}{\widehat\Upsilon } \), hence \(   \phi' = \pi^* \widehat\phi \). Then for a form \(  \alpha \in \dfcs{*}{V\cap V'} \) one has
\begin{align*}
\int_{V\cap V'} \alpha  \wedge \iaf{\hat\rho}' & (\phi)
:= 
\int_{\Upsilon'_{|V\cap V'}} ({\rho'}^*\alpha) \wedge \phi' 
=
\int_{\Upsilon'_{|V\cap V'}} (\pi^* {\rho}^*\alpha) \wedge \pi^*\widehat\phi \\
& =
\int_{\pi\left(\Upsilon'_{|V\cap V'}\right)} ({\rho}^*\alpha) \wedge \widehat\phi 
 =
\int_{\widehat{\Upsilon}_{|V\cap V'}} ({\rho}^*\alpha) \wedge \widehat\phi \ = \ \int_{V\cap V'} \alpha \wedge \iaf{{\hat\rho}}(\phi).
\end{align*}
Since the identity above holds for every \(  \alpha \in \Upsilon_c(V\cap V', \cala^*) \), one derives an equality of smooth forms
\(\,  {\iaf{\hat\rho}'}( \phi)_{|V\cap V'} = \iaf{\hat\rho}( \phi)_{|V\cap V'},  \)
thus proving the claim. 
\medskip

Now, fix a resolution of singularities \(  p \colon \widehat \Upsilon \to \Upsilon,  \)  and let \(  V \subset B \) be the domain of \(  \hat\rho_!(\phi) \), as above. 
Define a preliminary form  \(  \psi_{\text{o}}  \) on \(  B \) by 
\begin{equation}
\label{eq:psio}
\psi_\text{o} :=
\begin{cases}
 \hat\rho_!(\phi) &, \text{ on } V \\
 0 &, \text{ on } G = B - V.
\end{cases}
\end{equation}

Write \(  G = \bigcap_{n\in \bbn} U_n \) where \(  G \subset U_{n+1}\subset \overline{U}_{n+1} \subset U_n \subset \cdots \)\  \  is a nested family of  neighborhoods of \( F \). For each \(  n\in \bbn \) choose a smooth function \(  \sigma_n \colon B \to [0,1] \) satisfying \(  \sigma_n(x) = \begin{cases} 1 &, \text{ if } x \notin U_n \\ 0 &, \text{ if } x \in \overline{U}_{n+1}, \end{cases}\) and define 
\(  \theta_n~=~\sigma_n  \, \iaf{\hat\rho} ( \phi ) \) on \(  V \) 
and \(  \theta_n \equiv 0  \) on \(  U_{n+1}.  \) 
Then each \(  \theta_n \) is a smooth form on \(  B \) and \(  \theta_n(x) \xrightarrow[n\to\infty]{} \psi_\text{o} (x) \) for all \(  x\in B.  \) This shows that \(  \psi_\text{o} \) is a Baire form that is smooth on \(  V. \)

Now, observe that  the slicing map \(  t\in V \mapsto \la \llb \widehat \Upsilon \rrb, \rho , t \ra \) is continuous and  that \( \la \llb \widehat\Upsilon \rrb, \rho , t \ra = \llb \widehat\Upsilon_t \rrb  \) is the current given by algebraic cycle associated to the scheme theoretic fiber \(  \widehat\Upsilon_t = \rho^{-1}(t).\) Therefore,  we have a continuous family of effective algebraic cycles on a projective variety; see \cite[Thm. 3.3.2]{King-var}. Since V is connected, the degree of each cycle is constant and the mass (associated to the metric induced by the Fubini-Study metric of some projective embedding of  \(  X \)) of the fibers \( \llb \widehat{\Upsilon}_t \rrb  \) is uniformly bounded. As a result we conclude that   \(  \psi_\text{o} \) is a bounded Baire form.

It follows that \(  \psi_\text{o} \) represents a unique class \(  [\psi_\text{o} ] \) in \(  \llocdf{k-2(m-i)}{B} \). Furthermore, given   \( \alpha \in \dfcs{2(m+n-i) - k}{B}   \)
one has 
\begin{align*}
\label{}
\int_{B} \alpha \wedge \psi_\text{o} & =  
\int_{V} \alpha \wedge \iaf{\hat\rho}(\phi)  & (B - V  \ \text{ has measure } \ 0)\\
& = \int_{\widehat{\Upsilon}_{|V}} \rho^*\alpha \wedge  \widehat{\phi}_{|\rho^{-1}V}  &  \text{(Fubini)} \\
&= \int_{\widehat{\Upsilon}}\rho^*\alpha \wedge  \widehat{\phi} & \text{(measure \( 0 \) argument again)}\\
&=
\int_{\widehat{\Upsilon}} (\jmath\circ p)^*\pi_2^*\alpha \wedge   (\jmath\circ p)^* \phi & \text{(by definition)}\\
& =
\llb \widehat{\Upsilon} \rrb \left(  (\jmath\circ p)^* ( \pi_2^*\alpha \wedge \phi)  \right) \\
& =
(\jmath\circ p)_\#\llb \widehat{\Upsilon} \rrb ( \pi_2^*\alpha \wedge  \phi) &\\
 & =
\llb \Upsilon\rrb ( \pi_2^*\alpha \wedge \phi).  & \text{(\( p \) is birational isomorphism )}
\end{align*}
It follows that
\begin{equation}
\label{eq:repres}
\int_B \psi_\text{o} \wedge \alpha = (-1)^k\int_B \alpha \wedge \psi_\text{o} = (-1)^k\llb \Upsilon\rrb ( \pi_2^*\alpha \wedge \phi) = \llb \Upsilon\rrb ( \phi  \wedge \pi_2^*\alpha) .
\end{equation}
In conclusion,  \textbf{the form \(   \psi_\text{o}  \) represents the current \(  \pi_{2\#}\left(  \llb \Upsilon \rrb \lefthalfcup \phi \right) \).}
\smallskip

In order to prove the first three statements in the proposition, we first show that \(  \hat\rho_! (\phi) \) has a  continuous (and hence unique) extension to \(  U:=B-F \), the domain over which \(  \Upsilon  \) is equidimensional.

%
%
Fix \(   p_0 \in U \). 
Using partitions of unity, it suffices to assume that \(  \phi \)  is supported on \(  X\times W, \) where \(  W \) is the domain of a coordinate chart \(  \Psi \colon W \to W'\subset \bbc^n \) with coordinates \( \mbz = (z_1, \ldots, z_n)  \) so that \(  \Psi(p_0) = \mathbf{0} ,  \) and that  \(  \phi \) has the form 
\begin{equation}
\label{eq:phi_form}
  \phi = g\, \pi_1^*\alpha\wedge \pi_2^* \beta ,
\end{equation}
where \(  g = g(\mbx, \mbt) \) is a smooth function on \(  X\times B \),  \(  \alpha \in \cala^*(X) \)
 and  \( \beta \in \cala^*(B) \)  is given in coordinates by \(  \beta = h(\mbz)dz_I\wedge d\bar{z}_J.  \)

If \(  \deg{\alpha} \neq 2(m-i) \) define \(  \Upsilon_!(\phi) = 0. \) 
Now, assume \(  \deg{\alpha } = 2(m-i) \),  \(|I|+|J| = k -2(m-i)  \) and let
\(  \sigma_{I,J} \) be the sign of the shuffle so that  \( \Omega =  \sigma_{I,J}  dz_I\wedge d\bar{z}_J\wedge dz_{I^C}\wedge d\bar{z}_{J^C}  \),
 is the volume form in \(  \bbc^n  \).

Define \(  \gamma^\epsilon \) in local coordinates by 
\begin{equation}
\label{eq:bump} 
\gamma^\epsilon = \sigma_{I,J}\, f^\epsilon(\mbz)   dz_{I^C}\wedge d\bar{z}_{J^C} ,
\end{equation}
where 
\(  f^\epsilon(\mbz)  \) is a ``bump'' function whose support is contained in the \(  \epsilon \)-ball \(  D_\epsilon \subset W' \) around \(  \mathbf{0} \) in \(  \bbc^n  \), with \(  \int f^\epsilon(\mbz) d\call_{2n}(\mbz) = 1.  \)
Hence, \(  \beta \wedge \gamma^\epsilon =  h(\mbz) f^\epsilon(\mbz) \wedge \Omega.\) 

It follows from \cite{Fed-some} that
   \begin{align}
\int_U &  h(\mbz) f^\epsilon(\mbz)\,  \la \llb \Upsilon \rrb, \pi_2, \mbz \ra \left( g \, \pi_1^*\alpha\right) d\call_{2n}(\mbz)  = 
\left( \llb \Upsilon \rrb \lefthalfcup \pi_2^* ( h f^\epsilon\wedge \Omega) \right) \left( g\,  \pi_1^* \alpha \right)  \notag \\
\label{eq:lim2}
& =
 \llb \Upsilon \rrb 
 \left( g\, \pi_1^* \alpha   \wedge \pi_2^*( \beta \wedge \gamma^\epsilon ) 
    \right)    =
 \llb \Upsilon \rrb 
\left( 
\left\{ g\, \pi_1^* \alpha   \wedge \pi_2^* \beta \right\}
\wedge \pi_2^*\gamma^\epsilon \right) \\
& =
 \llb \Upsilon \rrb 
\left( 
\phi
\wedge  \pi_2^*\gamma^\epsilon \right) = \pi_{2\#}\left( \llb \Upsilon \rrb \lefthalfcup \phi \right) (\gamma^\epsilon)   \overset{\eqref{eq:repres}}{=}
\int_U  \psi_\text{o}\wedge 
 \gamma^\epsilon \notag
\end{align}
Therefore
\begin{align}  
\label{eq:lim3}
\lim_{\epsilon \to 0 }   \int_U  \psi_\text{o}\wedge 
 \gamma^\epsilon  =  \lim_{\epsilon \to 0} \int_U    & = 
 \lim_{\epsilon \to 0} \int_U   h(\mbz) f^\epsilon(\mbz) \la \llb \Upsilon \rrb, \pi_2, \mbz \ra \left( g \, \pi_1^*\alpha\right) d\call_{2n}(\mbz) \\
 & =
  h(p_0)  \la \llb \Upsilon \rrb, \pi_2, p_0 \ra \left( g \, \pi_1^*\alpha\right) \notag
 \end{align} 
Using the fact that the slicing function \(  p_0 \mapsto  \la \llb \Upsilon \rrb, \pi_2, p_0 \ra \) is a continuous function on \(  U \) \cite[Thm. 4.3]{Hardt-slice} one concludes that the last term in \eqref{eq:lim3}  is a continuous function at \(  p_0.  \) We conclude that the first term shows how to (re)define \(  \psi_\text{o} \) at \(  p_0  \) to make it continuous on \(  U. \) In other words, we can extend  \(  \hat\rho_!(\phi) \) \eqref{eq:iaf_borel} to a bounded Baire form \(  \psi_\text{o} \) which is continuous on \(  U = B -F. \) In particular, this extension does not depend on the resolution of singularities and is generically smooth. 
Finally, define
\begin{equation}
\label{eq:IAF2}
\Upsilon_!(\phi) =
\begin{cases}
\psi_\text{o} &  \text{ on }  U= B-F \\
0 & \text{ on } F.
\end{cases}
\end{equation}
This is the desired form, satisfying the first three statements of the proposition.
\smallskip

\noindent{\textsc{Step II:} \emph{Continuous forms}}
\smallskip

Now, let  \(  \phi \) be a continuous form. One can also assume that \(  \phi \) has the form \eqref{eq:phi_form}, with \(  g \) continuous. 
The essential ingredient here is the fact that the slices \(  \la \llb \Upsilon \rrb, \pi_2, \mbt \ra \) are normal currents, which are in fact represented by effective algebraic cycles, and hence they can be applied to continuous forms. Now, the  arguments in \eqref{eq:lim2} and \eqref{eq:lim3} apply, and show that \(  \Upsilon_!\phi \) can be defined and is a continuous form on \(  U= B-F \), and the same continuity arguments still hold in the continuous case to show that the extension of \(  \Upsilon_!\phi \) by zero on \(  F \) defines a bounded Baire form on \(  B. \)

To prove the last assertion of the proposition,   note that for each compact  \( K \subset B  \),  the current 
\( \pi_{2\#}  \left\{ \left[ \llb \Upsilon \rrb \cap (X\times K) \right] \lefthalfcup \phi_{|X\times K} \right\} \)
is represented by the form \(  \chi_K \cdot \Upsilon_! \phi \). 
The last assertion now follows directly from definitions.
\end{proof}

\begin{remark}
\label{rem:generic_pts}
\begin{enumerate}[i.]
\item When \(  \Upsilon  \) is not dominant over \(  B, \) we define \(  \Upsilon_!  \) as the zero map,   to have a homomorphism
\(  \sigma_! \colon\cdf{k}{X\times B} \to  \bbdf{k-2(m-i)}{B}  \)
associated to  any algebraic cycle \(  \sigma \in \calz^i(X\times B)  \).
\item
We also use \(  \Upsilon_!\phi \)\ \   to denote the class in \( \llocdf{*}{B} \) represented by \(  \Upsilon_!\phi. \) Restricting \(  \scL^1_\text{loc} \otimes_{\scA^0} \scA^*_B \) to the Zariski topology on \(  B, \) the proposition above states that the germ  \(  (\Upsilon_!\phi)_\sfp \) is continuous whenever \(  \sfp \) is a point of codimension \(  1 \). Furthermore, this germ is \(  \scC^\infty \) at the generic point \(  \eta \) when \(  \phi \) is a smooth form.
\end{enumerate}
\end{remark}
The \emph{continuity in codimension \(  1 \)} described in the remark above is essential in the  construction of the desired chain maps from higher Chow groups.  The following result is a first  step in that direction. 

\begin{proposition}
\label{lem:hyper}
Consider a smooth embedding \(  \jmath \colon H \hookrightarrow B \) and let \(  \Upsilon \subset X\times B \) be irreducible and dominant over \(  B,  \) as in Proposition \ref{prop:int_along_fib}, with its structure of reduced closed subscheme of \(  X\times B.  \) Let \(  \Upsilon_{|H}\) denote the algebraic cycle in \(  X\times H \) associated to the closed subscheme \(  \Upsilon \times_B H.  \) Given \(  \phi \in \df{*}{X\times B} \), the following holds. 
\begin{enumerate}[i.]
\item If \(  H \) is a hypersurface in \( B \) then the forms
 \(  \left( \Upsilon_{|H}\right)_! (1\times \jmath)^*(\phi)\) and  \(  \jmath^* \Upsilon_!(\phi) \) coincide on a dense Zariski open subset of \(  H. \) 
\item \label{lem:ii}If the codimension of \(  H \)  is arbitrary but \(  \Upsilon  \) is equidimensional and dominant over \(  B \), in the sense of Definition \ref{def:rel_equid},  then 
 \begin{equation}
 \label{eq:hyper}
  \left( \Upsilon_{|H}\right)_! (1\times \jmath)^*(\phi)\ = \ \jmath^* \Upsilon_!(\phi) .
  \end{equation}
\end{enumerate}
\end{proposition}
\begin{proof}
Let \(  F \subset B \) be the closed subvariety so that \(  \Upsilon_!\phi \) is continuous on \(  B-F \), and let \(  F' \subset H \) denote the corresponding subvariety for \(  \left( \Upsilon_{|H}\right)_! (1\times \jmath)^*(\phi) \). Recall that  \(  F \) and \(  F' \) have codimension at least \(  2,   \) and hence \(  F'' := F' \cup (F\cap H) \) is a proper Zariski closed subset  of \( H.   \) It follows that  both
\(  \left( \Upsilon_{|H}\right)_! \{ (1\times \jmath)^*(\phi)  \}\) and  \(  \jmath^* \Upsilon_! \phi \) are continuous on \(  B- F''.  \)

Let 
\(  \Phi_H^\epsilon  \in \df{2}{B} \) be a Thom form for the normal bundle of \(  H \) whose support is contained in an \(  \epsilon \)-neighborhood of \(  H  \) in \(  B  \).
Given \(  \alpha\in \df{*}{H} \) pick some \(  \hat\alpha \in \df{*}{B}  \) such that \(\jmath^*\hat\alpha = \alpha  \). Then

\begin{align*} 
 \int_H & \jmath^*\Upsilon_!(\phi) \wedge \alpha = \int_H \jmath^*\left\{ \Upsilon_!(\phi) \wedge \hat\alpha\right\}  = \lim_{\epsilon\to 0} \int_B \Phi_H^\epsilon  \wedge \Upsilon_!(\phi) \wedge \hat\alpha \\
 & =
 \lim_{\epsilon\to 0} \int_B \Upsilon_!(\phi) \wedge \Phi_H^\epsilon  \wedge \hat\alpha  
    =  \lim_{\epsilon\to 0} \ \ \llb \Upsilon \rrb \left(  \phi \wedge \pi_2^* ( \Phi_H^\epsilon \wedge \hat \alpha ) \right)       \\
    & =  \lim_{\epsilon\to 0} \ \ \int_\Upsilon    \phi \wedge \pi_2^* ( \Phi_H^\epsilon \wedge \hat \alpha ) = 
 \lim_{\epsilon\to 0}   \int_\Upsilon    \pi_2^*\Phi_H^\epsilon \wedge  \phi \wedge \pi_2^*(\hat \alpha ).
 \end{align*}
Since \(  \pi_2^*\Phi^\epsilon_H \) is a Thom form for the normal bundle of \(  X\times H \) in \(  X\times B \), one concludes from the identities above that
\(
 \int_H  \jmath^*\Upsilon_!(\phi) \wedge \alpha \ = \ 
 \left( [X\times H] \cap [\Upsilon ]\right) (\phi\wedge \pi_2^*\hat\alpha) .
\)
Here we use  \(  [\Upsilon] \cap [X\times B] \) to denote both the intersection of algebraic cycles and its associated current.

Applying the identity
 \( (1\times \jmath)_* \left( \Upsilon_{|H}\right) = [ \Upsilon ] \cap [ X \times H] \) \cite[]{Ful-IT}
of algebraic cycles on \(  X\times B \) one concludes that 
\begin{align*}
\int_H  \jmath^*\Upsilon_!(\phi)  \wedge \alpha & = 
(1\times \jmath)_\# \llb \left( \Upsilon_{|H}\right) \rrb
(\phi\wedge \pi_2^*\hat\alpha) \ = \
\llb \left( \Upsilon_{|H}\right) \rrb
(1\times \jmath)^* (\phi\wedge \pi_2^*\hat\alpha) \\
& =
\llb \left( \Upsilon_{|H}\right) \rrb\left( 
(1\times \jmath)^* (\phi) \wedge \pi_2^* \jmath^*\hat\alpha \right) =
\llb \Upsilon_{|H} \rrb\left( 
(1\times \jmath)^* (\phi) \wedge \pi_2^* \alpha \right) \\
& =
\int_H  \left(\Upsilon_{|H}\right)_!\left( (1\times \jmath)^*\phi \right) \wedge \alpha.
\end{align*}
Since \(  \alpha \) is arbitrary, this identity shows that 
\(  \jmath^*\Upsilon_!(\phi)  \) and 
\(  \left(\Upsilon_{|H}\right)_!\left( (1\times \jmath)^*\phi \right)  \)
coincide in the domain \(  B- F'' \) where both forms are continuous. This proves the first statement. 

For the second statement, observe that since \(  \Upsilon \to B \) is a proper, dominant map, it is surjective and hence the algebraic cycle \( \Upsilon_{|H} \) is  equidimensional and surjective over \(  H \), as well. In particular both forms  \(  \left( \Upsilon_{|H}\right)_! (1\times \jmath)^*(\phi)\) and  \(  \jmath^* \Upsilon_!(\phi) \) are continuous on \(  H \). Now, one can use the same arguments as in the proof of the first statement and the result follows. 
\end{proof}

\subsection{Correspondences and transforms of  currents}

In this section we show how to construct the pull-back homomorphisms on currents that are represented by integration, induced by correspondences, and prove the main properties of these constructions. . 

\subsubsection{The projective case}
\label{subsubsec:proj_case}
\begin{definition}
\label{def:pull-back_def}
Let \(  X \) be a connected, smooth projective variety and let \(  B \) be a smooth variety, with \(  \dim{X} =m \) and \(   \dim{B} = n \). 
Given an  irreducible subvariety \(\Upsilon \subset X \times B\)  of codimension \(  i \), define a pull-back homomorphism
\begin{equation}
\label{eq:pb-cur}
\Upsilon^\#\ \colon\ \curmeas{k}{B} \longrightarrow \curmeas{k+2i}{X\times B}, \quad  k \geq 0,
\end{equation}
by sending \(  S \in \curmeas{k}{B} \) to the current \( \Upsilon^\#S \) defined  on  \( \phi  \in \cdfcs{2(m+n-i)-k}{X\times B} \) 
as \(  S( \Upsilon_!(\phi) )  \). In other words, the pull-back is the adjoint of the generalized integration along the fiber. 
\end{definition}
In the proof  below we  show that this pull-back operation above  is well-defined. 
\begin{proof}(of Theorem \ref{prop:sharp})
The assignment \(  \phi \mapsto S(\Upsilon_!\phi) = S(\hat\rho_!\phi )\) is well-defined since \(  \Upsilon_!\phi \) is  a compactly supported bounded Baire form, as shown in Proposition \ref{prop:int_along_fib},  and \(  S \) is represented by integration.  Now,  let  \(  K \subset B \) be a compact set such that \(  \supp(\phi) \subset X \times K\) and, for a current \(  T \in \curmeas{*}{X\times B} \),  denote \(  \mathbf{M}_K(T) := \mathbf{M}(\chi_{X\times K} T), \) where \(  \chi_{X\times K} \) is the characteristic function of \(   X\times K. \)
Then 
\[  |S(\Upsilon_!\phi )| \leq \mathbf{M}_{K}( S) \, \norm{\Upsilon_! \phi} \leq  \mathbf{M}_{K}(S) \, \mathbf{M}_K(\llb \Upsilon \rrb )\, \norm{\phi},\]  where the last inequality follows from Proposition \ref{prop:int_along_fib}.\ref{it:ineqs}.
This suffices to show that \(  \Upsilon^\#S \) is indeed a current representable by integration.

Finally, the fact that \(  \Upsilon^\# \) sends a current of type \(  (r,s) \) to a current of type \(  (r+i, s+i) \) follows directly from Proposition \ref{prop:int_along_fib}.\ref{it:type}.
\end{proof}
\begin{remark}
\label{rem:L1loc}
If the current \(  S \in \curmeas{k}{B} \) is given by a form \(  \omega \in 
\llocdf{k}{B} \) then it follows from Proposition \ref{prop:int_along_fib}.\ref{it:Fubini}  that \(  \Upsilon^{\#}S = \llb \Upsilon \rrb \lefthalfcup  \pi_2^* \omega .   \)
\end{remark}

 \subsubsection{Properties of the pull-back and transform operations}
\label{subsubsec:proper}

We now proceed to prove Proposition \ref{cor:nor_gen} and discuss a few more properties of the pull-back maps. 

\begin{proof}(of Proposition \ref{cor:nor_gen})
Consider  \(  \Upsilon \subset X \times B \)  as in Theorem \ref{prop:sharp} and let
\(  S \) be a \emph{locally normal} current in \(  B \) whose support is contained in \(  B - F \), the domain over which \(  \Upsilon  \) is equidimensional. 

To prove the first assertion in the proposition, it suffices to assume that \(  \Upsilon \subset X \times B\) is equidimensional  and dominant over \(  B  \) and that \(  S \) is a normal current in \(  B. \) 
The following commutative diagram summarizes the notation for the indicated projections. 
\begin{equation}
\label{eq:diag_proj}
\xymatrix{
 & & X \times B \times B \ar[dd]^{\pi_2} \ar@/_{2.2pc}/[ddll]_{\pi_1} \ar@/^{2.2pc}/[ddrr]^{\pi_3}  \ar[dl]_{\pi_{12}}\ar[dr]^{\pi_{23}}& &  \\
& X\times B \ar[dl]^{\rho_1} \ar[dr]_{\rho_2}   & & B\times B  \ar[dl]^{q_1} \ar[dr]_{q_2} & \\
X & & B  & &   B 
}
\end{equation}

\begin{lemma}
\label{lem:corresp1}
Let \(  \Delta \colon B \to B\times B \) denote the diagonal inclusion and let \(  \Delta_B  \subset B\times B\) denote the diagonal. With \(  \Upsilon  \) and \(  S  \) as above then \(  \Upsilon \times B \) is equidimensional over \(  B\times B \) and 
\begin{enumerate}[i.]
\item 
\begin{equation} 
\label{eq:diag_prod}
(\Upsilon \times B)^\#(\llb B \rrb \times S) \ = \ \llb \Upsilon \rrb \times S. 
\end{equation}
\item 
Given \(  \varphi \in \cala^*(X\times B\times B) \) one has
\begin{equation}
\label{eq:diag_IAF}
\Delta^*\left( \{\Upsilon\times B\}_!(\varphi) \right) \ = \ 
\Upsilon_!\left( (1\times \Delta)^* \varphi \right).
\end{equation}
\end{enumerate}
\end{lemma}
\begin{proof}(of Lemma)
First we show the following
\begin{claim}
\label{claim:IAF}
If \(  \varphi \) is a form on \(  X\times B\times B \) of the form \(  \pi_{12}^*\alpha \wedge \pi_3^* \beta \), with \(  \alpha \in \df{r}{X\times B} \) and \(  \beta \in \df{s}{B} \) then
\begin{equation}
\label{eq:prod_IAF}
(\Upsilon \times B)_!( \varphi) \ = \ q_1^*( \Upsilon_!\alpha) \wedge q_2^*(\beta) .
\end{equation}
\end{claim}
To prove the claim, pick a form \(  \theta  \) in \(  B\times B  \) of the form \(  \theta = q_1^*\tau \wedge q_2^*\eta \), with \(  \tau \in \df{2(m+n-i) -r}{B} \) and \(  \eta \in \df{2n-s}{B} \). Then
\begin{align*}
\int_{B\times B} &  (\Upsilon\times B)_!(\varphi) \wedge \theta =
\llb \Upsilon \times B \rrb\left(  \varphi \wedge \pi_{23}^*\theta  \right)  \\ 
& = 
\llb \Upsilon \times B \rrb\left(  \pi_{12}^*\alpha \wedge \pi_3^*\beta  \wedge \pi_{23}^*(q_1^*\tau \wedge q_2^* \eta)   \right) \notag\\
& =
(\llb \Upsilon \rrb \times \llb B \rrb)\left(  \pi_{12}^*\alpha \wedge \pi_3^*\beta  \wedge \pi_2^*\tau \wedge \pi_3^* \eta   \right) \\ 
& =
(-1)^{r+s}(\llb \Upsilon \rrb \times \llb B \rrb)\left(  \pi_{12}^*\alpha   \wedge \pi_2^*\tau\wedge \pi_3^*\beta \wedge \pi_3^* \eta   \right) \\
& =
(-1)^{r+s}(\llb \Upsilon \rrb \times \llb B \rrb)\left(  \pi_{12}^*\{ \alpha   \wedge \rho_2^*\tau\}\wedge \pi_3^*\{ \beta \wedge  \eta\}   \right) \\
&= 
(-1)^{r+s}\llb \Upsilon \rrb \left(  \alpha   \wedge \rho_2^*\tau \right) \cdot \llb B \rrb ( \beta\wedge \eta)  \\
& =
(-1)^{r+s}  \left( \int_B \Upsilon_!\alpha \wedge \tau \right)\cdot \left( \int_B \beta\wedge \eta \right) \\
&= (-1)^{r+s}   \int_{B\times B} q_1^* (\Upsilon_!\alpha \wedge \tau )\wedge q_2^*( \beta\wedge \eta  ) \\
&=     \int_{B\times B} q_1^* (\Upsilon_!\alpha ) \wedge q_2^*( \beta) \wedge q_1^*( \tau ) \wedge q_2^*( \eta  ) \\
& =    \int_{B\times B} q_1^* (\Upsilon_!\alpha ) \wedge q_2^*( \beta) \wedge \theta. 
\end{align*}
The claim now follows from the identities above and the continuity of the forms \(  (\Upsilon\times B)_! (\varphi) \) and \(  q_1^*(\Upsilon_!\alpha)\wedge q_2^*(\beta) \).

From the definitions and the Claim, it follows that for a form \(  \varphi  \) as above one has
\begin{align*}
 (\Upsilon\times B)^\#\left( \llb B \rrb\times S \right)(\varphi) &  := \left( \llb B \rrb\times S \right)\left( (\Upsilon\times B)_! \varphi \right) \\
&   \overset{\rm Claim \ref{claim:IAF}}{=} 
\left( \llb B \rrb\times S \right) ( q_1^* (\Upsilon_!\alpha ) \wedge q_2^*( \beta)) 
=
\llb B \rrb (\Upsilon_!\alpha) \cdot S(\beta)  \\
&  = \llb \Upsilon \rrb(\alpha) \cdot S(\beta)
 = 
 \left( \llb \Upsilon \rrb \times S \right)(\varphi). 
\end{align*}
This proves the first assertion of the Lemma. 

The second assertion follows directly from Proposition \ref{lem:hyper}.\ref{lem:ii}. Indeed, one just  needs to replace \(  B \) by \(  B\times B \) and \(  \jmath \colon H \hookrightarrow B\) by \(  \Delta \colon B \hookrightarrow B \times B \) in the statement, and use the fact that \(  \Upsilon \equiv B\times_{\Delta} (\Upsilon \times B)  = (\Upsilon \times B)_{| \Delta_B}.   \)
\end{proof}
Now,  we conclude the proof of the first assertion in Proposition  \ref{cor:nor_gen}. Using partitions of the unity, it suffices to consider a form \(  \varphi \in \df{*}{X\times B\times B} \) whose support is contained in \(  X\times K \times K \), where \(  K \subset V \) is a compact contained in the domain  of a coordinate chart \(  \psi \colon V \xrightarrow{\cong} V'\subset \bbc^n.  \)

Then
\begin{align}
(1\times \Delta)_\#  (\Upsilon^\#S) (\varphi) & \overset{\rm def}{=} \
\Upsilon^\#S ( (1\times \Delta)^*\varphi )   \overset{\rm def}{=} S \left( \Upsilon_!\left( (1\times \Delta)^*\varphi \right) \right) \\
& \overset{\eqref{eq:diag_IAF}}{=}
S\left(\Delta^*\left( \{\Upsilon\times B\}_!(\varphi) \right) \right)  = (\Delta_\# S) \left(  \{ \Upsilon \times B \}_!(\varphi)  \right).
\end{align}

Given a form \( \beta  \) in \(  B\times B \), with \(  \supp{\beta} \subset K \times K \), let \(  \delta \colon V\times V \to \bbc^n \) be the ``difference map'' \(  \delta(u,v) = \psi(u)-\psi(v). \) Then 
\begin{align}
 (\Delta_\# S) \left(  \beta \right) & = \la \llb B\rrb \times S, \delta, 0 \ra (\beta) \ = \ \lim_{\epsilon\to 0} \int_B f^\epsilon(b) \la \llb B\rrb \times S, \delta, b \ra(\beta) d \mathscr{L}_{2n}(b) \\
 & =
  \lim_{\epsilon\to 0} (\llb B\rrb \times S) \left( \beta \wedge \delta^*(f^\epsilon \Omega) \right) 
\end{align}
On the other hand, taking \(  \beta = \{ \Upsilon\times B\}_!(\varphi) \), one has
\begin{align*}
(\llb B\rrb \times S) \left( \beta \wedge \delta^*(f^\epsilon \Omega) \right) & =
(\llb B\rrb \times S) \left(\{ \Upsilon\times B\}_!(\varphi) \wedge \delta^*(f^\epsilon \Omega) \right)  \\
& \overset{\rm def}{=}
\{ \Upsilon\times B\}^\# (B\times S) (\varphi \wedge \pi_3^*\delta^*(f^\epsilon \Omega) )  \\
&= 
\{ \Upsilon\times B\}^\# (B\times S) (\varphi \wedge {\delta'}^*(f^\epsilon \Omega) ) \\
& \overset{\eqref{eq:diag_prod}}{=}
(\llb\Upsilon \rrb\times S) (\varphi \wedge {\delta'}^*(f^\epsilon \Omega) ) \\
& = 
\int_B f^\epsilon(b) \la \llb\Upsilon \rrb\times S, \delta', b \ra (\varphi) d\mathscr{L}_{2n}(b).
\end{align*}
Taking the limits when \(  \ve\to 0 \) one gets:
\[
\la \llb\Upsilon \rrb\times S, \delta', 0 \ra (\varphi) \ = \
 (1\times \Delta)_\# \left( \Upsilon^\# S \right) (\varphi). 
\]
One concludes that  \(  \{ \llb\Upsilon \rrb\cap  ( \llb X \rrb \times S)\}(\varphi)  =  \left( \Upsilon^\# S \right) (\varphi). \)

To prove the second assertion of Proposition \ref{cor:nor_gen}, assume that \(  S \) is normal and vanishes suitable along \(  H. \) Given \(  \varphi  \in \df{k}{X\times B} \), the condition 
\(  \norm{S}(H) = 0  \) gives 
\(  \Upsilon^\#S(d\varphi) := S (\Upsilon_!(d\varphi)) = \{S\cap (B-H)\}(\Upsilon_!(d\varphi)) . \)
Since \(  S \) is normal one has
\[  \{ S \cap (B-H)\}(\Upsilon_!(d\varphi)) = \lim_{\tau \to 0} 
\{ S \cap (B-W^\tau)\}(\Upsilon_!(d\varphi)).  \] 

Now, for each \(  \tau > 0  \) sufficiently small the current \(  S\cap (B-W^\tau)  \) is normal, by hypothesis, and its support is contained in \(  B-W^\tau \subset B-F.  \) The first assertion of the proposition then gives
\begin{align*}
 & \{ S \cap (B-W^\tau)\} (\Upsilon_!(d\varphi))   := 
\Upsilon^\#\left( S \cap (B-W^\tau)\right)(d\varphi) \\
&  =
\left( \llb \Upsilon \rrb \cap \{ S \cap (B-W^\tau)\} \right)(d\varphi) =
\partial \left( \llb \Upsilon \rrb \cap \{ S \cap (B-W^\tau)\} \right)(\varphi)\\
& =
\left( \partial \llb \Upsilon \rrb \cap \{ S \cap (B-W^\tau)\} \right)(\varphi)   +
\left( \llb \Upsilon \rrb \cap \partial \{ S \cap (B-W^\tau)\} \right)(\varphi)  \\
& =
\left( \llb \Upsilon \rrb \cap \partial \{ S \cap (B-W^\tau)\} \right)(\varphi)  \\
& =
\left( \llb \Upsilon \rrb \cap \{ \partial S \cap (B-W^\tau)\} \right)(\varphi)  +
\left( \llb \Upsilon \rrb \cap \{ S \cap \partial(B-W^\tau)\} \right)(\varphi)  \\
& =
 \{ \partial S \cap (B-W^\tau)\}(\Upsilon_!\varphi)  +
 \{ S \cap \partial(W^\tau)\}(\Upsilon_!\varphi) .
\end{align*}

The hypothesis on \(  S  \) and the identities above give
\begin{align*}
S(\Upsilon(d\varphi)) & = \lim_{\tau\to 0} \{ S \cap (B-W^\tau)\} (\Upsilon_!(d\varphi)) \\
& =
\lim_{\tau\to 0}
 \{ \partial S \cap (B-W^\tau)\}(\Upsilon_!\varphi)  +
\lim_{\tau\to 0}  \{ S \cap \partial(W^\tau)\}(\Upsilon_!\varphi) \\
& =
(\partial S \cap (B-F)) (\Upsilon_!\varphi) = (\partial S ) (\Upsilon_!\varphi)  = \Upsilon^\#(\partial S) (\varphi). 
\end{align*}
This concludes the proof of the proposition. 
\end{proof}

The next result is a direct consequence of  Proposition \ref{lem:hyper}.
Recall that the transform 
\(
\trans{\Upsilon}{} \  \colon\ \curmeas{k}{B}  \longrightarrow \curmeas{k+2(i-n)}{X}\) is given by \( 
S \mapsto \trans{\Upsilon}{S} := \pi_{1\#} \left( \sigma^{\#}S\right).
\)
See Definition \ref{def:transform}.
\begin{proposition}
\label{prop:prop_hyper}
Let \(  \Upsilon \subset X\times B \) be dominant over \(  B,  \) as in Proposition \ref{lem:hyper}. 
Given a smooth hypersurface \(  \jmath \colon H \hookrightarrow B \) and a  current \(  S \in \curmeas{*}{H} \)  representable by integration,
the identity 
\begin{equation}
\Upsilon^\vee_{\jmath_{\#} S} \ = \ \left( \Upsilon_{|H}\right)^\vee_S  
\end{equation}
holds whenever one of the following conditions occur.
\begin{enumerate}[a.]
\item \label{it:equid}
 \(  \Upsilon \) is equidimensional over \(  B \);
\item For each proper Zariski closed subset \(  Z \subsetneq H \) one has \(  \norm{S}(Z) = 0. \) 
\end{enumerate}
\end{proposition}
\begin{proof}
In the first case, consider  \(  \varphi \in \dfcs{*}{X} \). Then
\begin{align}
\label{eq:ids1}
  \Upsilon^\vee_{\jmath_{\#} S} (\varphi) & :=  \Upsilon^\#_{\jmath_{\#} S}(\pi^*_1\varphi)   = \jmath_{\#S}\left( \Upsilon_!(\pi^*_1\varphi)\right) = S\left( \jmath^* \Upsilon_!(\pi^*_1\varphi)\right) \\
  & \overset{\eqref{eq:hyper}}{=} 
  S \left(  \left\{ \Upsilon_{|H}\right\}_! (1\times \jmath)^*(\pi^*_1\varphi )  \right) =   S \left(  \left\{ \Upsilon_{|H}\right\}_! \varphi    \right) :=  \left( \Upsilon_{|H}\right)^\vee_S (\varphi). \notag
\end{align}

Now,  observe that if \( F  \subset B\) is the set over which  \(  \Upsilon \) is not equidimensional, then 
\(  F\cap H \subsetneq H \), since \(  \codim{F} \geq 2.  \)  The second case now follows from Proposition \ref{lem:hyper}.i.  and  the identities \eqref{eq:ids1}. 
\end{proof}
%
%

\subsubsection{The quasi-projective case} \hfill 
\label{subsubsec:quasi_pro}

Let \(  \Upsilon \subset U \times \Delta^n \) be an irreducible subvariety of codimension \(  i \) in \(  U\times \Delta^n \), whose projection onto \(  \Delta^n \) is dominant and has equidimensional fibers.
Let \(  U\hookrightarrow X \hookleftarrow D \) be  a projective compactification 
where \(  D = X-U \) is a divisor with simple normal crossings, and let \(  \overline\Upsilon \subset X \times \bbp^n \) be the closure of \(  \Upsilon \) in \(  X\times \bbp^n \).

We know that Proposition \ref{prop:int_along_fib} gives a homomorphism 
\(  \overline\Upsilon_! \colon \cdf{k}{X\times \bbp^n} \to \bbdf{k}{\bbp^n} \) so that \( \overline\Upsilon_! \phi \) is continuous on \(  \bbp^n - F \), for any continuous form \(  \phi \) on \(  X\times \bbp^n. \)  Unfortunately, the exceptional set \(  F \) may intersect \(  \Delta^n \), and this prevents us from directly using the arguments from the projective case. The redeeming factor is that we are interested in currents with log poles.
\begin{lemma}
\label{prop:infty}
Using the notation above, let \(  \phi \in \cdf{k}{X\times \bbp^n} \) be a continuous form that vanishes at \(  D \times \bbp^n. \) Then \(  \overline\Upsilon_!\phi \) extends continuously to \(  \Delta^n.  \)
\end{lemma}
\begin{proof}
First assume that   \(  \supp{\phi} \cap (D \times \bbp^n) = \emptyset \), and denote \(  \Upsilon' = \overline\Upsilon \cap (U\times \bbp^n). \) 

Using partitions of unity, we may assume that \(  \phi = g \pi_1^*\alpha \wedge \pi_2^*\beta \) \eqref{eq:phi_form}, as in the proof of Proposition \ref{prop:int_along_fib}. 
It follows from the properties of the slicing of analytic chains \cite{Hardt-slice} that the identities \eqref{eq:lim2} and \eqref{eq:lim3} still hold for \(  \Upsilon' \), since \(  \supp{\phi} \) is compact and lies in \(  U\times \bbp^n. \)
As in the proof of Proposition \ref{prop:int_along_fib}, it follows that we can define \(  \Upsilon'_!\phi \) continuously on \(  \Delta^n \) and this clearly coincides with \(  \overline\Upsilon_!\phi \) on \(  \Delta^n - F. \)

In the general case, when  \(  \supp{\phi} \cap (D \times \bbp^n) \neq \emptyset \), we may still assume that \(  \phi \) has the form  \(  \phi = g \pi_1^*\alpha \wedge \pi_2^*\beta \). Fix a product Riemannian metric on \(  X \times \bbp^n\) and let \(  \{ U_n, \rho_n \}_{n\in \bbn} \) be a system of neighborhoods of \(  D \) in \(  X \), with \(  \rho_n \colon X \to [0,1] \) smooth and satisfying 
\begin{itemize}
  \item \( D \subset U_{n+1} \subset \overline{U}_{n+1} \subset U_n,\) for all \(  n\in \bbn; \) 
  \item \( \bigcap_{n\geq 0} U_n = D \), and 
  \item \(  \rho_n \equiv \begin{cases} 1 &, \text{  on }   \overline{U}_{n+1}, \\ 
  0 &, \text{  on }  X - U_n.  
  \end{cases} \)
\end{itemize}

Since \(  D \)  is compact, for each \(  \epsilon >0 \)  there is an \(  n_0 \gg 0\) such that 
\begin{equation}
\label{eq:sup}
 \norm{\rho_n \phi} \leq  \epsilon, \quad \text{ for }\ n \geq n_0.
\end{equation}
Now, denote \( \phi_n := (1-\rho_n)\phi. \) It is clear from the definition and \eqref{eq:sup} that \(  \phi_n \) converges uniformly to \(  \phi \) on \(  X\times \bbp^n \) and \(  \supp{(\phi)}\cap (D \times \bbp^n) = \emptyset.  \) In particular,  \(  \{ \overline \Upsilon_!(\phi_n) \}_{n\in \bbn} \) is a sequence of bounded Baire forms that are continuous on \(  \Delta^n \). We proceed to show that this sequence converges uniformly to \(    \overline \Upsilon_!(\phi)  \).

Denote \(  G_{r,k}(z) :=   \la \llb \Upsilon' \rrb, \pi_2, z \ra ( (\rho_{r}-\rho_{r+k}) g \pi_1^*\alpha)  \) and observe that this is a continuous function on \(\Delta^n \). Since \(  \Upsilon' = \overline\Upsilon \cap (U\times \bbp^n) \) and \(  \supp{(\rho_{r}-\rho_{r+k}) g \pi_1^*\alpha} \subset U\times \bbp^n \),   for each \(  z_0 \in \Delta^n-F \) one has 
\(  G_{r,k}(z_0) =   \la \llb \overline{\Upsilon} \rrb, \pi_2, z_0 \ra ( (\rho_{r}-\rho_{r+k}) g \pi_1^*\alpha) . 
\)

As explained in  the proof of Proposition \ref{prop:int_along_fib}, for \(  z_0 \in \Delta^n - F \), the mass of the slice 
\(  \la \llb \overline{\Upsilon} \rrb, \pi_2, z_0 \ra  \) is bounded by a constant \(  d \), regardless of \(  z_0 \), since this is a continuous family of effective algebraic cycles in a projective variety.
Therefore, 
\begin{equation}
\label{eq:aux_cont}
| G_{r,k}(z_0) | \leq d\, \norm{ (\rho_{r}-\rho_{r+k}) g}_\infty  \, \norm{\alpha}_X \ \leq \ d\, \norm{\rho_r g}_\infty \, \norm{\alpha}_X  .
\end{equation} 
For \(  r \) fixed, this inequality holds for all \(  k \) and \(  z_0\in \Delta^n - F.  \) Since \(  \Delta^n - F \) is dense in \(  \bbp^n \) and \(  G_{r,k}(z) \) is continuous on \(  \Delta^n \), it follows that \eqref{eq:aux_cont} holds for every \(  z_0\in \Delta^n.  \)

Now, fix \(  z_0 \in \Delta^n \) and let \(  \gamma^\epsilon =f^\epsilon dz_I\wedge d\bar z_J\), as in \eqref{eq:bump}. Following \eqref{eq:lim2}
we get
\begin{align*}
 \int_{\bbp^n}\left(\Upsilon'_!(\phi_{r+k})- \Upsilon'_!(\phi_{r}) \right)\wedge \gamma^\epsilon & = \int_{\bbp^n} \left(\Upsilon'_!(\phi_{r+k}- \phi_{r}) \right)\wedge \gamma^\epsilon  \\
 & = 
\int_{\bbp^n} h(z)f^\epsilon(z) G_{r,k}(z_0) d \call_{2n}(z).
\end{align*}
Therefore, 
\begin{align}
 \left| \int_{\bbp^n}\left(\Upsilon'_!(\phi_{r+k})- \Upsilon'_!(\phi_{r}) \right)\wedge \gamma^\epsilon \right| & \leq   
\int_{\bbp^n} |h(z)| f^\epsilon(z) | G_{r,k}(z_0)|  d \call_{2n}(z) \\
& \leq
\left( d\, \norm{\rho_r g}_\infty \, \norm{\alpha}_X   \right) \int_{\bbp^n} |h(z)| f^\epsilon(z)  d \call_{2n}(z) .\notag
\end{align}
Taking the limit as \(  \epsilon \to 0 \) gives
\begin{align*}
& \left| \lim_{\epsilon\to 0}  
\int_{\bbp^n}  \left(\Upsilon'_!(\phi_{r+k})- \Upsilon'_!(\phi_{r}) \right)\wedge \gamma^\epsilon \right|  = \lim_{\epsilon\to 0}   
\left|\int_{\bbp^n}\left(\Upsilon'_!(\phi_{r+k})- \Upsilon'_!(\phi_{r}) \right)\wedge \gamma^\epsilon\right|   \\
& \leq 
\left( d\, \norm{\rho_r g}_\infty \, \norm{\alpha}_X   \right) \lim_{\epsilon\to 0} \, \int_{\bbp^n} |h(z)| f^\epsilon(z)  d \call_{2n}(z) \\
& =
\left( d\, \norm{\rho_r g}_\infty \, \norm{\alpha}_X   \right) |h(z_0)| \leq  
\left( d\, \norm{h}_\infty \, \norm{\alpha}_X   \right) \norm{\rho_r g}_\infty .
\end{align*}
The inequalities above imply that \(  \norm{\Upsilon'_!(\phi_{r+k}) - \Upsilon'_!(\phi_r)}_K \leq    \left( d\, \norm{h}_\infty \, \norm{\alpha}_X   \right) \norm{\rho_r g}_\infty \)\ 
on each compact \(  K \subset \Delta^n \).
It follows that the sequence  \(  \{ \overline\Upsilon_!(\phi_r) \}_{r\in \bbn}  \) converges uniformly on each compact subset of \(  \Delta^n \),  and that the limit coincides with \(  \overline\Upsilon_!(\phi) \) on \(  \Delta^n - F. \) The result follows. 
\end{proof}

We now prove Proposition \ref{prop:log_cur}.

\begin{proof}(of Proposition \ref{prop:log_cur})

Consider \(  \Upsilon \subset U\times \Delta^n  \) as in Lemma \ref{prop:infty}, and let \(  H\subset \bbp^n \) be a smooth hypersurface \(  H\neq H_\infty.  \) It follows that \(  \overline\Upsilon \) intersects \(  X\times H \) properly and we denote the algebraic cycle \( [\overline\Upsilon]\cap [ X\times H]   \) \ by \ \(  \overline\Upsilon_{|H} \). 
Now, write \(  \overline\Upsilon_{|H}= A + B + C, \) where \(  \supp{B} \subset D\times \bbp^n \), \(  \supp{C} \subset X\times \{ H \cap H_\infty \} \), and no component of \(  A \) is contained in \(  D\times \bbp^n\, \bigcup\,  X\times \{ H\cap H_\infty\}.  \)

By definition, given \(  \phi \in \df{*}{X\times \bbp^n} \) one has \(  \left(  \overline\Upsilon_{|H}  \right)_!(\phi) = A_!(\phi) + B_!(\phi) \), since no component of \(  C \) is dominant over \(  H \). Furthermore, if \(  \phi \) vanishes at \(  D\times \bbp^n \) then \(  \left(  \overline\Upsilon_{|H}  \right)_!(\phi) = A_!(\phi) \). On the other hand, intersection theory shows that the restriction of  an intersection of two algebraic cycles to an open set coincides with the intersection of their respective restrictions. Therefore, 
\( \overline\Upsilon_{|H} \cap ( U\times \Delta^n) = A\cap (U\times \Delta^n) \)
coincides with 
\begin{equation}
\label{eq:coinc}  
\left(  \overline{\Upsilon} \cap \{U\times \Delta^n\} \right) \cap \left(  [X\times H]\cap \{U\times \Delta^n\}  \right) =  [ \Upsilon] \cap [U\times \mathring{\! H}] =: \Upsilon_{|\, \mathring{\! H}},
\end{equation} 
where \( \mathring{\! H} = H \cap \Delta^n = H- H_\infty  \).

Let \(  \overline{\left( \Upsilon_{|\, \mathring{\! H}} \right)} \) the algebraic cycle in \(  X\times H \) obtained by taking the closure of each component of \(  \Upsilon_{|\, \mathring{\! H}} \) while keeping their multiplicities. The arguments above show that \(   \overline{\left( \Upsilon_{|\, \mathring{\! H}} \right)}  = A \), and hence
\begin{equation}
\label{eq:clos_integ}
 \overline{\left( \Upsilon_{|\, \mathring{\! H}} \right)}_!(\phi) \ = \ \left(  \overline\Upsilon_{|H}  \right)_!(\phi),
\end{equation}
for all \(  \phi \in \cdf{*}{X\times \bbp^n} \) satisfying \(  \phi_{|Y\times \bbp^n}=0 .\)


It follows from Theorem \ref{prop:sharp} that the homomorphisms 
\( \left( \overline{\Upsilon}_{|H}\right)^\# \) and \( \overline{\left(  \Upsilon_{|\, \mathring{\! H}}  \right)}^{\, \# }\)
are well defined and we conclude that 
Proposition \ref{prop:log_cur}.\ref{proper:1st}  follows from \eqref{eq:clos_integ}.

To prove Proposition \ref{prop:log_cur}.\ref{proper:2nd} we use the same arguments   in the proof of Proposition~\ref{cor:nor_gen}.\ref{it:bdry} along Lemma \ref{prop:infty}.
\end{proof}

We conclude this section observing that the transform operation does not depend on the compactification chosen, up to canonical isomorphism.

\begin{proposition}
\label{prop:log_curr}
Let \(  \Upsilon \subset U\times \Delta^m \) be as above and let  \(  U\hookrightarrow X' \hookleftarrow D' \)  be another compactification.  Assume there is a map of pairs \(  f \colon (X', D') \to (X, D) \) which is the identity on \(  U \), and 
let \(  \overline \Upsilon \) and \(  \overline \Upsilon' \) denote the closures  of \(  \Upsilon \) in \(  X\times \bbp^n \) and \(  X' \times \bbp^n \), respectively. Then  the following diagram commutes.
 \begin{equation}
 \label{eq:diag_transf}
 \xymatrix{
 \curmeas{k}{\bbp^n} \ar[rr]^-{\trans{\overline\Upsilon'}{}} \ar[rrd]_-{\trans{\overline\Upsilon}{}} & &  \cur{*}{X'}(\log D') \ar[d]_\cong^{f_\#} \\
  &   &  \cur{*}{X}(\log D) ,
  }
 \end{equation}
where \(  \trans{\overline \Upsilon}{} \) and \(  \trans{\overline \Upsilon'}{} \) are the current transforms associated to \(  \overline \Upsilon \) and \(  \overline \Upsilon' \).
\end{proposition}
\begin{proof}
This result follows directly from the description of the integration along the fibers in the proof of Proposition \ref{prop:infty} and from the behavior of the slicing operation under orientation-preserving diffeomorphisms  \cite[Cor. 3.6(8)]{Fed-some} .
\end{proof}
%


\section{Geometric Currents on \texorpdfstring{$\bbp^n$}.}
\label{sec:geo_cur}

This section  introduces  the currents on complex projective spaces that  play a key role in subsequent constructions. The final outcome is the triple   \(  ( \Theta_n,  \simplex{n}, W_n) \) that we call the \emph{fundamental triple of currents} in \(\bbp^n. \)

\subsection{The basic semi-algebraic currents}
Let us start with the topological and algebraic simplices, seen as semi-algebraic subsets of  \(  \bbp^n. \) Write \(   \mbz = (z_0, \ldots, z_n) \) and let  \(  [\mbz] = [ z_0 : \cdots : z_n] \) denote the corresponding homogeneous coordinates on \(  \bbp^n. \)

\begin{definition}  
\label{def:sacurr}
For \(  0\leq j \leq n \), consider the linear form \(   \ve_j(\mbz) := z_0 + \cdots + z_j  \).
\begin{enumerate}[a)]
\item Denote by \(  H_j \subset \bbp^n  \) the \(  j \)-th \emph{coordinate hyperplane} given by \(  z_j=0  \), and let \(  L_j\subset \bbp^n \) denote  the  hyperplane given by \(  \ve_j(\mbz) = 0 \).  
\item  Let \(  \iota_r \colon \bbp^{n-1} \hookrightarrow \bbp^n \)  denote the natural inclusion identifying \(  \bbp^{n-1} \) with \(  H_r. \)  Given \(  0\leq k \leq n \) define the operation 
\begin{align}
 \tau_k \, \colon \, \cur{m}{\bbp^{n-1}} & \longrightarrow \cur{m+	2}{\bbp^n} \\
   T & \longmapsto  \tau_k(T) := \sum_{r=0}^k (-1)^r \iota_{r\#}\left( T \right).  \notag
   \end{align}
\item When \(  j=n \) we denote \(  H_\infty := L_n \subset \bbp^n \), and call it the \emph{hyperplane at infinity}. The \emph{algebraic \( n \)-simplex}
 \(
\Delta^n := \bbp^n - H_\infty,
\)
is canonically identified with the complex affine space 
\(
 \Delta^n  = \{ (u_0, \ldots, u_n) \mid u_0+ \cdots + u_n = 1 \} \subset \bba^{n+1}.  
\)
%
\item Via this identification, the \emph{standard topological simplex} \(  \simplex{n} \subset \bbr^{n+1} \) sits inside \(  \Delta^n \) as the semi-algebraic set
\begin{equation}
\label{eq:top_simp}
\simplex{n} = \left\{ (x_0,\ldots, x_n) \left| \sum_{r=0}^n x_r =1, \text{ and } 0\leq x_r \leq 1 \text{ for all } r =0,\ldots, n \right. \right\}.
\end{equation}
For simplicity, we also denote by \(  \simplex{n} \)  the current \(  \llb \simplex{n} \rrb \) associated to the canonical orientation of the topological simplex.
\end{enumerate} 
\end{definition}
%
%
%
Given \(  1\leq j \leq n \), consider the semi-algebraic set \( S_j  \subset \bbp^n \) given by
\begin{equation}
\label{eq:Sj}
 S_j := \{ [z_0:\cdots:z_n] \ \mid \ z_j = t \ve_j(\mbz) , \text{ for some } t \in [0, 1] \},
\end{equation}
and let \(  \llb S_j \rrb \in  \lint{1}{\bbp^n}  \) denote the corresponding \emph{semi-algebraic chain} 
oriented so that
   \( d \llb S_j \rrb = \llb L_{j-1}\rrb  - \llb H_j\rrb .\)
%
For \(  0\leq j < n \leq n \) define
\( R_{n,j}  = S_{n}  \cap S_{n-1} \cap \cdots \cap   S_{j+1} \subset \bbp^n.\)
We show in the next proposition that this intersection 
is \emph{proper}, and hence, we can suitably orient \(  R_{n,j} \) to have
\begin{equation}
\label{eq:orientR}
\llb R_{n,j} \rrb \ = \ \llb S_{n} \rrb \cap\llb S_{n-1} \rrb \cap \cdots \cap \llb S_{j+1} \rrb.
\end{equation}
Note that \(  R_{n,n-1}  = S_n \), and for completeness define \(  R_{n,n}  = \bbp^n \) and \(  R_{n,j} = \emptyset,  \) when \(  n<j.  \)

In order to perform calculations, it is useful to parametrize  \(  R_{n,j} \) as follows. Let   \(  [\mbu :\lambda] \) be  homogeneous coordinates on \(  \bbp^{j} \)  with \(  \mbu = (u_0, \ldots, u_{j-1} )\). Write \(  \mbs = (s_{0}, \ldots, s_{n-j}) \in \Delta^{n-j} \),\  with \(  \sum_{r=0}^{n-j} s_r = 1 \), and define the algebraic map
\begin{align}
\label{eq:alternative}
\Phi_{n,j}\ \colon \ \bbp^{j} \times \Delta^{n-j}  & \longrightarrow \bbp^n \\
([ \mbu :\lambda ], \mbs ) & \longmapsto [ \mbu : s_0 \lambda - \ve(\mbu) : s_1 \lambda : \cdots : s_{n-j} \lambda ].  \notag
\end{align}

%
\begin{proposition}
\label{prop:param}
Using the notation above, the following holds:
 \begin{enumerate}[a)]
 \item  \label{it:fourth} The map \(  \Phi_{n,j} \) induces an isomorphism between affine spaces
 \[
\Phi_{n,j} \colon  ( \bbp^{j} - H_j) \times \Delta^{n-j} \xrightarrow{\ \cong\ } \bbp^n - L_n = \bbp^n - H_\infty ,
\]
where  \(  H_j  \subset \bbp^{j}  \) is the hyperplane given by \(  \lambda = 0. \)  
\item  \label{it:first} The image of  \( \ \bbp^{j} \times \simplex{n-j} \) under \(  \Phi_{n,j} \) is precisely \(  R_{n,j}  \). In particular, \(  \Phi_{n,j} \) induces an isomorphism between the semi-algebraic set \(  \ \bbc^{j} \times \simplex{n-j} \) and \( \  R_{n,j} - B, \) where
we are identifying \(  \bbc^j = \bbp^j - H_j \) and  \(  B = L_n \cap R_{n,j} = H_\infty \cap R_{n,j}  =  L_j\cap H_{j+1}\cap \cdots \cap H_n \).
\item \label{it:second} The collection \(  \{ S_j,  \ldots, S_n \} \) intersects properly in the sense of Definition \ref{def:inter}. In particular, the real codimension of \(  R_{n,j} \) is \(  n-j \).
\item \label{it:third} \emph{(Boundary formula)}
\begin{align*} 
d \llb R_{n,j} \rrb \ & =   (-1)^{n-j-1}  \llb R _{n,j+1} \rrb \cap \llb L_{j}\rrb   
 - (-1)^n \sum_{r=j+1}^n(-1)^r \iota_{r\#} \left(\llb R_{n-1,j} \rrb  \right) 
\end{align*}
 \item \label{it:fifth} One has an identity of currents:
 \begin{equation}
 \label{eq:phiR}\Phi_{n,j \# } \llb \bbp^{j}\times \simplex{n-j}\rrb  =  \llb R_{n,j} \rrb.
 \end{equation}
\end{enumerate}

\end{proposition}
\begin{proof}
First observe that the identity  \(  \ve_n \circ \Phi_{n,j}((\mbu, \lambda), \mbt) = \lambda \)  shows that  \(  \Phi_{n,j} \) sends \(  \bbc^j \times \Delta^{n-1}  = (\bbp^j- H_j) \times \Delta^{n-j} \) into \(  \Delta^n.  \) On the other hand, the map
\begin{align*}
 \Psi\colon \Delta^n & \longrightarrow   \bbc^j \times \Delta^{n-j} \\
 ( z_0, \ldots, z_n) & \longmapsto \left( [z_0 : \cdots : z_{j-1} : 1 ] , ( 1- z_{j+1} - \cdots - z_n, z_{j+1}, \ldots, z_n ) \right) 
 \end{align*}
gives an inverse to the restriction of \(  \Phi_{n,j} \) to \(\bbc^j \times \Delta^{n-j} \). 
This concludes the proof of assertion \ref{it:fourth}).

Given \(  0\leq j < n  \), denote \(  \mbt = (t_{j+1}, \ldots, t_n) \) and define polynomials 
\begin{equation}
\label{eq:p}
\tsp^n_{j,r}(\mbt) = 
\begin{cases}
(1-t_{n})(1-t_{n-1}) \cdots (1-t_{j+1}) &, \text{ if } r = j \\
(1-t_{n})(1-t_{n-1})\cdots (1- t_{r+1}) t_{r} &, \text{ for } j< r< n \\
t_n  &, \text{ if } r =n.
\end{cases}
\end{equation}
It is easy to see that  
\begin{equation}
\label{eq:sum_t}
 \sum_{r=j}^n \tsp^n_{j,r}(\mbt) = 1,
\end{equation}
and that the induced map   \(  [0,1]^{n-j} \to \simplex{n-j} \)  sending \(  \mbt \mapsto  (s_0(\mbt), \ldots, s_{n-j}(\mbt)) \), with \(  s_k(\mbt) = \tsp^n_{j,j+k}(\mbt), k=0, \ldots, n-j \),  is a parametrization of the  simplex . 

By definition, the set \(  R_{n,j}  \) is described by the following conditions:
\begin{align}
z_{j+1} & = t_{j+1} \ve_{j+1}(\mbz), \ \  \text{for some\ } t_{j+1} \in [0,1] \notag\\
\label{eq:param0}
{\vdots }\quad  & \quad \quad {\vdots} \\
z_n \ \ & = t_n \  \ve_n(\mbz), \ \ \quad \text{for some\ } \ t_n \ \in [0,1] \notag
\end{align}
Performing successive substitutions of the type\ 
\(
\ve_{n-1}(\mbz)\, =\, \ve_n (\mbz) - z_n \, =\,  \ve_n(\mbz) - t_n \ve_n(\mbz) = (1-t_n)\ve_n(\mbz)
\)
one concludes that \( [\mbz] = [z_0 \colon\ \cdots\  \colon z_n]  \in R_{n,j}\) if and only if there are \(  t_{j+1}, \ldots, t_n  \in [0,1] \) such that

\begin{alignat}{6}
& z_{j+1} && = \ve_{n}(\mbz) \ (1-t_n) \cdots (1-t_{j+2})t_{j+1} && = \ve_n(\mbz) \tsp^n_{j+1,j}(\mbt) && = \ve_n(\mbz)s_0(\mbt) && \notag \\
&{\vdots }\quad  && \hspace{1in}  {\vdots}  &&  \vdots\notag \quad\quad\phantom{m} &&   \\
&  z_{n-1} && = \ve_{n}(\mbz)\ (1- t_n)t_{n-1}  {} &&  = \ve_n(\mbz) \tsp^n_{n-1,j}(\mbt) && = \ve_n(\mbz)s_{n-j-1}(\mbt) &&    \label{eq:param1}  \\
& z_n  \ && = \ve_n(\mbz)\ t_n  &&  = \ve_n(\mbz) \tsp^n_{n,j}(\mbt) && = \ve_n(\mbz)s_{n-j}(\mbt).  \phantom{m} && \notag  
\end{alignat}

 Hence,  it follows   from \eqref{eq:param1} that
the image of  \(  \bbp^j \times \simplex{n-j} \) under  \(  \Phi_{n,j} \) is \(  R_{n,j} \), \  and that \(  \Phi_{n,j} \) induces an isomorphism 
 \begin{equation}
 \label{eq:iso_op}
  \left(  \bbp^{j}-H_j \right) \times \simplex{n-j} \ \cong \  R_{n,j}   - B,
\end{equation}
with \(  B = R_{n,j} \cap H_\infty.  \) This proves assertion \ref{it:first}).

The proof that the  family \(  \{ S_j, \ldots, S_n\} \) intersects properly  follows  from the equations \eqref{eq:param1}. Since 
this is a family of codimension \(  1 \) semi-algebraic sets, the isomorphism \eqref{eq:iso_op} implies that \(  \dim{R_{n,j}} =  n+j \). Assertion \ref{it:second}) follows.

By definition,
\begin{align*}
d & \left(  \llb S_{n} \rrb \cap\llb S_{n-1} \rrb \cap \cdots \cap \llb S_{j+1} \rrb \right)  =\sum_{r=j+1}^n (-1)^{n-r} 
\llb S_{n} \rrb  \cap \cdots \cap d{\llb S_r \rrb} \cap \cdots \cap \llb S_{j+1} \rrb \\
&=
\sum_{r=j+1}^n (-1)^{n-r} 
\llb S_{n } \rrb  \cap \cdots \cap \left( \llb L_{r-1}\rrb  - \llb H_r\rrb\right)  \cap \cdots \cap \llb S_{j+1} \rrb \\
& =  (-1)^{n-j-1} \llb S_{n} \rrb \cap \cdots \cap \llb S_{j+2} \rrb \cap \llb L_{j}\rrb
-
\sum_{r=j+1}^n (-1)^{n-r}
\llb S_{n} \rrb  \cap \cdots \cap  \llb H_r\rrb  \cap \cdots \cap \llb S_{j+1} \rrb 
\\
& =  (-1)^{n-j-1}  \llb R_{n,j+1} \rrb \cap \llb L_{j}\rrb 
-
\sum_{r=j+1}^n (-1)^{n-r}
\iota_{r\#}\left( 
\llb S_{n-1} \rrb  \cap \cdots \cap   \llb S_{j+1}\rrb \right) \\
& =  (-1)^{n-j-1}  \llb R_{n,j+1} \rrb \cap \llb L_{j}\rrb  
- (-1)^n
\sum_{r=j+1}^n (-1)^{r}
\iota_{r\#}
\llb R_{n-1,j} \rrb 
\end{align*}
where the third identity follows from the fact that \(  \llb S_k \rrb \cap \llb L_k \rrb = 0 ,  \) for all \(  k. \) This proves the boundary formula in assertion \ref{it:third}).

Let \(  \iota_r \colon \bbp^{m-1} \to \bbp^m \) denote the inclusion as the \(  r \)-th coordinate hyperplane and note that  
\(
\Phi_{n,j} \circ (1 \times \iota_r) = \iota_{j+r}\circ \Phi_{n-1,j},
\)
for \( 1 \leq r \leq j  \).
Furthermore,  when \(  r=0 \), it is easy to see  that 
\(
\left\{ \Phi_{n,j} \circ (1 \times \iota_0)\right\}_\# \llb \bbp^{j} \times \simplex{n-j-1} \rrb \ = \ 
 \llb R_{n,j+1} \rrb\cap \llb L_j \rrb  . 
\)
Therefore, 
\begin{align}
\label{eq:dPhi}
 d& \Phi_{{n,j}{\#}}
\llb\bbp^{j} \times \simplex{n-j} \rrb 
=   
(-1)^{n-j+1}\partial  \Phi_{{n,j}{\#}}\left( \llb \bbp^{j} \times \simplex{n-j} \rrb \right)   \notag \\
&= 
(-1)^{n-j+1} \Phi_{{n,j}\#}\left(  \llb \bbp^{j} \rrb \times \partial \llb \simplex{n-j} \rrb \right) \notag \\
&= 
 (-1)^{n-j+1}\Phi_{{n,j}\#}\left( \llb \bbp^{j} \rrb \times \sum_{r=0}^{n-j}(-1)^r \iota_{r\#}( \simplex{n-j-1}) \right)\notag  \\
&=
 (-1)^{n-j+1} \sum_{r=0}^{n-j} (-1)^r 
 \left\{
 \Phi_{{n,j}\#} \circ ( 1\times \iota_{r}) \right\}_\# 
 \left(
 \llb \bbp^{j+n -n} \times \simplex{n-j-1} \rrb
 \right) \\
 & =
 (-1)^{n-j+1} \left\{ \Phi^n_{{n,j}\#} \circ ( 1\times \iota_{0})\right\}_\# 
 \llb \bbp^{j+n -n} \times \simplex{n-j-1}\rrb\notag  \\
 & \quad  +(-1)^{n-j+1}  \
 \sum_{r=1}^{n-j} (-1)^r \iota_{{j+r}\#} \circ \Phi^{n-1}_{{n-1,j}\#} \llb \bbp^{j} \times\simplex{n-j-1} \rrb\notag  \\
 & =
(-1)^{n-j-1} \llb R_{n,j+1} \rrb \cap \llb L_j \rrb   - \ (-1)^n \sum_{k=j+1}^n (-1)^k \iota_{k\#}\circ \Phi^{n-1}_{{n-1,j}\#} \llb \bbp^{j} \times\simplex{n-j-1} \rrb \notag \\
&=
d \llb R_{n,j} \rrb. \notag
\end{align}
Using the identity above and assertion \ref{it:first}) we conclude the proof of the proposition. 
\end{proof}
\bigskip

\begin{corollary}
\label{cor:trivial}
The given orientation on \(  R_{n,0} \) identifies the current
\( \llb R_{n,0} \rrb  \) with \( \llb  \simplex{n}\rrb .
\)
Furthermore,  if  \(  \norm{R_{n,j}} \) is the measure associated to the integral current \(  \llb R_{n,j}\rrb \)
for \( 0\leq j \leq n   \), then every proper Zariski closed subset \(  Z\subset \bbp^n \)   has \( \norm{R_{n,j}}  \)-measure zero. 
\end{corollary}
\begin{proof}
The first assertion follows directly from the proof of the proposition. To prove the second assertion, first note that each \(  R_{n,j} \) is Zariski dense in \(  \bbp^n \), since it contains \(  \simplex{n}.  \) Therefore, \(  Z\cap R_{n,j} \) must be a semi-algebraic subset of \(  R_{n,j} \) of dimension strictly less than \(  n+j = \dim_\bbr(R_{n,j}) \), and hence \(  \norm{R_{n,j}}(Z) =0.  \)
\end{proof}

\subsection{The canonical \texorpdfstring{$\scL^1_\text{loc}$}-forms}

\begin{definition} Fix  integers \( 0\leq j\leq n\).
\label{def:ometheta}
\begin{enumerate}[a)]   
   \item 
For \(  j\leq n \) define meromorphic \(  j \)-forms   \(  \tet{j}{n}  \) in \(  \bbp^n \) as \(  \tet{0}{n}  =1 \) and for \(  j>0 \)
   \[
 \tet{j}{n}  := \sum_{r=0}^j (-1)^r \frac{dz_0}{z_0}\wedge \cdots \wedge \widehat{\frac{dz_r}{z_r}} \wedge \cdots \wedge \frac{dz_j}{z_j}.
 \]
This is a form with log-poles along the divisor \(  D_j  := H_0\cup \cdots \cup H_j. \) 
\item 
\label{it:omega}
Define forms \( \om{j}{n} \) on \( \bbp^n \) by setting
\(
\om{0}{n}   = 0\), and  \( \om{j}{n}  := (-1)^j\hbar_j\,  \tet{j-1}{n} \
\)
for \(  1 \leq j \leq n \),
where
\[    
\hbar_j[\mbz] =\begin{cases}
\vspace{0.1in}
 \log \left( 1 - \frac{\ve_j(\mbz)}{z_j}\right)  
&, \text{ if } [\mbz] \notin S_j \\
\vspace{0.1in}
0 &, \text{ if }  [\mbz] \in S_j,
\end{cases} 
\]
and \(  \ve_j(\mbz) = z_0 + \cdots + z_j, \) as in Definition \ref{def:sacurr}. 
\end{enumerate}
\end{definition}

\begin{remark}
\label{rem:L1-forms}
The function  \(  \hbar_j \) is holomorphic on \(  \bbp^n - S_j \), see \eqref{eq:Sj},  and  lies in \(  \lloc(\bbp^n) \). The form \(  \om{j}{n}  \) is also in \(  
\llocdf{j-1}{\bbp^n} \), 
thus yielding a current \(  [\om{j}{n} ] \in \curmeas{j-1}{\bbp^n} \) represented by integration. 
Similarly, the forms \(  \theta^n_j  \)  define  \(
[\theta^n_j ] \in  \curmeas{2n-j}{\bbp^n} \).
\end{remark}
%
%

To simplify the statement and proof of the next result, we denote 
\begin{equation}
\label{eq:beta}
  \beta_{k,j}^n:= \tau_k\left(\llb R_{n-1,j} \rrb\lefthalfcup  \omega^{n-1}_j   \right) := \sum_{r=0}^k (-1)^r   \iota_{r\#} \left(\llb R_{n-1,j} \rrb\lefthalfcup  \omega^{n-1}_j   \right) \ \in \ \cur{n}{\bbp^n}. 
  \end{equation}
\begin{proposition}
\label{prop:basic_cur-I}
Given \(  1\leq j\leq n   \) the following holds:
\begin{enumerate}[a.]
\item \label{dTheta} Boundary formula for \(  [\tet{j}{n}]  \): \begin{equation}
 \label{eq:dTheta}
d[\tet{j}{n}] \ =\   - \twopii \sum_{r=0}^{j}\ (-1)^r \imath_{r\#}  [\tet{j-1}{n-1}] \ = \ 
- \twopii\, \tau_j\left(  [\tet{j-1}{n-1}]\right) .
\end{equation}

\item \label{inter-theta}
The  form \(  \tet{j-1}{n} \) is \(  \norm{S_j} \)-summable. Hence,
   \( \llb S_n \rrb \lefthalfcup \tet{j-1}{n} \)
is represented by integration. 
\item  \label{inter-Romega} The  form \(  \om{j}{n}  \) is \(  \norm{R_{n,j}} \)-summable. Hence,   \(  \llb R_{n,j} \rrb \lefthalfcup \om{j}{n}  \) is represented by integration. 
\item Boundary formula for \(  [\om{j}{n}] \):
\label{d_varpi} 
\begin{align*}
\label{eq:domega}
d[\om{j}{n}]  & =  [\tet{j}{n}]  - (-1)^j \twopii\, 
\llb S_j \rrb \lefthalfcup \tet{j-1}{n}  + \twopii \sum_{r=0}^{j-1}(-1)^r \iota_{r\#}([\om{j-1}{n-1}])  \\
& =  [\tet{j}{n}]  - (-1)^j \twopii\, 
\llb S_j \rrb \lefthalfcup \tet{j-1}{n}  + \twopii \tau_{j-1}\left( [\om{j-1}{n-1}] \right) 
\end{align*}
\item Boundary formula for \( \llb R_{n,j}\rrb \lefthalfcup [\om{j}{n}] \): (See \eqref{eq:beta} for notation.)
\begin{align*}
d  \left( \llb R_{n,j} \rrb \lefthalfcup \omega^n_j \right) 
 & =
 (-1)^n \left\{  \beta_{j,j}^n   -  \beta_{n,j}^n  + (-1)^j \twopii  \beta_{j-1,j-1}^n  \right\} \\
 & +
(-1)^{n} \left\{  (-1)^{j} \left( \llb R_{n,j}\rrb \lefthalfcup \theta^n_j \right)   
 -  \twopii \,  \left( \llb R_{n,j-1} \rrb  \lefthalfcup \theta_{j-1}^n\right)    \right\} .
\end{align*}
\end{enumerate}
\end{proposition}
\begin{proof} The proof of statement \eqref{dTheta} is a standard residue calculation.
The proof of statements \eqref{inter-theta} and \eqref{inter-Romega} are given in Proposition~\ref{cor:inter}, 
Appendix \ref{app:further}. 

Since \(  R_{n,j}  \cap S_j =  R_{n,j-1}  \), then \(  S_j \) has \(  \norm{R_{n,j}} \)-measure zero. It follows from equations \eqref{eq:param0}  that \(  Z_{n,j} := H_0 \cup \cdots \cup H_{j-1} \cup S_j \cup L_n\)  also satisfies \(  \norm{R_{n,j}} (Z_{n,j}) = 0. \)  Thus, for any smooth form \(  \beta  \) on \(  \bbp^n \) one has
\[
\llb R_{n,j} \rrb (\omega_j^n \wedge \beta)  = \left\{   \llb R_{n,j} \rrb \cap ( \bbp^n - Z_{n,j} )\right\}(\omega_j^n \wedge \beta)   =\lim_{\delta \to 0} \left( \llb R_{n,j}\rrb \cap U^\delta_{n,j}\right)  (\omega_{j}^n \wedge \beta),
\]
where \(  U^\delta_{n,j} \subset  \bbp^n\) is the  set described below. 

For \(  0\leq r \leq n  \)  let \(  N^\delta_{r} \) denote the complement of a tubular \(  \delta \)-neighborhood of the coordinate hyperplane \(  H_r \), thus having smooth real-analytic boundary and satisfying  \(  \bigcap_{\delta > 0} \left( \bbp^n - N^\delta_{r} \right)= H_r.  \)
Finally, consider the map
\(  \sigma_j \colon \bbp^n - S_j \to \bbc \) given by
\(\sigma_j([\mbz]) = 1 - \frac{\ve_j(\mbz)}{z_j}\) and define
\(  W^\delta_j := \sigma^{-1}_j (O_\delta) \) where 
\(  O_\delta \subset \bbc \) is the complement of the usual tubular \(  \delta \)-neighborhood of \(  (-\infty, 0] \) in \(  \bbc \). 
Now, define
\[
U^\delta_{n,j} := N^\delta_0 \cap \cdots \cap N^\delta_{j-1}\cap W^\delta_j.
\]
It is easy to see that when \(  \delta  \) is small enough, the  closed sets in this intersection  intersect properly and have  semi-analytic boundaries. Let us write
\begin{equation}
\label{eq:inter2}
\llb R_{n,j}\rrb \cap U^\delta_{n,j} = \llb R_{n,j}\rrb \cap 
\llb N^\delta_0 \rrb \cap \cdots \cap \llb N^\delta_{j-1}\rrb \cap \llb W^\delta_j \rrb. 
\end{equation}
Then

%
%

\begin{align}
& \partial  \left( \llb R_{n,j}\rrb \cap U^\delta_{n,j} \right) \ = \ 
  \left(\partial \llb R_{n,j}\rrb \right) \cap \llb U^\delta_{n,j} \rrb \ + \ 
   \llb R_{n,j}\rrb  \cap\partial  \llb U^\delta_{n,j} \rrb  \notag \\
& =
  \left(\partial \llb R_{n,j}\rrb \right) \cap \llb U^\delta_{n,j} \rrb \ + \ 
   \llb R_{n,j}\rrb  \cap\partial  \left( \llb N^\delta_0 \rrb \cap \cdots \cap \llb N^\delta_{j-1}\rrb \cap \llb W^\delta_j \rrb   \right) \notag  \\
\label{it:A}
& = \left(
\partial  \llb R_{n,j}\rrb \right) \cap \llb U^\delta_{n,j} \rrb  \\
\label{it:B}
& \quad +
\llb R_{n,j}\rrb \cap \left(  \sum_{r=0}^{j-1}  \llb N^\delta_0 \rrb \cap \cdots \cap \partial \llb N^\delta_r \rrb  \cap \cdots \cap \llb N^\delta_{j-1}\rrb  \right) \cap \llb W^\delta_j \rrb  \\
\label{it:C}
& \quad +
  \llb R_{n,j}\rrb \cap \llb N^\delta_0 \rrb \cap \cdots \cap \llb N^\delta_{j-1}\rrb \cap \partial \llb W^\delta_j  \rrb 
\end{align}

Notice that \(  \omega^n_j \) is smooth on \(  U^\delta_{n,j} \) for all \(  \delta>0 \). Hence, given a smooth form \(  \phi \) on \(  \bbp^n \) one has
\begin{align}
& d  \left( \llb R_{n,j} \rrb \lefthalfcup \omega^n_j \right)( \phi)  = 
(-1)^n  \llb R_{n,j} \rrb ( \omega^n_j \wedge d\phi)  \notag    = (-1)^n \lim_{\delta \to 0} \left( \llb R_{n,j}\rrb \cap U^\delta_{n,j}\right)  (\omega_j^n \wedge d\phi) \notag \\
&  = (-1)^n \lim_{\delta \to 0} \left( \llb R_{n,j}\rrb \cap U^\delta_{n,j}\right)  \left( (-1)^{j-1} \left\{
d( \omega_j^n \wedge \phi)  -  d\omega^n_j \wedge \phi\right\}  \right) \notag \\
& =
(-1)^{n+j-1} 
 \lim_{\delta \to 0} \ \partial \left( \llb R_{n,j}\rrb \cap U^\delta_{n,j}\right)  
 \left(  \omega_j^n \wedge \phi\right)   
- (-1)^{n+j-1}  \lim_{\delta \to 0}   \left( \llb R_{n,j}\rrb \cap U^\delta_{n,j}\right)  (\theta^n_j \wedge \phi) \notag  \\
  \label{it:4parts}
  & =
(-1)^{n+j-1} 
 \lim_{\delta \to 0} \ \partial \left( \llb R_{n,j}\rrb \cap U^\delta_{n,j}\right)  
 \left(  \omega_j^n \wedge \phi\right)  
- (-1)^{n+j-1} \left( \llb R_{n,j}\rrb \lefthalfcup \theta^n_j \right) (\phi)  
\end{align}

We now use  the terms \eqref{it:A} -- \eqref{it:C} to write down the limit  \eqref{it:4parts}.

First,  apply Proposition \ref{prop:param}.\ref{it:third} to \eqref{it:A} and get
\begin{align}
& \left(    \partial  \llb R_{n,j}\rrb  \cap \llb U^\delta_{n,j}\rrb  \right) ( \omega^n_j \wedge \phi) \\
& =
\left(  \left\{   
 \llb R_{n,j+1} \rrb \cap \llb L_{j}\rrb   
 + (-1)^j \sum_{r=j+1}^n(-1)^r \iota_{r\#} \left(\llb R_{n-1,j} \rrb  \right)
 \right\} \cap \llb U^\delta_{n,j}\rrb  \right) ( \omega^n_j \wedge \phi)  \\
& =
\left(   
  (-1)^j \sum_{r=j+1}^n(-1)^r \iota_{r\#} \left(\llb R_{n-1,j} \rrb  \right)
  \cap \llb U^\delta_{n,j}\rrb  \right) ( \omega^n_j \wedge \phi) ,
\end{align}
where the last identity follows from the fact that the restriction of \(  \omega^n_j \) to \(  R_{n,j+1} \cap L_j \cap U^\delta_{n,j} \) is equal to zero.

We use  the notation \eqref{eq:beta} in what follows. 
A direct inspection shows that for \(  j+1\leq r\leq n \) 
one has \(  \iota_r^*\omega_j^n = \omega_j^{n-1}.  \) 
Therefore,
\begin{align}
& \lim_{\delta\to 0}\   \left(    \partial  \llb R_{n,j}\rrb   \cap \llb U^\delta_{n,j}\rrb  \right) ( \omega^n_j \wedge \phi) \\
& = 
\lim_{\delta\to 0}\left(     
  (-1)^j \sum_{r=j+1}^n(-1)^r \iota_{r\#} \left(\llb R_{n-1,j} \rrb  \right)
   \cap \llb U^\delta_{n,j}\rrb  \right) ( \omega^n_j \wedge \phi) \\
 & = 
  (-1)^j \sum_{r=j+1}^n(-1)^r \iota_{r\#} \left(\llb R_{n-1,j} \rrb  \right)    ( \omega^n_j \wedge \phi) \\
\label{eq:tautau}
& 
= 
 (-1)^j \sum_{r=j+1}^n(-1)^r \iota_{r\#} \left(\llb R_{n-1,j} \rrb\lefthalfcup  \omega^{n-1}_j   \right)    ( \phi)  
 = 
 (-1)^j \left\{ \beta^n_{n,j} - \beta^n_{j,j}
    \right\}  ( \phi).
\end{align}

Now, for \(  0\leq r\leq j-1 \) one has 

\begin{align}
& \lim_{\delta\to 0} \, \left( \llb R_{n,j}\rrb \cap    \llb N^\delta_0 \rrb \cap \cdots \cap \partial \llb N^\delta_r \rrb  \cap \cdots \cap \llb N^\delta_{j-1}\rrb   \cap \llb W^\delta_j \rrb \cap  \llb V^\delta_n \rrb \right)  (\omega^n_j \wedge \phi) \\
& = (-1)^{r+1} \twopii \, \iota_{r\#}\left( \llb R_{n-1, j-1} \rrb\lefthalfcup \omega^{n-1}_{j-1}  \right) (\phi) .
\end{align} 
Indeed, noting that \(  \iota_r^*\omega_j^n = -\omega_{j-1}^{n-1}  \),  denoting by  \(E_{r,\delta}\) the total space of the circle bundle \(\varrho^\delta_r \colon \partial  N^\delta_r \to H_r\)  and using integration along the fiber of \(\varrho_r^\delta\), we obtain
\begin{align*}
 & \lim_{\delta\to 0} \, \left( \llb R_{n,j}\rrb \cap    \llb N^\delta_0 \rrb \cap \cdots \cap \partial \llb N^\delta_r \rrb  \cap \cdots \cap \llb N^\delta_{j-1}\rrb   \cap \llb W^\delta_j \rrb   \right)  (\omega^n_j \wedge \phi) \\
 & =  \lim_{\delta\to 0} \, \left( \llb R_{n,j} \rrb  \cap \llb {E_{r,\delta}}|_ {N^\delta_0  \cap \cdots \cap \widehat{N^\delta_r}  \cap \cdots \cap  N^\delta_{j-1}  \cap  W^\delta_j  \cap H_r } \rrb \lefthalfcup \omega^\delta_j  \right)  ( \phi)\\
  & = 
 \lim_{\delta\to 0} 
   \llb R_{n, j}   \rrb  \cap \iota_{r\#} \left( \llb H_r\cap  N^\delta_0  \cap \cdots \cap \widehat{ N^\delta_r }  \cap \cdots \cap  N^\delta_{j-1}  \cap  W^\delta_j  \rrb \lefthalfcup{\varrho_r^\delta}_!\omega^n_j\right) (\phi)\\
   & = 
 \twopii 
   \iota_{r\#}  \left( \llb R_{n-1, j-1}\rrb  \lefthalfcup  (-1)^{r+1}\omega^{n-1}_{j-1} \right) (\phi).
 \end{align*} 
 %
 %
%
%
%
Therefore, the term in \eqref{it:B} gives
\begin{align*}
& \lim_{\delta\to 0} \, \left\{ \llb R_{n,j}\rrb \cap \left(  \sum_{r=0}^{j-1}  \llb N^\delta_0 \rrb \cap \cdots \cap \partial \llb N^\delta_r \rrb  \cap \cdots \cap \llb N^\delta_{j-1}\rrb  \right) \cap \llb W^\delta_j \rrb  \right\} ( \omega^n_j \wedge \phi)  \\
& =  - \twopii\, \sum_{r=0}^{j-1} (-1)^{r} \iota_{r\#}\left( \llb R_{n-1, j-1} \rrb\lefthalfcup \omega^{n-1}_{j-1}  \right) (\phi) \\
& = -\twopii \beta_{j-1,j-1}^n(\phi).
\end{align*}
%

Now a direct calculation gives:

\begin{align}
 & \lim_{\delta\to 0} \, \left\{     \llb R_{n,j}\rrb \cap \llb N^\delta_0 \rrb \cap \cdots \cap \llb N^\delta_{j-1}\rrb \cap \partial \llb W^\delta_j \rrb \right\} ( \omega^n_j \wedge \phi) \notag \\
 & =
 \twopii \,  \left(  \llb R_{n,j} \rrb \cap \llb S_j \rrb \right) ((-1)^j\theta_{j-1}^n\wedge \phi) =
  (-1)^j \twopii \,    \llb R_{n,j-1} \rrb  (\theta_{j-1}^n\wedge \phi) \notag  \\
 & = 
(-1)^j \twopii \,  \left( \llb R_{n,j-1} \rrb  \lefthalfcup \theta_{j-1}^n\right) (\phi).
 \end{align}

%

%
%



%

%
%

Putting all together in \eqref{it:4parts} one obtains

\begin{align}
d & \left( \llb R_{n,j} \rrb \lefthalfcup \omega^n_j \right)( \phi)  =
(-1)^{n+j-1} 
 \lim_{\delta \to 0} \ \partial \left( \llb R_{n,j}\rrb \cap U^\delta_{n,j}\right)  
 \left(  \omega_j^n \wedge \phi\right)
\notag  \\
& \quad  
- (-1)^{n+j-1} \left( \llb R_{n,j}\rrb \lefthalfcup \theta^n_j \right) (\phi)  \notag \\
& =
(-1)^{n+j-1} (-1)^j \left\{   \beta_{n,j}^n(\phi)  \ - \ \beta_{j,j}^n (\phi)  \right\} \notag \\
& +
(-1)^{n+j-1} \left\{  -\twopii \beta_{j-1,j-1}^n(\phi)  \right\} \\
& +
(-1)^{n+j-1} \left\{ 
 (-1)^{j} \twopii \,  \left( \llb R_{n,j-1} \rrb  \lefthalfcup \theta_{j-1}^n\right) (\phi)  \right\} \notag\\
 &  - 
 (-1)^{n+j-1} \left( \llb R_{n,j}\rrb \lefthalfcup \theta^n_j \right) (\phi) \notag \\
 & =
 (-1)^n \left\{  \beta_{j,j}^n (\phi) -  \beta_{n,j}^n(\phi) + (-1)^j \twopii  \beta_{j-1,j-1}^n(\phi) \right\} \notag\\
 & +
(-1)^{n} \left\{  (-1)^{j} \left( \llb R_{n,j}\rrb \lefthalfcup \theta^n_j \right) (\phi) 
 -  \twopii \,  \left( \llb R_{n,j-1} \rrb  \lefthalfcup \theta_{j-1}^n\right) (\phi)  \right\}  
 \notag
\end{align}

\end{proof}

\begin{corollary}
\label{cor:normal}
The currents \(\llb R_{n,j}\rrb \lefthalfcup [\om{j}{n}]  \) and \(  [\tet{j}{n}] \) are normal, for all \(  0\leq j \leq n.  \)
\end{corollary}

\subsection{The fundamental triple of currents on \texorpdfstring{$\mathbb{P}^n$}{}}
We  now have all the ingredients to build the desired triple.

\begin{definition}
\label{def:Wn} 
 Set
\(  W_0  = 0 \) and for \(n \geq 1  \) define
\begin{equation}
\label{eq:Wn}
W_n   := (-1)^{\binom{n}{2}} \sum_{j = 1}^{n} (-1)^{\binom{j}{2}}\twopii^{n-j}\  \llb R_{n,j}  \rrb\lefthalfcup \om{j}{n} . 
\end{equation}
Denote \(  \Theta_n := [\tet{n}{n}] \) and call
\begin{equation}
\label{eq:FTC}
( \Theta_n,  \simplex{n}, W_n) \ \in \ F^n \cur{n}{\bbp^n} \oplus \scI^{n}(\bbp^n) \oplus \cur{n-1}{\bbp^n} 
\end{equation}
the \emph{fundamental triple of currents} on \(  \bbp^n. \)
\end{definition}
\begin{corollary}
\label{cor:fund_relation}
The fundamental triple \(  ( \Theta_n,  \simplex{n}, W_n) \) satisfies the following identity
\begin{align}
\label{eq:fund_relation}
d W_n  &=  \Theta_n -  (-1)^{\binom{n+1}{2}} (2\pi\mbi)^n \simplex{n} + \twopii \sum_{r=0}^n (-1)^r \iota_{r\#}\left( W_{n-1} \right) \\
& =  \Theta_n - (-1)^{\binom{n}{2}} (-2\pi\mbi)^n \simplex{n} + \twopii \tau_n \left( W_{n-1} \right). \notag 
\end{align}
\end{corollary}
\begin{proof}
It follows from the proposition that
\begin{align*}
& dW_n = 
(-1)^{\binom{n}{2}}  \sum_{j=1}^n (-1)^{\binom{j}{2}} (2\pi \mbi)^{n-j} 
d \left( \llb R_{n,j} \rrb \lefthalfcup \omega_j \right)  \\
& =
(-1)^{\binom{n}{2}} \sum_{j=1}^n (-1)^{\binom{j}{2}} (2\pi \mbi)^{n-j} 
(-1)^n \left\{  \beta_{j,j}^n  -  \beta_{n,j}^n + (-1)^j (2\pi \mbi)  \beta_{j-1,j-1}^n \right\}   \\
 & \quad + 
  (-1)^{\binom{n}{2}} \sum_{j=1}^n (-1)^{\binom{j}{2}} (2\pi \mbi)^{n-j}
 (-1)^{n+j} \left\{ \left( \llb R_{n,j}\rrb \lefthalfcup \theta^n_j \right)  \right\}   \\
  &\quad - 
 (-1)^{\binom{n}{2}} \sum_{j=1}^n (-1)^{\binom{j}{2}} (2\pi \mbi)^{n-j}
(-1)^n  \twopii\left\{\,  \left( \llb R_{n,j-1} \rrb  \lefthalfcup \theta_{j-1}^n \right)   \right\}.
\end{align*}
Using the fact that \(  \beta^n_{n,n}= \beta^n_{0,0} = 0\), \(   R_{n,n} =  \bbp^n  \)  and \(  \llb R_{n-1,n} \rrb = 0\), along with a simple reindexing, the above sum gives
\begin{align*}
& dW_n =  
 - (-1)^{{\binom{n}{2}} + n  } \sum_{j=1}^n (-1)^{\binom{j}{2}} (2\pi \mbi)^{n-j}   
  \beta_{n,j}^n        + 
    \llb R_{n,n}\rrb \lefthalfcup \theta^n_n      - 
 (-1)^{\binom{n}{2}+n}  (2\pi \mbi)^{n} 
 \llb R_{n,0} \rrb   \\
& =  
(-1)^{\binom{n-1}{2}}  \sum_{j=1}^{n-1} (-1)^{\binom{j}{2}} \twopii^{n-j}   
 \tau_n\left(\llb R_{n-1,j} \rrb\lefthalfcup  \omega^{n-1}_j   \right)  \\
& \quad     + 
   \llb R_{n,n}\rrb \lefthalfcup \theta^n_n      - 
 (-1)^{\binom{n+1}{2}}  (2\pi \mbi)^{n} 
 \llb R_{n,0} \rrb    \\
 & =
    \Theta_n   -   (-1)^{\binom{n+1}{2}}  (2\pi \mbi)^n \simplex{n} \\
    & \quad  +   \twopii\, 
    \tau_n\left((-1)^{\binom{n-1}{2}}\sum_{j=1}^{n-1} (-1)^{\binom{j}{2}} \twopii^{n-1-j}   
\llb R_{n-1,j} \rrb\lefthalfcup  \omega^{n-1}_j   \right)      
     \\
 & =
    \Theta_n   \ - \ (-1)^{\binom{n+1}{2}}  (2\pi \mbi)^n \simplex{n} \ +  \   \twopii\, \tau_n\left( W_{n-1}\right).
\end{align*}
\end{proof}
%


\section{The regulator map}
\label{sec:Regulator_Map}

Let \(  U \) be a smooth quasiprojective variety and let \(  U\hookrightarrow X \hookleftarrow D \) be a projective compactification of \(  U \), with \(  D:= X-U \) a divisor with normal crossings.

\begin{definition}
\label{def:triple}
Let \(  \Upsilon \subset U \times \Delta^n \) be an irreducible subvariety of codimension \(  p \leq \dim{U} = m \) which is equidimensional and dominant over \(  \Delta^n \). Using  Propositions \ref{prop:log_cur} and \ref{cor:nor_gen}, along with  \eqref{eq:FTC} and Remark \ref{rem:inter_integ}, one can define a triple of currents
\( \triple{\Upsilon}\)  by
\[
\begin{aligned}
 {\Upsilon}_{\simplex{}}  :=   &  (-1)^{\binom{n}{2}}(-2\pi \mbi)^p \,  \left(\trans{\overline\Upsilon}{\simplex{n}}\right)\cap {U} \ &   \in \ \ &   {\scI^{2p-n}_\text{loc}(U; \bbz(p))} &   &    \\
   { {\Upsilon}_{\Theta} }  := &   (- 2\pi \mbi)^{p-n} \, \trans{\overline\Upsilon}{\Theta_n} \ & \in  \ \ &  { F^p \curlog{2p-n}{X}{D}    } &  \\
   {\Upsilon}_{W}  := &  (-2\pi \mbi)^{p-n}\, \left( \trans{\overline\Upsilon}{W_n}\right)\cap U & \in \ \ & \cur{2p-n-1}{U} .   &
\end{aligned}
\]
\end{definition}

Define group homomorphisms
\begin{align*}
\reg \colon \mathscr{Z}^p_\text{eq}(U;n) & \longrightarrow \Gamma\left( X; \bbz(p)_{\scD, (X,U)}^{2p-n} \right) \\
\Upsilon & \longmapsto \triple{\Upsilon}, 
\end{align*}
where \(  \bbz(p)_{\scD, (X,U)}^{2p-n}  \) is the complex giving Deligne cohomology, as in  Section \ref{subsubsec:DBcoh}.
\begin{theorem}
\label{thm:MAIN}
Let \(  U \) be a smooth quasiprojective variety  and let  \(  U \hookrightarrow X  \hookleftarrow Y=X-U\) be a projective compactification of \(  U \), with \(  Y \) a divisor with normal crossings. The assignment \(  \Upsilon \mapsto \reg{(\Upsilon)} \) defines a map of complexes 
\[
\reg \colon \mathscr{Z}_{\Delta, \text{eq}}^p(U;*) \longrightarrow \bbz(p)_\scD^{2p-*}(U) 
\]
which gives the regulator maps 
\[
\reg \colon CH^p(U;n) \longrightarrow H^{2p-n}_\scD(U;\bbz(p)),
\]
for all \(  0\leq n\leq 2p, \) whenever \(  \dim(U) \geq p.  \)
\end{theorem}
\begin{proof}
The theorem follows directly from Lemmas \ref{lem:FT-Trans} and \ref{lem:bdry} below. 
\end{proof}

\begin{lemma}
\label{lem:FT-Trans}
The currents \(  \Theta_n, \ \simplex{n} \) and \(  W_n \) in \(  \bbp^n \) vanish suitably along the hyperplane at infinity \(  H_\infty \), in the sense of Definition \ref{def:mild}. 
\end{lemma}
\begin{proof}
This is clear for \(\simplex{n}\) and  \(  \Theta_n\).  For  \(  W_n \)  it suffices to show that \(  \llb R_{n,j}  \rrb\lefthalfcup \om{j}{n}\)
vanishes suitably at \(  H_\infty \), for all \(  0\leq j \leq n. \)

We use the parametrization    \(\Phi=\Phi_{n,j}\ \colon \ \bbp^{j} \times \simplex{n-j} \to \bbp^n\)  of   \(  R_{n,j} \) given in \eqref{eq:alternative}. Consider   the neighborhood of \(  R_{n,j} \) given by
\( W^\tau = \Phi (N_j^\tau\times \simplex{n-j} )\), where \(  N_j^\tau \) denotes the \(\tau\)-tubular neighborhood of the hyperplane \(H_j\subset \bbp^j\) in the Fubini-Study metric.
Given \(\beta \in  \dfcs{n}{\bbp^n} \) we need to show that
\[
\lim_{\tau\to 0}S\cap \llb \partial W^\tau \rrb (\beta) = \lim_{\tau\to 0} \int_{R_{n,j}\cap  \partial W^\tau } \om{j}{n} \wedge \beta=0.
\]

Now we can assume, with no loss of generality,  that
\(\supp (\Phi^*( \om{j}{n} \wedge \beta)) \subset  U_0 \times\simplex{n-j} = (\bbp^j - H_0)\times\simplex{n-j} \). Note that 
\[
U_0\cap \partial W^\tau = \{[1:\mbx:\lambda] \in\bbp^j \mid \mbx\in\bbc^{j-1}, \,
\lambda \in\bbc \, \text{and} \, |\lambda| = \sqrt{1+|\mbx|^2}\delta(\tau) \} ,
\]
with \(\delta(\tau)\xrightarrow{\tau\to 0}0\).

Using  an appropriate parametrization of 
\(U_0\cap \partial W^\tau\cong \bbc^j\times S^1\times \simplex{n-j}\)
we can write \(S\cap \llb \partial W^\tau \rrb (\beta) \) as the integral
\[
 \int_{\bbc^{j-1}\times [0,2\pi]\times [0,1]^{n-j}} 
\log \left( \frac{1+\ve(\mbx)}{1+\ve(\mbx)+\delta(\tau)\sqrt{1+| \mbx |^2} e^{i\theta} t_1 } \right) 
\frac{\varphi(\mbx,\lambda,\mbt)}{x_1\dotsb x_{j-1}} \delta(\tau)dV,
\]
where \(\ve(\mbx)=x_1+\dotsb +x_{j-1}\),  \(dV\) denotes the volume form on \( \bbc^{j-1}\times [0,2\pi]\times [0,1]^{n-j}\) and  \(\varphi\) is a smooth function with compact support. The result now follows from the dominated convergence theorem.
\end{proof}

\begin{lemma}
\label{lem:bdry}
With \(  \Upsilon \) as above, denote \(  \partial_k \Upsilon := \Upsilon_{|\, \mathring{\! H}_k}, \) where \(  H_k \) is the \(  k \)-th coordinate hyperplane in \(  \bbp^n \), image of the inclusion \(  \iota_k \colon \bbp^{n-1} \hookrightarrow \bbp^n \) and\, \(  \mathring{H_k} =  H_k \cap \Delta^n \). Then
\begin{enumerate}[\textsc{p}\rm 1.]
   \item \(   (\partial \Upsilon)_{\simplex{} }:= \sum_{k=0}^n (-1)^k \left( \partial_k \Upsilon \right)_{\simplex{}} \ =\  d\left( \Upsilon_{\simplex{}}\right)   \)
   \item \( \left( \partial \Upsilon\right)_{\Theta} := \sum_{k=0}^n (-1)^k \left( \partial_k \Upsilon \right)_{\Theta}\ = \ 
   d\left( \Upsilon_\Theta\right)  \)
   \item \( \left( \partial \Upsilon \right)_W :=    \sum_{k=0}^n (-1)^k \left( \partial_k \Upsilon \right)_{W}  \ =\  \Upsilon_\Theta- \Upsilon_{\simplex{}} -   d\left( \Upsilon_W \right) \)
\end{enumerate} 
\end{lemma}
\begin{proof}
By definition
\begin{align}
d\Upsilon_{\simplex{}}& := (-1)^{\binom{n}{2}} (-2\pi\mbi)^p d \trans{\overline \Upsilon}{\simplex{n}} \ \ 
\overset{Cor. \ref{rem:prop_for_reg}}{=} \ \ 
(-1)^{\binom{n}{2}}(-2\pi\mbi)^p  \trans{\overline\Upsilon}{d \simplex{n}}\notag \\ 
& = (-1)^{\binom{n}{2}} (-2\pi\mbi)^p (-1)^{n+1} \trans{\overline\Upsilon}{\partial \simplex{n}}\notag \\ 
\label{eq:dGammaDelta} & =   (-1)^{\binom{n-1}{2}} (-2\pi\mbi)^p     \sum_{k=0}^n (-1)^k \trans{\overline\Upsilon}{\iota_{k\#}(\simplex{n-1})} \\
& =   (-1)^{\binom{n-1}{2}} (-2\pi\mbi)^p    \sum_{k=0}^n (-1)^k  \left(\overline\Upsilon_{|H_k}\right)^{\vee}_{\simplex{n-1}}, \notag
\end{align}
where the last identity follows from Proposition \ref{prop:prop_hyper} and Corollary \ref{cor:trivial}.
On the other hand,  Corollary~\ref{rem:prop_for_reg} shows that
\(  \left(\overline\Upsilon_{|H_k}\right)^{\vee}_{\simplex{n-1}} = \overline{\left(\Upsilon_{|\overset{\circ}{H}_k}\right)}^{\vee}_{\simplex{n-1}} \), and one concludes from \eqref{eq:dGammaDelta} that 
\begin{align*}  
d\Upsilon_{\simplex{}} & =  
   (-1)^{\binom{n-1}{2}} (-2\pi\mbi)^p      \sum_{k=0}^n (-1)^k  \overline{\left(\Upsilon_{|\overset{\circ}{H}_k}\right)}^{\vee}_{\simplex{n-1}} \\
& \overset{\text{def}}{=}
 (-1)^{\binom{n-1}{2}} (-2\pi\mbi)^p     \sum_{k=0}^n (-1)^k  \overline{\left(\partial_k \Upsilon\right)}^{\vee}_{\simplex{n-1}} 
\\  & 
=  (-1)^{\binom{n-1}{2}} (-2\pi\mbi)^p     \,\overline{\left( \partial \Upsilon \right)}^\vee_{\simplex{n-1}}
\ \overset{\text{def}}{=} \ 
\left( \partial \Upsilon \right)_{\simplex{}} .
\end{align*}
This proves identity \textsc{p1}.

Similarly, Corollaries  \ref{cor:trivial} and \ref{rem:prop_for_reg}.II,  Lemma \ref{lem:FT-Trans} and Proposition \ref{prop:prop_hyper}  give
\begin{align}
d\Upsilon_{\Theta}& := (-2\pi\mbi)^{p-n} d \,\trans{\overline \Upsilon}{\Theta_{n}} = (-2\pi\mbi)^{p-n}\, \trans{\overline\Upsilon}{d\Theta_n}\notag 
\\  & 
\overset{\eqref{eq:dTheta}}{=} (-2\pi\mbi)^{p-n} (-2\pi \mbi) \sum_{k=0}^n (-1)^k \trans{\overline\Upsilon}{\iota_{k\#}(\Theta_{n-1})} 
\\ & = (-2\pi\mbi)^{p-(n-1)} \sum_{k=0}^n (-1)^k  \left(\overline\Upsilon_{|H_k}\right)^{\vee}_{\Theta_{n-1}}.
\label{eq:dGammaTheta}
\end{align}
Then, once again Corollary  \ref{rem:prop_for_reg}.I shows that
\(  \left(\overline\Upsilon_{|H_k}\right)^{\vee}_{\Theta_{n-1}} = \overline{\left(\Upsilon_{|\overset{\circ}{H}_k}\right)}^{\vee}_{\Theta_{n-1}} \), and one concludes from \eqref{eq:dGammaTheta} that 
\begin{align*}  
d\Upsilon_{\Theta} & =  
(-2\pi\mbi)^{p-(n-1)}
\sum_{k=0}^n (-1)^k  \overline{\left(\Upsilon_{|\overset{\circ}{H}_k}\right)}^{\vee}_{\Theta_{n-1}} 
\overset{\text{def}}{=}
(-2\pi\mbi)^{p-(n-1)} \sum_{k=0}^n (-1)^k  \overline{\left(\partial_k \Upsilon\right)}^{\vee}_{\Theta_{n-1}} \\ 
& =  (-2\pi\mbi)^{p-(n-1)} \,\overline{\left( \partial \Upsilon \right)}^\vee_{\Theta_{n-1}}
\ \overset{\text{def}}{=} \ 
\left( \partial \Upsilon \right)_{\Theta} ,
\end{align*}
proving identity \textsc{p2}.

Finally, 
\begin{align}
& d\, \Upsilon_W  =
 (-2\pi\mbi)^{p-n} d \,\trans{\overline \Upsilon}{W_{n}} = (-2\pi\mbi)^{p-n}\, \trans{\overline\Upsilon}{dW_n}\notag \\ 
\label{eq:dGammaW} 
& \overset{\text{Cor.} \ref{cor:fund_relation}}{=} (-2\pi\mbi)^{p-n} \left\{ 
\trans{\overline\Upsilon}{\Theta_n} - (-1)^{\binom{n}{2}}(-2\pi\mbi)^n\trans{\overline\Upsilon}{\simplex{n}} + \twopii \sum_{k=0}^n (-1)^k \trans{\overline\Upsilon}{\iota_{k\#}W_{n-1}}  \right\} \\
& \overset{\text{def}}{=}  
\Upsilon_{\Theta} - \Upsilon_{\simplex{}} - (-2\pi \mbi)^{p-(n-1)} \sum_{k=0}^n (-1)^k \trans{\overline\Upsilon}{\iota_{k\#}W_{n-1}}\notag \\
& =
\Upsilon_{\Theta} - \Upsilon_{\simplex{}} - (-2\pi \mbi)^{p-(n-1)} \left(\partial \Upsilon \right)^\vee_{W_{n-1}}
\ \overset{\text{def}}{=} \
\Upsilon_{\Theta} - \Upsilon_{\simplex{}} - \left(\partial \Upsilon \right)_{W}, \notag 
\end{align}
where the second to last identity follows from
Corollaries  \ref{cor:trivial} and \ref{rem:prop_for_reg}.II, along with  Proposition \ref{prop:prop_hyper}
in the same fashion as the previous cases. 
\end{proof}
%


\section{Examples and beyond}
\label{sec:exmps}

\subsection{On the Mahler measure of  polynomials}

In Example \ref{exmp:mah_cor} we use a polynomial \(  \sfp(\mbt) \in \bbf[\mbt] = \bbf[ t_1, \ldots, t_n ] \) with  \(  \bbf \subset \bbc \), to construct a cycle \(  \Gamma_\sfp \in  \sbcxeq{n+1}{n+1}{U_\sfp} \), where \(  U_\sfp = \left( \gm^n - Z_\sfp\right) \times \gm \) and \(  Z_\sfp \subset \gm^n\) is the zero set of~\(  \sfp,\)  given by
\begin{equation*}
\Gamma_\sfp = 
\begin{cases}
z_{n+1}( \lambda + \sfp(\mbt ) ) & = \lambda \\
z_0 -  t_1z_1  & = 0 \\
z_0 + z_1 - t_2 z_2  & = 0 \\
\vdots & \vdots \\
z_0 + z_1 + \cdots + z_{n-1} - t_n z_n & = 0
\end{cases}.
\end{equation*}
It is easy to see that the same construction works when \(  \sfp \) is a Laurent polynomial. 
In order to calculate \(  \reg{(\Gamma_\sfp)} \) we discuss a few preliminaries. 

\subsubsection{General conditions and calculations}
\begin{definition}
\label{def:maps}
Consider \(  \sfp(\mbt) \in \bbf[\mbt]  \) with \(  \deg(\sfp) = d \),  and recall that we introduced 
\[
R_\sfp(\mbz) = (z_1 \cdots z_n)^d\sfp\left( \frac{\ve_0(\mbz)}{z_1}, \frac{\ve_1(\mbz)}{z_2}, \ldots, \frac{\ve_{n-1}(\mbz)}{z_n}  \right) \in \bbf[\mbz] 
\]
to define a divisor 
\(  Y_\sfp \subset \Delta^{n+1} \) given by  \(  (z_0 z_1 \cdots z_{n+1})(\ve_0(\mbz) \cdots \ve_n(\mbz)) R_\sfp(\mbz) =0 \) 
that satisfies \(  \Gamma_\sfp \cap (U_\sfp\times Y_\sfp) = \emptyset. \) 

\begin{enumerate}[a.]
   \item 
Define 
\(
\scG_\sfp := \left\{  (\mbt, \lambda) \in U_\sfp \mid (1+t_1)\cdots (1+t_n)(\lambda + \sfp(\mbt)) =  0 \right\} \subset U_\sfp.
\)
This is the union of the graph of \(  - \sfp \) with \(\bigcup_{r=1}^n\left\{ \left( \gm \times \cdots \times \{- 1 \} \times \cdots \times \gm \right) \times \gm\right\} \cap U_\sfp \). Note that  \( \left( \scG_\sfp \times \Delta^{n+1} \right) \cap \Gamma_\sfp = \emptyset \). 
   \item  Define
\begin{align}
\label{eq:E}
 E   \colon \Delta^{n+1} - Y_\sfp \ &  \longrightarrow   U_\sfp, \quad    \mbz \longmapsto \left( \frac{\ve_0(\mbz)}{z_1}, \ldots, \frac{\ve_{n-1}(\mbz)}{z_n}, \frac{\ve_{n}(\mbz)}{z_{n+1}}   \right)  \\
 \label{eq:ell}
 \ell_\sfp   \colon U_\sfp & \longrightarrow U_\sfp, \quad ( \mbt, \lambda) \longmapsto \left( \mbt, \frac{\sfp(\mbt)}{\lambda} \right) \\
\phi \colon U_\sfp - \scG_\sfp & \longrightarrow \Delta^{n+1}, \quad ( \mbt; \lambda) \longmapsto \left(z_0(\mbt;\lambda), \ldots, z_{n+1}(\mbt;\lambda) \right),
\end{align}
where
\begin{equation}
\label{eq:phi}
\phi:
\begin{cases}
z_0(\mbt;\lambda) & = \left(  \prod_{r=1}^n  \frac{t_r}{1+t_r}  \right)\frac{\sfp(\mbt)}{\sfp(\mbt) + \lambda}\\
z_j(\mbt;\lambda) & = \frac{1}{1+t_j}\left(  \prod_{r=j+1}^n  \frac{t_r}{1+t_r}  \right)\frac{\sfp(\mbt)}{\sfp(\mbt) +\lambda}, \quad \text{for} \ \ \ j=1, \ldots, n; \\
z_{n+1}(\mbt;\lambda) & = \frac{\lambda}{\sfp(\mbt) + \lambda}. 
\end{cases}
\end{equation}
\end{enumerate}
\end{definition}
\begin{lemma}
\label{lem:properties}
Using the notation above along with Definition \ref{def:ometheta}, the following holds. 

\begin{enumerate}[a.]
\item Let \(  \omega_j := \omega_j^{n+1}  \) be as in Definition \ref{def:ometheta}.\ref{it:omega} and denote
\(  \beta_j := -  \log(-  t_j)  \frac{dt_1}{t_1} \wedge \cdots \wedge \frac{dt_{j-1}}{t_{j-1}}  . \)
Then
\begin{align}
 \phi^* \omega_j  &=   \beta_j , \ \ \text{for} \ \ j=1,\ldots n, \\ 
\phi^* \omega_{n+1} &=  - \log\left\{ \frac{- \sfp(\mbt)}{\lambda} \right\}  \frac{dt_1}{t_1} \wedge \cdots \wedge \frac{dt_{n}}{t_{n}}, \ \quad \text{and} \\
\phi^*\theta_{n+1}  &  = (-1)^{n} \frac{dt_1}{t_1} \wedge \cdots \wedge \frac{dt_n}{t_n} \wedge \frac{d \lambda}{\lambda} 
\label{eq:pb_mah3}
\end{align}

   \item The  map  \(  \psi  \colon \Delta^{n+1}  - Y_\sfp \to U_\sfp \), 
introduced in   \eqref{eq:map_mah}, can be factored as
\[
\psi \ : \ (\Delta^{n+1} - Y_\sfp)\  \xrightarrow{\ E\ } U_\sfp \xrightarrow{ \ \ell_\sfp\ } \  U_\sfp,
\]
and its image is \(  U_\sfp - \scG_\sfp . \)
\item The map \( \phi \colon U_\sfp - \scG_\sfp \to \Delta^{n+1}\) is the inverse of \(  \psi. \)
\end{enumerate}
\end{lemma}
\begin{proof}
The proof follows from straightforward calculations. 
\end{proof}
Next we introduce some integral currents that appear in subsequent calculations.
\begin{definition}
\label{def:Tj}
Let \(  \bbt_n \) denote the compact torus \(  \bbt_n := (S^1)^n \subset (\cs)^n \). For \(  j= 0, \ldots, n \) define
\[  
T_j :=   (\cs)^{ j} \times (0, \infty)^{n+1 - j}  \ \subset \ (\cs)^{n+1} .
\] 
Then \(  \llb T_j \rrb \) defines an analytic chain in \(  (\cs)^{n+1} \), and we use the same symbol to denote its restriction to \(  U_\sfp. \)

\end{definition}

\begin{conditions}
\label{cond:polys}
From now on,  assume that the polynomial \(  \sfp \) satisfies the following:
\begin{enumerate}[P1.]
   \item \(  \sfp(\mbt) \) has no zeros on the torus \(  \bbt_{n}= (S^1)^{n}, \)  i.e. \(  (S^1)^n \times \gm \subset U_\sfp \);
   \item \(\sfp( (0, \infty)^n ) \subset (0,\infty) \); 
   \item If \(  \text{Gr}_{-\sfp} \subset U_\sfp \) denotes the graph of the restriction of \(\,  - \sfp \) to \(  \bbt_n, \) then \(  \text{Gr}_{-\sfp} \) intersects \(T_j  \) properly,   for each \(  j=1,\ldots, n \). 
\end{enumerate}
\end{conditions}

\begin{remark}
\label{rem:equi_cond}
Condition \(  \text{P}2. \) implies  that \( \ell_{\sfp\#} \llb (0, \infty)^{n+1}\rrb   =   - \llb (0, \infty)^{n+1}\rrb   \), see~\eqref{eq:ell}. 
\end{remark}

\begin{example}
\label{exmp:part_exmp}
The following polynomials satisfy the conditions above. 
\begin{enumerate}[a.]
   \item \label{it:k} Constant polynomials \(  \sfp(\mbt) = k \in \bbr\), with \(  k>0. \)
   \item \label{it:alpha} The Laurent polynomial 
\[
\sfp_\alpha(x, y) = \alpha +x + \frac{1}{x} + y + \frac{1}{y}
\]
with \(  \alpha \in \bbr \), whenever  \(  \alpha > 4.  \)
\end{enumerate}
\end{example}

Assume that one has a fixed NCD compactification \(  U_\sfp \hookrightarrow X_\sfp \hookleftarrow D \).

\begin{lemma}
\label{lem:reg1}
For all \(  j = 0,\ldots, n \) the restriction of \(  \overline{\Gamma}^\vee_{\!\sfp, \llb R_{n+1, j}\rrb} \) to \(  U_p \)  (see Defn. \ref{def:log_transf}) is the semi-algebraic chain \ 
\(
\Gamma^\vee_{\!\sfp, \llb R_{n+1,j} \rrb} :=    \ell_{\sfp \#} 
\llb T_j \rrb .
\)
\end{lemma}
\begin{proof}
A direct calculation shows that the map \(  E \)  \eqref{eq:E} induces an  orientation-preserving diffeomorphism between  
\(   R_{n+1, j} \cap (\Delta^{n+1} - Y_\sfp) \) and  \(  T_j  \cap (U_\sfp - \scG_\sfp ) .\)  By definition, the restriction of  \(  \overline{\Gamma}^\vee_{\!\sfp, \llb R_{n+1, j}\rrb} \) to \(  U_p - \scG_\sfp \)
is given by 
\begin{align*}
\psi_{\#}\left\{ \llb R_{n+1,j} \rrb \cap  (\Delta^{n+1} - Y_\sfp) \right\} & =    \ell_{\sfp \#} \left\{ E_\# \llb R_{n+1,j} \rrb \cap  (\Delta^{n+1} - Y_\sfp)\right\} \\
& =    \ell_{\sfp \#} 
\left\{ \llb T_j \rrb \cap (U_\sfp - \scG_\sfp ) \right\} .
\end{align*}
The result follows. 
\end{proof}
\begin{corollary}
\label{cor:reg2}
For \(  j=1, \ldots, n \), the restriction of \  \( \overline{\Gamma}^\vee_{\!\sfp, \llb R_{n+1,j} \rrb \lefthalfcup \omega_j} \) to \(  U_\sfp \) is given by 
\[
\Gamma^\vee_{\!\sfp, \llb R_{n+1,j} \rrb \lefthalfcup \omega_j} := 
\ell_{\sfp \#} \left( 
\llb T_j \rrb  \lefthalfcup \beta_j
\right).
\]
and
\[
\Gamma^\vee_{\!\sfp, \llb R_{n+1,n+1} \rrb\lefthalfcup \omega_{n+1} } = 
\Gamma^\vee_{\!\sfp, [ \omega_{n+1}] } = -  \left[ \log\left\{ \frac{-\sfp(\mbt)}{\lambda} \right\}  \frac{dt_1}{t_1} \wedge \cdots \wedge \frac{dt_{n}}{t_{n}}\right] .
\]
\end{corollary}

It follows from Lemma \ref{lem:properties} and Corollary \ref{cor:reg2}, together with Remark \ref{rem:equi_cond}, that  \(  \reg(\Gamma_\sfp)=( \Gamma_{\!\sfp,\, \simplex{}}, \Gamma_{\!\sfp, \Theta}, \Gamma_{\!\sfp, W})  \) -- see Definition \ref{def:triple} -- is completely determined by

\bigskip

%
%
%
%
%
%
%
%
%
%
%

\begin{align}
\Gamma_{\!\sfp, \simplex{}} & =  (-1)^{\binom{n}{2}} (2\pi \mbi)^{n+1} \llb (0, \infty) \rrb^{n+1}\label{eq:Gamma_simplex}
\\
\Gamma_{\!\sfp, \Theta} & = (-1)^{n}   \left[ \frac{dt_1}{t_1} \wedge \cdots \wedge \frac{dt_n}{t_n} \wedge \frac{d \lambda}{\lambda}    \right] 
\label{eq:GammaTheta}\\
\Gamma_{\!\sfp, W} & =   (-1)^{\binom{n+1}{2}} \sum_{j=1}^n (-1)^{\binom{j}{2}} \twopii^{n+1 - j } \ell_{\sfp \#} \left( \llb T_j \rrb \lefthalfcup \beta_j \right)   
\label{eq:GammaW}  \\
& \quad -  \left[ \log\left\{ \frac{-\sfp(\mbt)}{\lambda} \right\}  \frac{dt_1}{t_1} \wedge \cdots \wedge \frac{dt_{n}}{t_{n}}\right] 
 \notag
\end{align}

\subsubsection{Explicit periods}

Let \(  \sfp \in \bbr[\mbt] \) be an arbitrary (Laurent) polynomial satisfying Conditions~\ref{cond:polys}, and let \(  \bone 
\in  \bbr[\mbt] \) be the constant polynomial. By definition, \(  U_\sfp \subset U_\bone = \gm^{n+1} \) and we can restrict \(  \Gamma_\bone  \) to \(  U_\sfp \) and  obtain a cycle in
\[
\Gamma := \Gamma_\sfp \ - \ \Gamma_\bone \ \in \  \sbcxeq{n+1}{n+1}{U_\sfp}.
\]
It follows from \eqref{eq:Gamma_simplex}, \eqref{eq:GammaTheta} and \eqref{eq:GammaW} that
\[
  \reg(\Gamma )=( 0,\ 0, \Gamma_{\!\sfp, W} - \Gamma_{\bone, W}) 
\]
and, as a result,  we see that \(  \reg [\Gamma]  \)  
comes from  \(  H^{n}(U_\sfp; \bbc)/ H^{n}(U_\sfp, \bbz(n+1)) \) in the  exact sequence
{ \begin{multline}
\label{eq:ses_mah}
0\to  H^n(U_\sfp, \bbc)/H^n(U_\sfp, \bbz(n+1))  \to  \dcz{n+1}{n+1}{U_\sfp} \\ \to H^{n+1}(U_\sfp, \bbz(n+1))\oplus F^{n+1}H^{n+1}(U_\sfp, \bbc) \to \cdots .
\end{multline}
}
In particular,  \(  \reg [\Gamma]  \)    induces a homomorphism
\begin{equation}
\label{eq:gamma}
 \gamma \colon H_n(U_\sfp, \bbz) \to \bbc/\bbz(n+1).
 \end{equation}

Condition P1 on \(  \sfp \) shows  that \(  \llb \bbt_n\times \{ -1 \}\rrb \) represents a non-trivial class in \( H_n(U_\sfp,\bbz) \), and we use the following lemma to calculate \(  \gamma (\llb \bbt_n \times \{ -1 \} \rrb) \).
\begin{lemma}
\label{lem:mah1}
Let \(  \sfp(\mbt) \in \bbr[\mbt]  \) satisfy the three conditions above. Then
\begin{enumerate}[a.]
   \item For all \(  j=1, \ldots, n \) one has \(  \llb T_j  \rrb \ \cap \ \llb \bbt_n \times \{ -1 \} \rrb = 0 \).
   \item The intersection 
      \( \mathbb{X}_{n,j} :=   \ell_{p\#}\left( \llb T_j   \rrb \right) \cap \ \llb \bbt_n \times \{- 1 \} \rrb \) is well-defined and satisfies 
      \begin{equation}
      \label{eq:X}
       \mathbb{X}_{n,j} (\beta_j) = 0 ,\ \ \text{ for all } \ \    j=1, \ldots, n. 
       \end{equation}
      \item One has 
      \begin{enumerate}[i)]
   \item \(  \int_{\bbt_n\times \{ -1 \} }\beta_{n+1} = 0 \)
   \item \(  \int_{\bbt_n\times \{ -1 \} } \log\left\{ \frac{-\sfp(\mbt)}{\lambda} \right\}  \frac{dt_1}{t_1} \wedge \cdots \wedge \frac{dt_{n}}{t_{n}}   =  \twopii^n \mbm(\sfp), \) where \(  \mbm(\sfp) \) is the (logarithmic) Mahler measure of the polynomial \(  \sfp.  \)
\end{enumerate}
   \end{enumerate}
\end{lemma}
\begin{proof}
The first statement follows from the fact the the supports of the respective currents do not intersect.
Now, observe that 
\begin{equation*}
  \ell_{p}^{-1} (\bbt_n \times \{ -1 \} ) =   \{ (\mbt, \lambda) \in U_\sfp \mid \mbt \in \bbt_n \ \ \text{and} \ \  \lambda = - \sfp(\mbt)\} = \text{Gr}_{-\sfp}.  
\end{equation*}
The identity above together with Conditions \ref{cond:polys}.P3 show that the intersection 
\begin{equation}
\label{eq:exist}\
  \llb T_j \rrb \cap \llb \text{Gr}_{-\sfp} \rrb = \pm  \llb T_j \rrb \cap \ell_{\sfp\#}^{-1}\llb \bbt_n \times \{-1 \} \rrb
\end{equation}
exists. Applying \(  \ell_{p\#} \) to \eqref{eq:exist} shows that \(  \mathbb{X}_{n,j} \) exists. 

Now, let \(  \sigma_j \colon U_\sfp \to U_\sfp \) denote complex conjugation on the \(  j \)-th coordinate, \(  j=1, \ldots, n. \)
Then, \(  \sigma_j \) reverses the orientation of both \(  T_j   \) and \(  \bbt_n\times \{ -1 \}.  \) If follows that 
\(   \sigma_{j\#} \mathbb{X}_{n,j} = \mathbb{X}_{n,j}.  \) On the other hand, restricted to \(  \bbt_n \) one has \(  \sigma_j^* \beta_j = - \beta_j  \). It follows that \(  \mathbb{X}_{n,j}(\beta_j) = \left( \sigma_{j \#} \mathbb{X}_{n,j} \right)(\beta_j) = \mathbb{X}_{n,j}(\sigma_j^*\beta_j) = - \mathbb{X}_{n,j}(\beta_j). 
\) This concludes the proof of statement b. 

To prove statement \emph{c.i)} just note that the restriction of \(  {\beta_{n+1}} \) to \(  \bbt_n\times \{ - 1 \}  \) is zero. 

Finally, since \(  \sfp \) has real coefficients,  the integral of \( \arg\{ \sfp(\mbt)\}\frac{dt_1}{t_1} \wedge \cdots \wedge \frac{dt_n}{t_n} \) over \(  \bbt_n \) is zero. This proves statement \emph{c.ii)}. 

\end{proof}
\begin{corollary}
\label{cor:mah_final}
Let \(  \sfp(\mbt) \) be a Laurent polynomial with real coefficients, satisfying Conditions \ref{cond:polys}. Then the homomorphism
\(  \gamma  \colon H_n(U_\sfp, \bbz) \to \bbc/\bbz(n+1)  \)  \eqref{eq:gamma}
satisfies 
\[  \gamma ( \llb \bbt_n\times \{ - 1 \} \rrb) =   -  \twopii^{n} \mbm(\sfp) \in \bbc/\bbz(n+1).     \]
\end{corollary}
\begin{proof}
This follows directly from the definition of \(  \Gamma  \) and from Lemma \ref{lem:mah1}. 
\end{proof}
\begin{example}
\label{rem:exmp}
For the polynomial \( \displaystyle \sfp_\alpha(x, y) = \alpha + x + \frac{1}{x} + y + \frac{1}{y}
 \) in Example \ref{exmp:part_exmp}, it is shown in
 \cite{Vill-ModMah} that for any \(  \alpha \in \bbc \)  one can describe \(  \mbm(\sfp_\alpha) \) in terms of hypergeometric functions: 
 \[
\mbm(\sfp_\alpha) = \mathbf{Re}\left\{ \log \alpha -  \frac{2}{\alpha^2}\pFq{4}{3}{\frac{3}{2}, \frac{3}{2}, 1, 1}{2,2,2}{\frac{16}{\alpha^2}}   \right\}. 
\]
For \(  \alpha > 0 \), one also has
\[
\mbm(\sfp_\alpha) = \frac{\alpha}{4} \mathbf{Re}\left\{  \pFq{3}{2}{\frac{1}{2}, \frac{1}{2}, \frac{1}{2}}{1, \frac{3}{2}}{\frac{\alpha^2}{16} }\right\},
\]
see \cite{Kur-Och_MahCrys} and \cite{Rog-Zud_Ell-Mahler}. In particular, for \(  \alpha= 8 \), Corollary \ref{cor:mah_final} shows that  the homomorphism \(  \gamma \colon H_2(U_{\sfp_8}, \bbz) \to \bbc/\bbz(3) \) satisfies
\[
\gamma  ( \llb \bbt_2\times \{ -1 \} \rrb) = - \twopii^2 \mbm(\sfp_8) = 96\,  L(E_{24}, 2) \neq 0 \in \bbc/\bbz(3),
\]
where \(  L(E_{24},z)  \) is the \(  L \)-series of the rational elliptic curve \(  E_{24} \)  of conductor~\(  24. \)
\end{example}

\subsection{Beyond the equidimensional framework}
\label{subsec:beyond}

Equidimensional cycles provide a simple conceptual framework that yields an explicit construction of the regulator map at the level of complexes. However, the fundamental triple \(  (\Theta_*, \simplex{*}, W_*) \) can still be used in ``suitably transversal'' situations where equidimensionality is not satisfied. 
\begin{definition}
\label{def:transversal}
Given a smooth quasiprojective variety \(  U \), let  \( \calz_{\Delta,\pitchfork}^p(U;*) \subset  \calz_\Delta^p(U;*)\) denote the subcomplex
generated by those irreducible correspondences \(  \Upsilon \subset U\times \Delta^n \)  satisfying the following. Given a proper face \(  f \colon \Delta^r \hookrightarrow \Delta^n \) then
\begin{enumerate}[a.]
   \item \(  \Upsilon \) intersects  \(  U \times f(\Delta^r) \) properly.
   \item Given a NCD compactification \(  X \) of \(  U \), with divisor \(  D = X-U \),  then
   \(   \llb \overline{\Upsilon} \rrb \cap ( \llb X \rrb \times f_\# \Theta_r)  \) induces a current in \(  \cur{*}{X}\la \log {D} \ra \), and
      \(  \llb \overline{\Upsilon} \rrb \cap ( \llb U \rrb \times f_\# W_r) \)  induces a current in \(  \cur{*}{U} \) and
       \(   \llb \overline{\Upsilon} \rrb \cap ( \llb X \rrb \times f_{\#} \simplex{r}) \) induces a current in \(  \scI_\text{loc}^*(U ) .  \)
\end{enumerate}
\end{definition}

The arguments  in the equidimensional case can be used to define a map of complexes
 \(\reg \colon \calz_{\Delta,\pitchfork}^p(U;*) \to \bbz(p)^*_{\scD}(U) \), and this  suffices to provide another approach to an explicit regulator map, along the lines of \cite{KLMS-AJHCG} once the following is proven. 
 \begin{quote}
 \emph{\textsc{Claim:} \ The inclusion of complexes \( \calz_{\Delta,\pitchfork}^p(U;*) \hookrightarrow \calz_{\Delta}^p(U;*) \) 
is a quasi-isomorphism when \(  U \) is a smooth quasiprojective variey.}
\end{quote}
This will be addressed in a forthcoming note.


\appendix

\section{Background on currents}
\label{sec:back_curr}

\subsection{Forms and Currents}

\begin{notation}
\label{not:site}
Let \(  \scS\) denote one of the following sites: \emph{smooth manifolds} with the \(  \scC^\infty \) topology; \emph{real} or \emph{complex manifolds} with the analytic topology; or \emph{algebraic varieties} with the Zariski topology. If \(  \scF \) is a sheaf on \(  \scS \) and \(  M \in \mathsf{Obj}( \scS )\), we denote by \(  \scF_{M} \) the restriction of \(  \scF \) to the small site \(  M_{\, \scS} \) of \(  M.  \)
\end{notation}

\subsubsection{Forms on smooth manifolds}
\label{subsec:fsm}

Let \(   \scA^k \) be  the sheaf of \(  \bbc \)-valued smooth differential forms of degree \(  k \) on smooth manifolds,  and let  
\(  \df{k}{M} := \Gamma(M,  \scA^k) \) and \(  \dfcs{k}{M} := \Gamma_c(M,\scA^k)  \)  respectively denote the spaces of   \(  k \)-forms and \(  k \)-forms with compact support  on a manifold \(  M\).  

The \(  \scC^\infty \) topology on \(  \df{k}{M} \) is defined  by uniform convergence on compact subsets of the derivatives of all orders of the coefficients of the forms in local coordinates. One topologizes the compactly supported forms \(  \dfcs{k}{M} \) by saying that \(   \beta_i \xrightarrow[i\to \infty]{} \beta \in \dfcs{k}{M} \) if there is a compact set \(  K \subset M\) such that \(  \supp{(\beta_i)} \subset K \) for all \(  i\in \bbn \), and the sequence converges to \(  \beta \) in the \(  \scC^\infty \) topology on \(  \df{k}{M}. \) See \cite[\S 9]{dRham-DM} or \cite[4.1.1,4.4.6]{Fed-GMT}.

\begin{definition}
\label{def:cont}
Consider the  sheaves \(  \scC \) and \(\lloc\) of continous, respec. locally integrable,  complex-valued functions on manifolds.
\begin{enumerate}[a)]
   \item \label{it: L1loc} 
   Since \(  \lloc \) is a sheaf of \(  \scA^0 \)-modules, one can define  \(  \lloc\otimes_{\scA^0} \scA^k \) and let
   \(  \llocdf{k}{M} := \Gamma(M, \lloc\otimes_{\scA^0} \scA^k) \) denote the space of forms of degree \(  k \) on \(  M \) with coefficients in \(  \lloc, \) or the ``\(  \lloc \) \(  k \)-forms on \(  M \)''.
   \item  \label{it:cont}
   Similarly,  \(  \cdf{k}{M} := \Gamma(M, \mathscr{C}\otimes_{\scA^0} \scA^k) \) denotes the  continuous \(  k \)-forms on \(  M. \)\  \(  \cdf{k}{M} \)  is given the topology of uniform convergence on compact sets of the component functions of forms in any local coordinate system.
   \item \label{it:cont_c}The continuous forms with compact support, \(  \cdfcs{k}{M} := \Gamma_c(M, \scC\scA^k)  \), can be topologized by saying that a sequence \(  \{ \beta_i \}_{i\in \bbn} \) converges to \(  \beta \in \cdfcs{k}{M} \) if there is a compact set \(  K \subset M\) such that \(  \supp{(\beta_i)} \subset K \) for all \(  i\in \bbn \), and the sequence converges to \(  \beta \) in the \(  \scC^\infty \) topology on \(  \cdf{k}{M}. \)

\end{enumerate}
\end{definition}

The bundles \(  \Lambda^k TM \otimes \bbc \) and \(  \Lambda^k T^*M\otimes \bbc \) inherit  canonical Hermitian metrics induced by the Riemannian metric on \(  M \). In particular, given \(  \gamma \) either in \( \Lambda^k T_x M \) or in \(  \Lambda^k T^*_xM \) for some \(  x\in M \),  we denote by \(  |\gamma|_x \) the length of \(  \gamma \) in the induced metric.  

\begin{definition}(\cite[4.1.6]{Fed-GMT}, \cite[\S 2.1]{King-var})
\label{def:forms_norms}
 Let \(  \varphi \)  be a (possibly discontinuous) \(  k \)-form on \(  M \).
 \begin{enumerate}[a)]
   \item Given \(  x \in M \), define
   \begin{equation*}
\label{eq:norm_fn}
\norm{\varphi}(x) := 
\sup \left\{ | \varphi (\gamma) |  \, \mid \, \gamma \in \Lambda^kT_xM  \text{ is a decomposable } r{\text{-vector and}} \ |\gamma |_x  \leq 1  \right\}.
\end{equation*}
When \(  \varphi \)  is a continuous form, the assignment \(  x \mapsto \norm{\varphi}(x) \) defines a continuous function \(  \norm{\varphi} \colon M \to \bbr \). 
  \item Given any \( K\subseteq M \), define  the \emph{comass} of \(  \varphi \) on \(  K \) as
\begin{equation}
\label{eq:comass}
\nu_K(\varphi) := \sup \{ \norm{\varphi}(x) \mid x \in K \}. 
\end{equation}
The form  \(  \varphi \) is called  \emph{bounded} when  \(  \nu_M(\varphi) < \infty\).
   \item \label{it:bBaire}
   We say that \(  \varphi \) is a \emph{Baire form} if \(\varphi \) is the pointwise limit of a sequence \(  \{ \phi_r , r \in \bbn\} \) of continuous forms \(  \phi_r \in \cdf{k}{M}. \)  Denote by \(  \bbdf{k}{M} \) the space of \emph{bounded Baire \(  k \)-forms} on \(  M \).  \end{enumerate} 
\end{definition}

\subsubsection{Forms on complex manifolds}(See \cite[\S 1.1]{King-AJ})
\label{subset:fcm}

For  \(  p\geq 0 \), let \(  \Omega^p \) be the sheaf of holomorphic \(  p \)-forms on complex manifolds and denote by \(  \scO\) the sheaf of holomorphic functions. 

On complex manifolds,   one has a decomposition  \(  \scA^k = \oplus_{p+q=k} \scA^{p,q} \), where \(  \scA^{p,q} \) is the sheaf of smooth forms of type \(  (p,q) \). 
The \emph{exterior derivative} \(  d \colon \scA^k \to \scA^{k+1} \) is canonically written as \(  d = d' + d'' \) with \(  d' \colon \scA^{p,q} \to \scA^{p+1,q} \) and \(  d'' \colon \scA^{p,q} \to \scA^{p,q+1}.  \)

%
%

Defining
\(  F^p \scA^k := \oplus_{p\leq r \leq k} \scA^{r,k-r} \), the exterior derivative makes \(  (F^p\scA^*, d)  \) into a subcomplex of \(  (\scA^*,d) \) so that one gets the Hodge filtration  \( \scA^* = F^0 \scA^* \supset \cdots \supset F^p \scA^* \supset F^{p+1}\scA^* \supset \cdots  \) on forms. 

Consider  a smooth proper complex manifold  \(  X \), and
let \(  D= \bigcup_i D_i  \) be a \emph{simple normal crossing divisor} (DNC) on \(  X \).  This means that the irreducible components \(  D_i  \) are smooth, and each \(  x \in X \) has a coordinate neighborhood \(  U \) where \(  x \equiv  (0,\ldots, 0) \) and \(  D\cap U = \{ (z_1, \ldots, z_d) \mid z_1 \cdots z_k = 0 \} \), for some \(  0 \leq k \leq d.  \)

Denote \( \jmath\, \colon\, U:= X- D \hookrightarrow X  \) and let \(  \shflp{1}{X}{D} \subset \jmath_*\Omega^1_X\) be the (locally free) sheaf of meromorphic \(1 \)-forms on \(  X \) generated over \(  \scO_X \) by the forms \(  \frac{df}{f} \) where \(  f \) is a holomorphic function whose zero set is contained in \(  D. \)
Define \(  \shflp{p}{X}{D} := \shflp{1}{X}{D} \wedge \cdots \wedge \shflp{1}{X}{D} \). Since \(  D \) is a DNC,  \(  \shflp{*}{X}{D} \subset \jmath_* \Omega^*_U \) is a locally free subsheaf of graded \(  \scO_X \)-algebras generated locally in the neighborhoods described above by \(  \Omega^*_X \) and the forms \(  \dlog{i}, i=1,\ldots, k. \)
The Hodge filtration \(  F \) of \(   \shflp{*}{X}{D} \) is the decreasing filtration \(  F^p\shflp{*}{X}{D} := \bigoplus_{r \geq p} \shflp{p}{X}{D}.  \)

\begin{definition}
\label{def:null}
Let \(\xymatrix{ D\  \ar@{^{(}->}[r] & X & \ar@{_{(}->}[l]_-{\jmath} U } \) be as above.
\begin{enumerate}[a)]
   \item \label{it:null_hol}
    \(  \shfnull{p}{X}{D} \subset \Omega^p_X\):  the subsheaf consisting of the holomorphic \(  p \)-forms that vanish on \(  D.  \) 
   \item \label{it:null_sm}
   \(  \sdfnull{p}{q}{X}{D} \ := \ \shfnull{p}{X}{D}\otimes \scA^{0,q}_X \).
   \item \label{it:log_sm} 
   \(  \sdflog{p}{q}{X}{D} \ := \ \shflp{p}{X}{D}\otimes \scA^{0,q}_X \)
\end{enumerate}
\end{definition}

\subsubsection{Currents on smooth manifolds}(\cite[Ch. III]{dRham-DM}, \cite[\S 1.2]{King-var}, \cite[4.1.7]{Fed-GMT})
\label{subsubsec:cur}

Let \(  M  \) be a smooth \(  m \)-dimensional manifold. The sheaf \(  \scur{k}{M} \) of \emph{(deRham) currents of degree \(  k\) on \(  M \)} associates to each open subset \(  U \subset M \) the vector space \(  \cur{k}{U} \) consisting of the continuous linear functionals on  \(  \dfcs{m-k}{U} \). 
In this context, elements in  \(  \dfcs{m-k}{U} \) are called
 \emph{test forms}.
\begin{definition}
\label{def:oper}
Let \(  M \) be a smooth \(m \)-dimensional manifold, and consider \(  T \in \cur{k}{M} \), \(  S \in \cur{p,q}{X} \), \(  \omega \in \df{r}{M}, \) and test forms \(  \varphi \in \dfcs{*}{M} \).
\begin{enumerate}[a)]
   \item If \(  N \subset M \) is an oriented submanifold of codimension \(  r \) denote by \(  \llb N \rrb \in \cur{r}{M} \) the current \( \varphi \mapsto \llb N \rrb(\varphi) :=  \int_N \omega \), defined by integration along \(  N. \) 
    \item We define \(  [ \omega ]  \in \cur{r}{M} \)  by  \(  [\omega](\varphi) = \int_M\, \omega\wedge \varphi \). Hence, the assignment \(   \omega \mapsto [\omega] \) induces an inclusion \(  \df{k}{M} \hookrightarrow \cur{k}{M}.  \) More generally, a form \(  \beta \) in \( \llocdf{k}{M} \)  defines a current \(  [\beta] \in \cur{k}{M}  \) in a similar fashion. 
   \item Define \(  T\lefthalfcup \omega \in \cur{k+r}{M}  \) by \(  (T\lefthalfcup \omega)(\varphi) = T(\omega \wedge \varphi) \). Note that this is a generalization of the exterior product of forms, namely, \(  [\beta] \lefthalfcup \omega = [ \beta \wedge \omega ]. \)
The current \(  T\lefthalfcup \omega \) is often denoted by \(  T\wedge \omega \) in the literature. 
   \item The \emph{boundary} \(  bT \) of \(  T \) is the adjoint of the exterior derivative on forms: \[  bT(\varphi) = T(d\varphi) .\]
   \item The \emph{exterior derivative} \(  d \colon \cur{k}{M}\to \cur{k+1}{M} \) is defined as \[ dT := (-1)^{k+1} bT, \] so that the inclusion \(  \df{*}{M} \hookrightarrow \cur{*}{M} \) becomes a map of complexes. 
   \item Denote the restriction of \(  T \) to an open set \(  W\subset M \) by \(  T_{|W} \) and define the \emph{support} \(  \supp(T) \) of \(  T \) as the intersection of all closed sets \(  F \) such that \(  T_{|M-F} \ = \ 0.  \)
   \item If \(  f \colon M \to N \) is a smooth map such that \(  f_{|\,\supp{T}} \) is proper, one defines the \emph{push-forward} \(  f_{\#}T \in \cur{k+m-n}{N} \) by \(  f_\# T(\psi) := T(f^* \psi).  \) Note that \(  d\circ f_\# = f_\#\circ d.  \)
\end{enumerate}
\end{definition}

\subsubsection{Currents on complex manifolds}(\cite[\S 1.3]{King-AJ})
\label{subsubsec:ccm}

When \(  X \) is a complex manifold of dimension \(  d \),  the sheaf \(  \scur{p,q}{X} \) of \emph{currents of type \(  (p,q) \) } consists of those currents that vanish on all test forms of type \(  (r,s) \neq (d-p, d-q).  \)

\begin{definition}
\label{def:curr_poles}
Consider  \(\xymatrix{ D\  \ar@{^{(}->}[r] & X & \ar@{_{(}->}[l]_-{\jmath} U } \) as before. 
Define the \emph{currents with log poles} \(  \curpqlog{p}{q}{X}{D}  \) as the continuous linear functionals on 
   \[   \scA^{d-p,d-q}_c(X)\la \mathsf{null}\, D \ra := \Gamma_c(X,\sdfnull{d-p}{d-q}{X}{D} ),\] in the relative topology from \(  \dfcs{p,q}{X}.  \)
Define 
\[  \curlog{k}{X}{D} := \oplus_{p+q=k}\curpqlog{p}{q}{X}{D}   \] 
and
\[   F^p \curlog{k}{X}{D} := \bigoplus_{r\geq p} \curpqlog{r}{k-r}{X}{D}  .\]
    This gives a filtered complex \[  \curlog{*}{X}{D} = F^0\curlog{*}{X}{D} \supset \cdots \supset F^p
    \curlog{*}{X}{D} \supset \cdots .\]

\end{definition}
\subsubsection{Special currents} 
\label{subsubsec:spcur}

\begin{enumerate}[a)]
\item \label{enu:c_meas} We say that a current \(   T \in \cur{k}{M} \)  is \emph{representable by integration}, or  a current of  \emph{order \(  0 \)}, or a current \emph{with measure coefficients} if it extends to a continuous linear functional on \(  \cdfcs{m-k}{M}. \) Denote the space of all such currents by \(  \curmeas{k}{M}. \)
\item If \(  T \) is represented by integration, it follows from Riesz representation theorem that there is  a Radon measure \(   \norm{T} \) on \(  M \), and a \(  \norm{T} \)-measurable \(  (m-k) \)-vector field \(  \xi_T  \) such that \(  T \) is given by
\[
T(\varphi) = \int_M \la \varphi, \xi_T \ra d \norm{T},
\]
on test forms \(  \varphi. \) By the dominated convergence theorem, one can define \(  T(\beta)  \)  on bounded Baire forms \(  \beta \in \scB\scA_c^{k}(M).  \) See \cite[4.1.5]{Fed-GMT}.
\item Define
\(  \mbM(T) := \sup \{ T(\varphi) \mid \varphi \in \dfcs{k}{M} \text{ and } \nu_K(\varphi) \leq 1 \}.  \) The current \(  T \) is said to have \emph{finite mass} when \(  \mbM(T)< \infty.  \)
\item \label{it:nor} A current \(  T \) is called \emph{locally normal} if  both \(  T \) and \(  dT \) are represented by integration. It is called \emph{normal} if it is locally normal and \(  \supp(T) \) is compact. 
We denote by \(  \lnor{k}{M} \) and \(  \nor{k}{M} \) the spaces of locally normal and normal currents of degree \(  k \) on \(  M \), respectively.  See \cite[4.1.7]{Fed-GMT} or \cite{FF-NIC}.
%
\item The rectifiable currents are defined as the completion in the mass norm \(  \mbM \) of the group of Lipschitz push-forwards of finite polyhedral chains in some Euclidean space, and a current \(  T \) is locally rectifiable if for each \(  x\in M \) there is a rectifiable current \(  T_x \) such that \(  x\notin \supp(T-T_x).  \)
\item \label{it:int_cur} A current \(   T \) is called  \emph{locally integral} if both \(  T \) and \(  dT \) are locally rectifiable. It is called \emph{integral} if it is locally integral and \(  \supp(T) \) is compact. 
We denote by \(  \lint{k}{M} \) and \(  \intcur{k}{M} \) the spaces of locally integral and integral currents of degree \(  k \) on \(  M \), respectively.  See \cite[4.1.8, 4.1.24]{Fed-GMT}, \cite{FF-NIC} and \cite[2.1]{King-var}
\item Standard arguments show that \(  \lint{k}{M} \subset \lnor{k}{M}.  \)
\item \label{it:intG} If \(  G \) is a finitely generated abelian group, we denote \(  \scI^k_{\text{loc}}(M;G) := \lint{k}{M}\otimes G,  \)
and
\(  \scI^k(M;G) := \intcur{k}{M}\otimes G  \) for the groups of locally integral and integral \(G\) chains.
%
\end{enumerate}

\subsection{Slicing and intersection of locally normal currents}

\subsubsection{Slicing locally normal currents}
\label{subsecsec:slice}
Here we summarize material from \cite{Fed-some} and \cite[4.3]{Fed-GMT}. 
Let \(  M \) be a smooth \(  m \)-dimensional Riemannian manifold,  and let \(  N \) be an oriented \(  n \)-dimensional Riemannian manifold
with (unit) orientation form \(  \vol{n}.  \) Denote by \(  \scH_n \) the Hausdorff measure on \(  N \) induced by the metric. 
Given a locally Lipschitz map \(  f \colon M \to N \), a normal current \(  T \in \nor{k}{M} \) and a bounded Baire form \(  \phi \in \bbdf{i}{N} \) one can define a current 
\begin{equation}
\label{eq:init_sl}
  \la T, f, \phi \ra \in \cur{k+i}{M} ,
\end{equation}
which coincides with \(  T\lefthalfcup f^*\phi \) whenever \(  f \) is a smooth map. 
\begin{theorem}
\label{thm:Fed_slicing}
If \(  f \colon M \to N \) is a locally Lipschitzian map and  \(  T  \in \nor{k}{M}\) is a normal current of degree \(  k\leq m-n \), then for \(  \scH_n \)-almost all \(  y\in N \) there exists a unique current \(  \la T, f, y \ra \in \nor{k+n}{M} \) which can be defined as follows.
Denote by \(  B_\rho(y) \) the ball  of radius \(  \rho \) centered on \(  y\in N,  \) and let \(  \chi_{B_\rho(y)} \) denote its characteristic function. Then
\[
\la T, f, y \ra (\psi) = (-1)^{n(m-n-k) }\lim_{\rho \to 0}\ \frac{1}{\scH_n(B_\rho(y))}\,   \la T, f,   \chi_{B_\rho(y)} \vol{n} \ra (\psi).
\]
\end{theorem}
Properties of the slicing function:

\begin{enumerate}[P1.]
   \item \(  \supp \la T, f, y \ra  \subset \supp(T)\cap f^{-1}\{ y \} \)
   \item Whenever \(  k < m-n \) and \(  \la T, f, y \ra \) exists, so does   \(  \la dT, f, y \ra = d \la T, f, y \ra \).
   \item Whenever \(  \psi \in \df{q}{M} \), with \(  q \leq m-n-k \), and \(  \la T, f, y \ra \) exists, so does 
   \[ \la T\lefthalfcup \psi, f, y \ra = (-1)^{nq}  \la T, f, y \ra \lefthalfcup \psi .\]
   \item For every bounded Baire form \(  \phi \in \bbdf{k-n}{M} \) one has
   \[ \la T, f, \vol{n} \ra(\phi) = \int_{N} \la T, f, y \ra(\phi) d\scH_n(y). \]
   \item If  \(  \mbu \colon N \to \bbc \) is a  bounded Baire function and \(  \psi \in \df{k-n}{M} \) then
   \[ \la T, f,  \mbu\, \vol{n} \ra(\psi) = \int_{N} \la T, f, y \ra(\psi)\, \mbu(y)  d\scH_n(y)\]
   and
   \[\mbM(\la T, f,  \mbu\, \vol{n} \ra) =  \int_{N} \mbM( \la T, f, y \ra ) \, |\mbu(y)|  d\scH_n(y).
    \]
\end{enumerate}

Using slicing of currents, Federer introduces the notion of intersection of (locally) normal  currents.\footnote{This can be defined, more generally, for locally flat currents}  Given \(  S \in \lnor{k}{M} \), \(  T\in \lnor{r}{M} \), with \(  k+r \leq m \), one says that \emph{the intersection of \(  S\) and \(  T \) exists} provided there is a current \(  S\cap T \in \cur{k+r}{M} \) characterized by the condition:

Let \(  \gamma \colon M \to M\times M \) be the diagonal map. If \(  h\colon U \xrightarrow{\cong} U' \) is an orientation-preserving diffeomorphism from an open subset of \(  U \) to \(  U'\subset \bbr^m \), and \(  \delta \colon \bbr^m\times \bbr^m \to \bbr^m \) is the ``difference" map \(  \delta(\mba, \mbb) = \mba-\mbb \), then
\[
\left(\gamma_{|U}\right)_\#\left\{ (S\cap T)_{|U}\right\} = (-1)^{k(m-r)} \la (S\times T)_{|U\times U}, \delta\circ (h\times h), \mathbf{0} \ra
\]
In other words, when the intersection exists, it is determined locally by the slice of the product of the currents under the 
difference map, over the origin \(  \mathbf{0} \in \bbr^m. \)

\begin{remark}
\label{rem:inter_integ}
 When the currents \(  S, T \) are locally integral (repec. semi-algebraic chains, sub-analytic chains) and the intersection exists, then \(  S\cap T \) is integral (respec. semi-algebraic chains, sub-analytic chain). 
\end{remark}

\subsubsection{Slicing Analytic Chains}
Here we summarize material from \cite{Hardt-slice}. 
\begin{definition}
\label{def:an_chains}
Let \( X \) be an oriented real analytic manifold.
\begin{enumerate}[a.]
\item  A \( k \)-dimensional locally integral current \( T \) in \( X \) is called a \( k\)-\emph{dimensional analytic
chain} if \(X\) can be covered by open sets \( U \) for which there exist \(k \) and \( (k - 1) \)-dimensional real analytic subvarieties \( Z \) and \( W \) of \( U \) with \( U \cap \supp{T}\subset Z \) and \( U \cap \supp{(b T)} \subset  W\). 
\item It
follows from \cite[4.2.28]{Fed-GMT} that an analytic chain $T$ is a locally finite sum of chains corresponding to integration over certain \( k \) dimensional oriented analytic submanifolds of \( X \).
\item  If \( X \) is a real analytic manifold, we denote by \( \realan{k}{X} \subset \scI_{k}^\text{loc}(X) \) the group of \( k \)-dimensional real analytic chains on \( X \).
\end{enumerate}

\end{definition}
\begin{theorem}{\rm (\cite[Thm. 4.3]{Hardt-slice})}
Let \(  f \colon M \to N \) be an analytic map between oriented real analytic manifolds of dimensions \( m  \) and \(  n \), respectively.   Given an analytic chain \(  T \) in \(  M \) of dimension \(  k \), denote
\begin{equation}
\label{eq:locus}
Y =
\left\{ y \in N \left| 
	\begin{array}{ll}
		 \dim\{ f^{-1}(y) \cap \supp{(T)} \} & \leq k-n, \ \text{ and }  \\
		 \dim\{ f^{-1}(y) \cap \supp(bT) \} & \leq k -n -1 
	\end{array}
     \right. 
\right\}.
\end{equation}
Then the slicing function \(  y  \mapsto \la T, f, y \ra \), maps \(  Y \) into the \(  k-n \) dimensional analytic chains in \(  M \), and is continuous in the flat norm topology in \(  \scF^\text{loc}_{k-n}(M).  \)
\end{theorem}

It follows directly from this theorem that when \(  M \) is an oriented real analytic manifold and the currents \(  S, T \) are analytic chains, then \(  S\cap T \) is an analytic chain, whenever this intersection exists. 

There are simple support criteria that ensure  the existence of the intersection of analytic chains.

\begin{definition}(\cite[\S 5]{Hardt-slice})
\label{def:inter}
Given analytic chains \(  S, T  \) in \(  M \), with \(  \dim{S} = s, \dim{T} = t \), one says that \(  \{ S, T \} \) \emph{intersect properly} (or \emph{intersect suitably}) iff
\begin{enumerate}[i.]
   \item \(  s+t \geq m ,\)
   \item \( \dim\left[ \supp(S) \cap \supp(T)\right] \leq s+t -m, \)
   \item \( \dim\left[ \left\{ \supp( bS) \cap \supp(T) \right\} \cup \left\{ \supp(S) \cap \supp( bT) \right\} \right] \leq s+t -m -1
\)
\end{enumerate}
\end{definition}
\begin{theorem}
\label{thm:hardt}
If \(  \{ S, T\} \) intersect properly, then the intersection of \(  S \) and \(  T \) exists and \(  S\cap T \) is an analytic chain in \(  M \) of dimension \(  s+t-m,  \)  and \[  b (S\cap T) = (-1)^{m-t}bS \cap T + S \cap bT. \]
\end{theorem}

\subsection{Main quasi-isomorphisms}

See \cite[\S 3.1]{Del-TH_II},   \cite[Thm. 2.1.1]{King-AJ},   \cite[\S 2.2]{King-AJ}. 
\begin{facts}
\label{fac:qi-I}
Let \(  M \) be a smooth manifold. 
Then the sheaves \(  \sdf{k}{M} \), \( {'\scD}^k_M  \) and \(  \scI^k_{M,\text{loc}} \) are acyclic, for all \(  k\geq 0 \) and the following maps of complexes are quasi-isomorphisms:
   \[
\bbc_M \xrightarrow{\simeq} \sdf{*}{M} \xrightarrow{\simeq}  {'\scD}^*_M  \quad \text{and}\quad \
\bbz_M \xrightarrow{\simeq}  \scI^*_{M,\text{loc}} .
\]
\end{facts}
\begin{facts}
\label{fac:qi-II}
Let \(\xymatrix{ D\  \ar@{^{(}->}[r] & X & \ar@{_{(}->}[l]_-{\jmath} U } \) be an NCD compactification of the smooth complex variety \(  U \). Then one has a commuting diagram of quasi-isomorphisms of complexes on \(  X \), where the rightmost arrows are filtered quasi-isomorphisms. 
\begin{equation}
\label{eq:qi-comm}
\xymatrix{
   &  R\jmath_* \Omega_U^* = \jmath_* \Omega_U^*   \ar[d]^\simeq & \left( \shflp{*}{X}{D}, F^*\right) \ar[l]_{\simeq} \ar[d]_{\simeq}^{\text{filtered}} \\
   R\jmath_*\bbc_U \ar[ur]^{\simeq}   \ar[r]^{\simeq}  \ar[dr]^{\simeq}  & \jmath_*\sdf{*}{U} \ar[d]^\simeq & \left( \scA^{*}_{X}\la\log D\ra, F^*\right) \ar[l]_{\simeq} \ar[d]_{\simeq}^{\text{filtered}} \\ 
   &  \jmath_*{'\scD}^{*}_{U} & \left( {'\scD}^{*}_{X}\la \log D \ra, F^*\right) \ar[l]_{\simeq} 
 }
\end{equation}

\end{facts}

%
%
%
\medskip

\section{Further technical proofs}
\label{app:further}

\begin{lemma}
\label{lem:assoc_inter}
Let \(  S, T  \) be normal currents of dimensions \( s   \) and \(  t \), respectively, on an \(  m \)-dimensional oriented manifold \(  M \). If the intersection \(   S\cap T \) exists, and  \(  \omega \) is a form on \(  M \) which is both \(   \norm{T} \)-summable and \(  \norm{S\cap T} \)-summable, then
the intersection \(   S \cap (T \lefthalfcup \omega) \) exists and satisfies
\[
 S \cap (T \lefthalfcup \omega) \ = \ (S\cap T) \lefthalfcup \omega. 
\]
\end{lemma}
\begin{proof}
One may assume that \(  M = \bbr^m \). Denote \(  \delta \colon \bbr^m \to \bbr^m \times \bbr^m \) the diagonal map and let \(  f \colon \bbr^m \times \bbr^m \to \bbr^m \) denote the difference map \(   f(\mbx, \mby) = \mbx - \mby. \)

Then, by definition, \(  S\cap T \) is the unique normal current on \(  M \) that satisfies
\[
\delta_{\#} ( S\cap T) = (-1)^{(m-s)t} \la S\times T, \ f, \mathbf{0} \ra.
\]
Now, if \(  \pi_2 \colon \bbr^m \times \bbr^m \to \bbr^m \) is the projection onto the second factor, one has:
\(  \delta_{\#} \left(  (S\cap T) \lefthalfcup \omega \right)
= 
\delta_{\#} \left(  (S\cap T) \lefthalfcup \delta^* \pi_2^* \omega \right)  
= 
\delta_{\#}( S\cap T) \lefthalfcup \pi_2^*\omega. \)

Given a test form \(   \varphi \) on \(  M \) one then has 
\begin{align}
\label{eq:slice1}
\delta_{\#}( S\cap T) \lefthalfcup \pi_2^*\omega\ (\varphi) & =
\delta_{\#}( S\cap T)( \pi_2^*\omega \wedge \varphi) = 
(-1)^{(m-s)t} \la S\times T, \ f, \mathbf{0} \ra ( \pi_2^*\omega \wedge \varphi) \notag \\
& =
(-1)^{(m-s)t}\lim_{\epsilon\to 0} \frac{1}{V_\epsilon} (S\times T) ( f^* \Omega_\epsilon \wedge \pi_2^* \omega \wedge \varphi) \\
& = 
(-1)^{(m-s)t + |\omega| |\varphi|}\lim_{\epsilon\to 0} \frac{1}{V_\epsilon} (S\times T) ( f^* \Omega_\epsilon \wedge  \varphi  \wedge\pi_2^* \omega) \notag
\end{align}
where \(  \Omega_\epsilon \) is the restriction of the volume form on \(  \bbr^m \) the ball \(  B_\epsilon(\mathbf{0}) \)  of radius \(  \epsilon \) around \(  \mathbf{0} \),  \(  V_\epsilon \) is the volume of \(  B_\epsilon(\mathbf{0}),  \) and \(  |\varphi|, |\omega| \) are the degrees of the forms. 

On the other hand, one can approximate \(  f^*\Omega_\epsilon \wedge \varphi \) by a sequence of sums of forms of the type \(  \pi_1^* \alpha \wedge \pi_2^* \beta \), see \cite[4.1.3]{Fed-GMT}. We have
\begin{align*}
(S\times T) ( \pi_1^*\alpha \wedge \pi_2^* \beta  \wedge\pi_2^* \omega)
& =(-1)^{|\beta| |\omega| } 
(S\times T) ( \pi_1^*\alpha \wedge \pi_2^*(\omega\wedge \beta) ) 
 \\
&=(-1)^{|\beta| |\omega| } S(\alpha) T(\omega\wedge \beta) \\
& = (-1)^{|\beta| |\omega| } \left\{ S \times ( T \lefthalfcup \omega) \right\}( \pi_1^* \alpha \wedge \pi_2^*\beta).
\end{align*}

Now, note that 
\(  |\alpha| + | \beta| = m + |\varphi|   \) and we can assume \(  |\alpha| = s \), hence 
\(  |\beta| = m-s +  |\varphi | \). 
This gives
\begin{equation}
\label{eq:approx}
(S\times T) ( f^* \Omega_\epsilon \wedge  \varphi  \wedge\pi_2^* \omega) = 
(-1)^{(m-s+| \varphi |) | \omega |}\left\{ S \times ( T \lefthalfcup \omega) \right\}( f^*\Omega_\epsilon \wedge \varphi)
\end{equation}

Taking \(  \epsilon \to 0 \) in \eqref{eq:slice1} and using \eqref{eq:approx} one gets 
\begin{multline*}
\label{eq:lim_slice}
\delta_{\#}( S\cap T) \lefthalfcup \pi_2^*\omega\  (\varphi)  = 
(-1)^{(m-s)t + |\omega| |\varphi| + (m-s+| \varphi |) | \omega | } \la S \times (T \lefthalfcup \omega),\ f , \mathbf{0} \ra (\varphi) \\
 = 
(-1)^{(m-s)t + |\omega| |\varphi| + (m-s+| \varphi |) | \omega | + (m-s) (t - |\omega|)}
\delta_{\#} \left( S \cap (T\lefthalfcup \omega) \right)(\varphi)\\
= (-1)^{ 2(m-s)t + 2 |\omega| |\varphi|}\delta_{\#} \left( S \cap (T\lefthalfcup \omega) \right)(\varphi).
\end{multline*}
\end{proof}

\begin{lemma}
\label{lem:geom}
Given \(  0\leq j <n \leq \tsn \) the following holds.
\begin{enumerate}[i.]
   \item \label{geom:inter} For all \(  j+1 \leq k \leq n \) one has \(  \llb L_k \rrb \cap \llb R_{n,j} \rrb \ = \ 0 \).
   \item If \(  \llb R_{n-1,j}(\tsn)\rrb \) and \(  \llb R_{n-1,j}(\tsn-1)\rrb \) denote the current \(  \llb R_{n-1,j} \rrb \) in the projective spaces \(  \bbp^\tsn \) and \(  \bbp^{\tsn-1}, \) respectively, and \(  \iota_{n} \colon \bbp^{\tsn-1} \hookrightarrow \bbp^\tsn \) is the inclusion of \(  \bbp^{\tsn-1} \) as the coordinate hyperplane \(  H_n \), then
\[
\iota_{n \#} \llb R_{n-1,j}(\tsn-1) \rrb
=
\llb H_n \rrb \cap \llb R_{n-1,j}(\tsn) \rrb.
\]
\end{enumerate}
\end{lemma}
\begin{proof}
It follows from \eqref{eq:orientR} and definitions that  \(  L_k \cap R_{k,j} = L_k \cap S_{j+1} \cap \cdots \cap S_k  = L_k \cap H_{j+1} \cap \cdots \cap H_k = L_j \cap H_{j+1} \cap \cdots \cap H_k,  \) and the latter is an intersection of \(  k-j+1 \) linearly independent hyperplanes. It follows that
\begin{align*}  L_k & \cap R_{n,j}(\tsn)  = L_k \cap R_{k,j}(\tsn) \cap R_{n,k}(\tsn) = 
 L_j \cap H_{j+1} \cap \cdots \cap H_k  \cap R_{n,k}(\tsn) \\
& \cong R_{n-(k-j+1),k -(k-j+1)}(\tsn  -(k-j+1)) = R_{n-k+j-1,j-1}(\tsn-k+j-1). 
\end{align*}
Now, if \( \llb L_k \rrb \cap \llb R_{n,j} \rrb \neq 0 \)  then \(  \dim (L_k \cap R_{n,j}(\tsn))  \) must be \(  2\tsn+j-n-2 \). On the other hand  \(  \dim R_{n-k+j-1, \, j-1}(\tsn-k+j-1)= (2\tsn + j -n-2) - (k-j)  \), which is strictly less than what it should be. This proves the first assertion. 

The second assertion follows directly from the definitions.
\end{proof}

\begin{proposition}
\label{cor:inter}
Given \(  1 \leq   j \leq n \), the following holds:
\begin{enumerate}[i.]
\item The differential form \(  \theta_{j-1}  \) is \(   \norm{S_j} \)-summable and hence one can define the flat current \(  \llb S_j \rrb \cap \theta_{j-1} =  \llb S_j \rrb \lefthalfcup \theta_{j-1}  \in \scF^{j}(\bbp^n).\)
\item The differential form \(  \omega^n_{j}  \) is \(   \norm{R_{n,j}} \)-summable and hence one can define the flat current \(   \llb R_{n,j} \rrb \cap  \omega^n_j = \llb R_{n,j} \rrb \lefthalfcup \omega^n_j   \in \scF^{n-1}(\bbp^n).\)
\end{enumerate}
\end{proposition}
\begin{proof}
In order to prove the first statement we show that the restriction of \(  \theta_{j-1}  \) to \(  S_{j} \) lies in \(  \Gamma( S_{j}, \lloc\otimes_{\cala^0} \cala^{j-1}) \), using the parametrization of \(  S_j \),
\(
\label{eq:parametrizationS_j}
\Phi \colon \bbp^{n -1 } \times \simplex{1} \to \bbp^n\) given by \(\Phi([\mbu:\lambda:\mbw],\mbs) =[\mbu:s_0\lambda-\ve(\mbu):s_1\lambda:\mbw],
\)
where \(\ve(\mbu)=u_0+\dotsb+u_{j-2}\), \(\lambda\in\bbc\) and \(\mbw\in \bbc^{n-j}\) (we assume that 
\(j>1\)). 
In the case \(j=n\), \(\Phi\) is the parametrization \(\Phi_{n-n-1}\) of Proposition~\ref{prop:param}. 

Denote \(  \beta := s_0\lambda - \ve(\mbu) \). Then, as a form on \(  \bbp^{n-1}\times \simplex{1} \) one has
\(
\Phi^*\theta_{j-1} = \theta_{j-2} \wedge \frac{d\beta}{\beta}\ + \ (-1)^{j-1} \wedge \operatorname{dlog}(\mbu),
\)
where \( \operatorname{dlog}(\mbu) := \frac{du_0}{u_0} \wedge \cdots \wedge \frac{du_{j-2}}{u_{j-2}} \). Write
\( \displaystyle \frac{d\beta}{\beta} =  
\frac{s_0\lambda}{\beta}\cdot \frac{d\lambda}{\lambda}\, + \,  \frac{\lambda}{\beta} \cdot ds_0 \, - \, 
\frac{d\ve(\mbu)}{\beta}  \) 
and observe that
\(\theta_{j-2}(\mbu) \wedge \frac{d\ve(\mbu)}{\ve(\mbu)} \ = \ (-1)^{j-2}  \operatorname{dlog}(\mbu)\).
Therefore, 
\begin{align*}
 \Phi^*\theta_{j-1} & = \frac{s_0\lambda}{\beta}\cdot \theta_{j-2}\wedge \frac{d\lambda}{\lambda} \ + \ 
\frac{\lambda}{\beta} \cdot \theta_{j-2}\wedge ds_0 \\ 
&
\quad -
(-1)^{j-2} \frac{\ve(\mbu)}{\beta} \cdot \operatorname{dlog}(\mbu)
+(-1)^{j-1} \operatorname{dlog}(\mbu) \\
& = 
\frac{s_0\lambda}{\beta}\cdot \theta_{j-2}\wedge \frac{d\lambda}{\lambda} \ + \ \frac{\lambda}{\beta} \cdot \theta_{j-2}\wedge ds_0  + (-1)^{j-1} 
\left\{ \frac{\ve(\mbu)}{\beta} + 1    \right\} \operatorname{dlog}(\mbu) \\
&=
\frac{s_0\lambda}{\beta}\cdot \theta_{j-1}
\ + \ \frac{\lambda}{\beta} \cdot \theta_{j-2}\wedge ds_0 \\
& =  
\frac{s_0\lambda}{s_0\lambda - \ve(\mbu)}\cdot \theta_{j-1}
\ + \ \frac{\lambda}{s_0\lambda - \ve(\mbu)} \cdot \theta_{j-2}\wedge ds_0. 
\end{align*}
To show this form is in \(  \Gamma\left(\bbp^{n-1} \times \simplex{1}, \lloc \otimes_{\cala^0}\cala^{j-1}\right) \) one simply needs to restrict it to standard coordinate charts and observe that  the coefficients are in \(  \lloc \).
\smallskip

For the proof of the second  statement we consider the parametrization \(\Phi=\Phi_{n,j}\) of \(R_{n,j}\)
given  in \eqref{eq:alternative}. Locally \(\Phi^*\omega^n_j\) is a sum of terms of the form \(\log y \operatorname{dlog}(\mbx)\) with \((y,\mbx)\in\bbc\times\bbc^{j-1}\). One easily checks that 
\(\log y \operatorname{dlog}(\mbx)\in \llocdf{j-1}{\bbc^j}\) hence the result follows.

\end{proof}
%

\bibliographystyle{amsalpha}
\bibliography{References}

\end{document}